\documentclass[dvipsnames,reqno,10pt]{amsart}

\usepackage{amsmath, amssymb ,amsthm, amsfonts, amsgen, color}
\usepackage{bbm}
\usepackage{tikz}
\usepackage{subcaption}
\usetikzlibrary{intersections,calc,patterns}
\usepackage[colorlinks=true, linkcolor=blue, citecolor=blue]{hyperref}

\definecolor{mix}{rgb}{0.6, 0.745, 0.4}

\numberwithin{equation}{section}

\newcommand{\R}{{\mathbb{R}}}

\newcommand{\beq}{\begin{equation}}
\newcommand{\eeq}{\end{equation}}

\newcommand{\weakly}{\rightharpoonup}

\newcommand{\Ical}{{\mathcal{I}}}

\newcommand{\Scal}{{\mathcal{S}}}

\newcommand{\dd}{{\,\rm d}}

\newcommand{\Acal}{{\mathcal{A}}}
\newcommand{\Ystiff}{Y_{\rm stiff}}
\newcommand{\Zstiff}{Z_{\rm stiff}}
\newcommand{\Ysoft}{Y_{\rm soft}}
\newcommand{\Ysoftr}{Y_{\rm soft}^r}
\newcommand{\eYstiff}{\eps Y_{\rm stiff}}
\newcommand{\jYstiff}{\eps_j Y_{\rm stiff}}
\newcommand{\eYsoft}{\eps Y_{\rm soft}}

\newcommand{\eps}{ \varepsilon}
\newcommand{\Z}{{\mathbb{Z}}}
\newcommand{\N}{{\mathbb{N}}}
\newcommand{\ffi}{\varphi}
\newcommand{\one}{\mathbbm{1}}
\newcommand{\ui}[1]{^{(#1)}}
\newcommand{\qc}{{\rm qc}}

\newcommand{\Wstiff}{W_{\rm stiff}}
\newcommand{\Wrig}{W_{\rm rig}}
\newcommand{\Wsoft}{W_{\rm soft}}
\newcommand{\Whom}{W_{\rm hom}}
\newcommand{\Wcell}{W_{\rm cell}}
\newcommand{\eWstiff}{W_{\rm stiff, \eps}}
\DeclareMathOperator{\Id}{Id}
\DeclareMathOperator{\Tr}{Tr}

\newcommand{\vertiii}[1]{{\left\vert\kern-0.25ex\left\vert\kern-0.25ex\left\vert #1 
\right\vert\kern-0.25ex\right\vert\kern-0.25ex\right\vert}}

\newcommand*\circled[1]{\tikz[baseline=(char.base)]{
            \node[shape=circle,draw,inner sep=2pt] (char) {#1};}}

\def\x{{\times}}

\makeatletter
\def\rightharpoonupfill@{\arrowfill@\relbar\relbar\rightharpoonup}

\newcommand{\xrightharpoonup}[2][]{\ext@arrow
0359\rightharpoonupfill@{#1}{#2}} \makeatother

\newtheoremstyle{thmlemcorr}{10pt}{10pt}{\itshape}{}{\bfseries}{.}{10pt}{{\thmname{#1}\thmnumber{ #2}\thmnote{ (#3)}}}
\newtheoremstyle{thmlemcorr*}{10pt}{10pt}{\itshape}{}{\bfseries}{.}\newline{{\thmname{#1}\thmnumber{ #2}\thmnote{ (#3)}}}
\newtheoremstyle{defi}{10pt}{10pt}{\itshape}{}{\bfseries}{.}{10pt}{{\thmname{#1}\thmnumber{ #2}\thmnote{ (#3)}}}
\newtheoremstyle{remexample}{10pt}{10pt}{}{}{\bfseries}{.}{10pt}{{\thmname{#1}\thmnumber{ #2}\thmnote{ (#3)}}}
\newtheoremstyle{ass}{10pt}{10pt}{}{}{\bfseries}{.}{10pt}{{\thmname{#1}\thmnumber{ A#2}\thmnote{ (#3)}}}

\theoremstyle{thmlemcorr}
\newtheorem{theorem}{Theorem}
\numberwithin{theorem}{section}
\newtheorem{lemma}[theorem]{Lemma}
\newtheorem{corollary}[theorem]{Corollary}
\newtheorem{proposition}[theorem]{Proposition}

\theoremstyle{thmlemcorr*}
\newtheorem{theorem*}{Theorem}
\newtheorem{lemma*}[theorem]{Lemma}
\newtheorem{corollary*}[theorem]{Corollary}
\newtheorem{proposition*}[theorem]{Proposition}
\newtheorem{problem*}[theorem]{Problem}
\newtheorem{conjecture*}[theorem]{Conjecture}

\theoremstyle{defi}

\theoremstyle{remexample}

\theoremstyle{remexample}
\newtheorem{remarkx}[theorem]{Remark}
\newenvironment{remark}
  {\pushQED{\qed}\remarkx}
  {\popQED\endremarkx}

\newcommand\restrict[1]{\raisebox{-.5ex}{$|$}_{#1}}
\renewcommand{\epsilon}{\varepsilon}
\newcommand{\norm}[1]{\|#1\|}
\DeclareMathOperator{\SO}{SO}

\DeclareMathOperator{\dist}{dist}

\def\Xint#1{\mathchoice
{\XXint\displaystyle\textstyle{#1}}%
{\XXint\textstyle\scriptstyle{#1}}%
{\XXint\scriptstyle\scriptscriptstyle{#1}}%
{\XXint\scriptscriptstyle\scriptscriptstyle{#1}}%
\!\int}
\def\XXint#1#2#3{{\setbox 0=\hbox{$#1{#2#3}{\int}$}
\vcenter{\hbox{$#2#3$}}\kern-.5\wd0}}
\def\dashint{\Xint-}
\newcommand{\qand}{\quad\text{and}\quad}

\title[A variational perspective on auxetic metamaterials]{A variational perspective on auxetic metamaterials of checkerboard-type}

\author{Wolf-Patrick D\"ull}
\address{Institut f\"ur Analysis, Dynamik und Modellierung, Universit\"at Stuttgart, Pfaffenwaldring 57, 70569 Stuttgart, Germany}
\email{duell@mathematik.uni-stuttgart.de}

\author{Dominik Engl}
\address{Mathematisch-Geographische Fakult\"at, Katholische Universit\"at Eichst\"att-Ingolstadt, Ostenstra{\ss}e 28, 85071 Eichst\"att, Germany}
\email{dominik.engl@ku.de}

\author{Carolin Kreisbeck}
\address{Mathematisch-Geographische Fakult\"at, Katholische Universit\"at Eichst\"att-Ingolstadt, Ostenstra{\ss}e 28, 85071 Eichst\"att, Germany}
\email{carolin.kreisbeck@ku.de}

\begin{document}

\maketitle

\begin{abstract}

The main result of this work is a homogenization theorem via variational convergence for elastic materials with stiff checkerboard-type heterogeneities under the assumptions of physical growth and non-self-interpenetration. While the obtained energy estimates are rather standard, determining the effective deformation behavior, or in other words, characterizing the weak Sobolev limits of deformation maps whose gradients are locally close to rotations on the stiff components, is the challenging part. To this end, we establish an asymptotic rigidity result, showing that, under suitable scaling assumptions, the attainable macroscopic deformations are affine conformal contractions. This identifies the composite as a mechanical metamaterial with a negative Poisson's ratio. Our proof strategy is to tackle first an idealized model with full rigidity on the stiff tiles to acquire insight into the mechanics of the model and then transfer the findings and methodology to the model with diverging elastic constants. The latter requires, in particular, a new quantitative geometric rigidity estimate for non-connected squares touching each other at their vertices and a tailored Poincar\'e type inequality for checkerboard structures.

\medskip


\medskip
 

\medskip
 
\noindent\textsc{Date:} \today.
\end{abstract}

\section{Introduction}

When speaking of metamaterials, one usually refers to engineered and artificially fabricated materials tailored to show specific desirable properties that are rare to find naturally. Among the many different types of metamaterials are electrical, magnetic, acoustic, and mechanical. We focus here on the latter, specifically on those characterized by a negative Poisson's ratio, meaning a positive ratio of transversal and axial strains, which are called auxetic. In contrast to standard materials, like a piece of rubber, they respond to stretching in uniaxial direction by thickening in the direction orthogonal to the applied force. Among the special characteristics of auxetics are 
enhanced shear moduli, increased fracture resistance, and higher shock absorption capacity, which renders them beneficial for numerous industrial applications.
Even though the roots of auxetics are reported to date back already to the 1920s~\cite{Voi28}, the topic started to attract increased attention in the materials science and engineering communities only decades later, when Lakes \cite{Lak87} was the first to manufacture foams with negative Poisson's ratio in 1987. Several mechanisms have since been presented in the literature that give rise to auxetic material behavior, for instance, re-entrant~honeycomb or bow tie structures in~\cite{ENHR91} (where also the term 'auxetic' from the ancient Greek word for 'stretchable' was coined), multiscale laminates in~\cite{Mil92} and, most relevant for this work, rotating rigid squares connected by hinges at the vertices, as introduced in \cite{GrE00} by Grima \& Evans. Also other rigid building blocks, such as triangles \cite{GCMACGE12,GrE06} or rectangles \cite{GAE04, GGAE05, GMA11}, have been used by these (and other co-) authors to produce a negative Poisson's ratio. More recently, there are thrusts of combing several geometric arrangements at the microscale to design state-of-the-art materials with unique characteristics~\cite{DMUK22,Mil13}. For more on the subject, we refer to the review article~\cite{GGLR11} and the references therein.
Compared with the research activities in the mechanics disciplines, the coverage of auxetic structures from the standpoint of mathematics seems rather  sporadic. The works by Borcea \& Streinu, e.g., \cite{BoS15, BoS18}, approach the problem by recoursing to algebraic geometry. They investigate what crystalline and artificial structures give rise to auxetic behavior and devise design principles, based on their earlier graph-theoretical papers on the deformation of periodic frameworks with rigid edges~\cite{BoS10}.

This paper contributes to the mathematical theory of auxetic metamaterials, approaching the problem from a new perspective, namely that of 
asymptotic variational analysis.
We study a class of composites in a two-dimensional setting of nonlinear elasticity that show a small-scale pattern of stiff and soft tiles arranged into a checkerboard structure as illustrated in Figure~\ref{fig:checkerboard} (cf.~also~\cite{GrE00, KoV13}). Working with a variational model (detailed in Section~\ref{subsec:setup} below), the task is to rigorously determine the effective material behavior and, in particular, characterize the attainable macroscopic deformations along with their energetic cost. To this end, we resort to homogenization via $\Gamma$-convergence (for a general introduction to $\Gamma$-convergence, see~\cite{Bra05,Dal93}). Our main result (Theorem~\ref{theo:homogenization_elastic_intro}) is a homogenization result that is non-standard compared to classical papers like~\cite{Bra85, Mul87} and the works on high-contrast media~\cite{ChC12, CCN17, DGP22, DKP21}. Instead, it can be interpreted in the context of asymptotic rigidity statements for other reinforcing elements like layers~\cite{ChK17, ChK20, DFK21} and  fibers~\cite{EKR22}.  This means, generally speaking, that the models are governed by an interesting interplay between the specific geometric pattern of the heterogeneities and a strong contrast in the elasticity constants, which leads to global effects and overall, to  a strongly restricted macroscopic material response. For the checkerboard composites under consideration, we prove in a suitable scaling regime between the stiffness and length scale parameters that the macroscopic deformations are given by affine maps describing conformal contraction, confirming a negative Poisson's ratio. 
Since the identified effective behavior coincides with that of an idealized version of the model with fully rigid elasticity of the stiff components (Theorem \ref{theo:homogenization_rigid}), one may also view our main theorem as part of a robustness analysis, which is significant with a view to the practical applicability of these metamaterials.  
As a closely related issue relevant for the manufacturing process,  which is, however, beyond the scope of this work, is a solid understanding of the sensitivity of imperfections and perturbations in the geometry of the small-scale structures. 
Further interesting research directions, which we plan to address in future work, include the study of metamaterial with rotating triangle structures and the optimal design of the stiff components. 
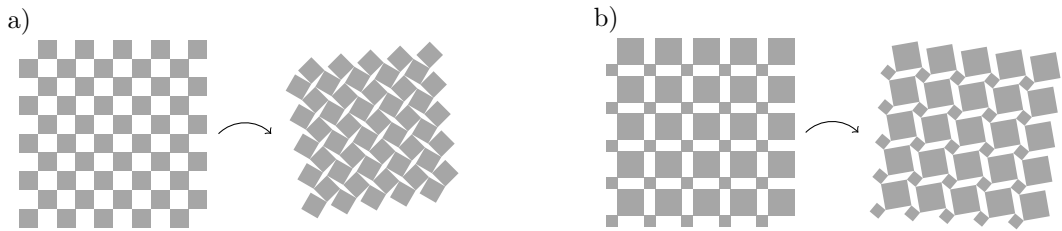
\begin{figure}
	\centering
	\begin{subfigure}{.48\linewidth}
		\begin{tikzpicture}[scale=.5]
			\def\glob{-30}
			\def\angle{75}
			\def\length{0.5}
			\pgfmathsetmacro\lengthb{1-\length}
			
			\draw (0,5.5) node {a)};
			\foreach \x in {0,...,4}{
				\foreach \y in {0,...,4}{
					\fill[black!35!white] (\x,\y) rectangle ($(\x,\y)+(\length,\length)$);  			
					\fill[black!35!white] ($(\x,\y)+(\length,\length)$) rectangle ($(\x,\y)+(1,1)$);	
			}}
			\draw [->] (5.3,2.5) to [out=45,in=135] (6	.7,2.5);
			\begin{scope}[shift = {(7.5,.5)}, rotate = {\glob}]
			\foreach \x in {0,...,4}{
				\foreach \y in {0,...,4}{
					\begin{scope}[shift={($(\length*\x,\length*\y) + (\angle:\x*\lengthb) + (\angle+90:\y*\lengthb)$)}]	
						\fill[black!35!white] (0,0) rectangle (\length,\length);						
						\begin{scope}[shift={($(\length,\length)-(\angle:\length) - (\angle+90:\length)$)}, rotate={\angle}]	
							\fill[black!35!white] (\length,\length) rectangle (1,1);					
						\end{scope}
					\end{scope}
			}}
			\end{scope}
		\end{tikzpicture}
	\end{subfigure}
	\begin{subfigure}{.48\linewidth}
		\begin{tikzpicture}[scale=.5]
			\def\glob{-40}
			\def\angle{50}
			\def\length{0.3}
			\pgfmathsetmacro\lengthb{1-\length}
		
			\draw (0,5.5) node {b)};
			\foreach \x in {0,...,4}{
				\foreach \y in {0,...,4}{
					\fill[black!35!white] (\x,\y) rectangle ($(\x,\y)+(\length,\length)$);  			
					\fill[black!35!white] ($(\x,\y)+(\length,\length)$) rectangle ($(\x,\y)+(1,1)$);	
			}}
			\draw [->] (5.3,2.5) to [out=45,in=135] (6	.7,2.5);
			\begin{scope}[shift = {(7,.4)}, rotate = {\glob}]
			\foreach \x in {0,...,4}{
				\foreach \y in {0,...,4}{
					\begin{scope}[shift={($(\length*\x,\length*\y) + (\angle:\x*\lengthb) + (\angle+90:\y*\lengthb)$)}]	
						\fill[black!35!white] (0,0) rectangle (\length,\length);						
						\begin{scope}[shift={($(\length,\length)-(\angle:\length) - (\angle+90:\length)$)}, rotate={\angle}]	
							\fill[black!35!white] (\length,\length) rectangle (1,1);					
						\end{scope}
					\end{scope}
			}}
			\end{scope}
		\end{tikzpicture}
	\end{subfigure}
	\caption{Illustrations of the auxetic deformation behavior of checkerboard-type composites with differently sized stiff squares (colored in gray).}\label{fig:checkerboard}
\end{figure}

\subsection{Setup of the problem}\label{subsec:setup}
Let $\Omega\subset \R^2$ be a bounded Lipschitz domain that models the reference configuration of a two-dimensional elastic body.   Deformations of that body are described by maps $u:\Omega\to \R^2$, which - unless mentioned otherwise - are taken to lie in $W^{1,p}(\Omega;\R^2)$ with $p>2$, and are thus, in particular, continuous by Sobolev embedding; note that some of our results also extend to $p=2$. 
We generally require our deformations to be orientation preserving, meaning with positive Jacobi-determinant almost everywhere, and forbid self-interpenetration of matter by imposing the Ciarlet-Ne\v{c}as condition~\cite{CiN87},
\begin{align}\label{ciarlet_necas} 
	\int_\Omega |\det \nabla u| \dd x \leq |u(\Omega)|,\tag{CN}
\end{align}
which corresponds to injectivity of $u$ a.e. in $\Omega$; for more on the topic of global invertibility of Sobolev maps, we refer, for instance, to the classical works~\cite{Bal81, MuS95} or to~\cite{BHM20, HMO21, Kro20} for some recent developments. 
With these assumptions, we introduce the class of admissible deformations as
\begin{align}\label{Acal}
	\Acal=\{u\in W^{1,p}(\Omega;\R^2): \det \nabla u>0 \text{ a.e.~on $\Omega$ and $u$ satisfies \eqref{ciarlet_necas}}\}.
\end{align}

Next, we formalize the geometry of the material heterogeneities, arranged in a checkerboard-like fashion. To this end, the periodicity cell $Y=(0,1]^2$ is subdivided into four tiles, precisely,
\begin{align}\label{tiles}
	Y_1=(0,\lambda]^2,\  Y_2=(0, \lambda]\times (\lambda, 1], \ Y_3 = (\lambda,1]^2,\  Y_4=(\lambda, 1]\times (0, \lambda]
\end{align}
for a given parameter $\lambda\in (0,1)$, and we define
\begin{center}
	$\Ystiff = Y_1\cup Y_3$ \quad and \quad $\Ysoft=Y_2\cup Y_4$, 
\end{center}
so that $Y = \Ystiff \cup \Ysoft$.
\begin{figure}
	\begin{tikzpicture}
		\def\length{.6}
		\pgfmathsetmacro\lengthb{1-\length}
		\begin{scope}[scale=3]
			\draw (0,0) rectangle (1,1);
			\draw [fill=black!25!white] (0,0) rectangle (\length,\length);
			\draw [fill=black!25!white] (\length,\length) rectangle (1,1);
			\draw (\length/2,\length/2) node {$Y_1$};
			\draw (\length/2,{\length+\lengthb/2}) node {$Y_2$};
			\draw ({\length+\lengthb/2},{\length+\lengthb/2}) node {$Y_3$};
			\draw ({\length+\lengthb/2},\length/2) node {$Y_4$};
			\draw (0,1) node [anchor=south east] {$Y$};
			\draw [<->] (0,-.1) --++(\length/2,0) node[anchor=north] {$\lambda$} --++ (\length/2,0); 
			\draw [<->] (-.1,0) --++(0,\length/2) node[anchor=east] {$\lambda$} --++ (0,\length/2);
			\draw [<->] (1.1,0) --++(0,.5) node[anchor=west] {$1$} --++ (0,.5);
			\draw (-.25,1) node [anchor=east] {$a)$};
			\draw (1.5,1) node [anchor=west] {$b)$};
		\end{scope}
		\begin{scope}[shift={(5,1)},scale=.8]
			\draw (0,0) to [out=90,in=180] (2,2) to [out=0,in=200] (3,2.25) to [out=20,in=225] (4,3) to [out=45, in=135] (6,3) to [out=-45, in=90] (6.3,2) to [out=-90,in=45] (5,0) to [out=225,in=0] (2,-1) to [out=180,in=-90] (0,0);
			\draw [<->](1.6,-1.2)--++(.4,0) node[anchor=north] {$\eps$} --++(.4,0);
			\draw (2,2.5) node {$\Omega$};
			\clip (0,0) to [out=90,in=180] (2,2) to [out=0,in=200] (3,2.25) to [out=20,in=225] (4,3) to [out=45, in=135] (6,3) to [out=-45, in=90] (6.3,2) to [out=-90,in=45] (5,0) to [out=225,in=0] (2,-1) to [out=180,in=-90] (0,0);
			
			\foreach \x in {-10,...,10}{
				\foreach \y in {-10,...,10}{
				\begin{scope}[scale=.8]
					\draw [fill=black!25!white] ($(0,0)+(\x,\y)$) rectangle ($(\length,\length)+(\x,\y)$);
					\draw [fill=black!25!white] ($(\length,\length)+(\x,\y)$) rectangle ($(1,1)+(\x,\y)$);
				\end{scope}
				}
			}
		\end{scope}
	\end{tikzpicture}
	\caption{a) The partition of the unit cell $Y$ into the four tiles $Y_1,\ldots,Y_4$ as in \eqref{tiles} with $\Ystiff=Y_1\cup Y_3$ colored in gray and $\Ysoft=Y_2\cup Y_4$ in white. b) The reference configuration $\Omega$ with its stiff components $\Omega\cap\eYstiff$ marked in gray.}\label{fig:setup}
\end{figure}
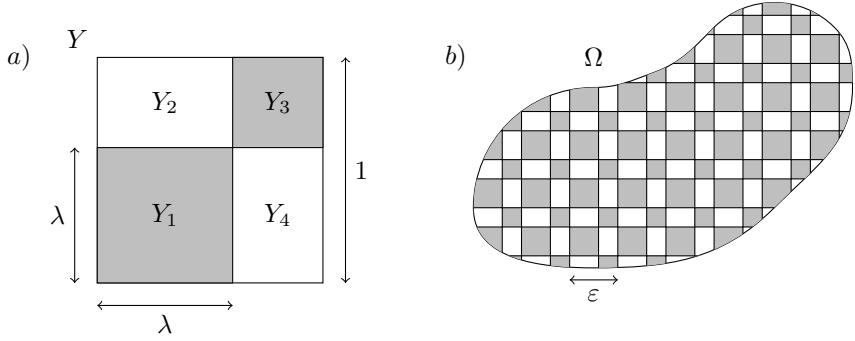
Note that, without further mentioning, the sets $\Ystiff, \Ysoft, Y_1, \ldots, Y_4$ will also be identified throughout with its $Y$-periodic extensions.
The stiff and soft components of the elastic body forming a periodic pattern at length scale $\eps>0$ are then described by the intersection of $\Omega$ with 
$\eYstiff$ and $\eYsoft$, respectively.  For an illustration of the geometric setup, we refer to Figure \ref{fig:setup}.

The material properties of the composite are modeled by the two energy elastic densities $\eWstiff$ and $\Wsoft$. 
On the stiff parts, we take $\eWstiff: \R^{2\times 2}\to [0, \infty]$ for $\eps>0$ as a continuous function such that
\begin{align}\label{Wrig}
	\begin{split}
		 \eWstiff = 0 \text{ on } \SO(2) \qand \frac{1}{c\eps^\beta} \dist^{p}(F,\SO(2))\leq \eWstiff(F)&\quad \text{ if } \det F >0,\\
		\eWstiff(F) = \infty &\quad \text{ if } \det F\leq 0,
	\end{split}
\end{align}
with a constant $c>0$ and a parameter $\beta>0$. While rotations do not cost any energy, deviations from $\SO(2)$ are energetically penalized with diverging elastic constants as $\eps$ tends to zero, i.e., the stiff material is asymptotically rigid. Qualitatively, this means that the stiff components become stiffer and stiffer as the length scale shrinks. The tuning parameter $\beta$ controls the degree of increasing stiffness and will be chosen later to be sufficiently large.

On the soft components, we consider a continuous function $\Wsoft:\R^{2\times 2}\to [0,\infty]$ that satisfies for $p\geq 2$,
\begin{align}\label{Wsoft}
	\begin{split}
		\frac{1}{c}|F|^p + \frac{1}{c}\theta(\det F) - c \leq \Wsoft(F) \leq c|F|^p + c\theta(\det F) + c\quad&\text{ if } \det F>0,\\
		\Wsoft(F) = \infty\quad&\text{ if } \det F\leq 0,
	\end{split}
\end{align}
where $c>0$  and $\theta:(0,\infty)\to[0,\infty)$ is a convex function such that
\begin{align*}
	\theta(st) \leq c(1+\theta(s))(1+\theta(t))\quad\text{ for all } s,t\in(0,\infty),
\end{align*}
cf.~also \cite[Equations (2.1), (2.2)]{CoD15}.

Merging the modeling assumptions introduced above gives rise to a variational problem with the elastic energy functional (defined on deformations with zero mean value)
\begin{align}\label{energy}
	\Ical_\eps: L^p_0(\Omega;\R^2)\to [0,\infty],\quad u\mapsto 
	\begin{cases}
		\displaystyle\int_\Omega W_\eps\big(\tfrac{x}{\eps},\nabla u(x)\big) \dd x	&\text{ for } u\in\Acal,\\
		\infty	&\text{ otherwise,}
	\end{cases}
\end{align}
with the inhomogeneous energy density
\begin{align}\label{density}
	W_\eps: \R^2\times \R^{2\times2}\to [0,\infty],\quad (y,F)\mapsto \Wsoft(F)\one_{\Ysoft}(y) +  \eWstiff(F) \one_{\Ystiff}(y). 
\end{align}
Hence, the observed deformations of the composite with checkerboard structure at scale $\eps$, correspond to minimal energy states of $\Ical_\eps$ - up to accounting for external forces, which we do not explicitly include here, as they can be handled via continuous perturbations. In the following, we focus on capturing the effective material behavior through the convergence of minimizers in the limit of vanishing length scale.

\subsection{The main result} With this setup at hand, we can now state the main contribution of this work, the following homogenization result by $\Gamma$-convergence for the elastic energies $(\Ical_\eps)_\eps$ as $\eps\to 0$. 
\begin{theorem}[Homogenization of checkerboard structures]\label{theo:homogenization_elastic_intro}
 Let $\Omega\subset \R^2$ be a bounded Lipschitz domain, $p \geq 2$, and let $\Ical_\eps$ for $\eps>0$ as in~\eqref{energy}, \eqref{Acal}, and \eqref{density} with $\eWstiff$ as in \eqref{Wrig} and $\Wsoft$ as in \eqref{Wsoft} such that $\Wsoft^\qc$ is polyconvex.
Then, the family of functionals $(\Ical_\eps)_\eps$ $\Gamma$-converges for $\eps\to 0$ with respect to the strong $L^p(\Omega;\R^2)$-topology to
\begin{align}\label{Gamma_limit}
	\Ical_{\rm hom}: L^p_0(\Omega;\R^2) \to [0,\infty],\ u\mapsto
		\begin{cases} 
			|\Omega| \Whom(F) & \text{if $\nabla u=F\in K$,}\\
			\infty & \text{otherwise,}
		\end{cases}
\end{align} 
where
\begin{align}\label{Klambda}
	K:=\{\lambda S + (1-\lambda)R : R,S\in \SO(2), Re_1\cdot Se_1 \geq 0\}=\{\alpha Q: \sqrt{|\Ystiff|}\leq \alpha\leq 1, Q\in \SO(2)\}
\end{align}
and the homogenized density is given for $F\in K$ by
\begin{align}\label{Whom}
	\Whom(F)=\frac{1}{2}|\Ysoft|\min_{R, S\in \SO(2), \lambda S+(1-\lambda)R=F, Re_1\cdot Se_1\geq 0}\big(\Wsoft^\qc(Se_1|Re_2) + \Wsoft^\qc(Re_1|Se_2)\big).
\end{align} 

Moreover, any sequence $(u_\eps)_\eps\subset L^p_0(\Omega;\R^2)$ with $\sup_{\eps}\Ical_\eps(u_\eps)<\infty$ has a subsequence that converges weakly in $W^{1,p}(\Omega;\R^2)$ to an affine function $u:\Omega\to \R^2$ with vanishing mean value and $\nabla u\in K$.
\end{theorem}
This theorem~shows rigorously that the effective behavior of materials with high-contrast checkerboard structures, governed by the variational problem with functional $\Ical_{\rm hom}$, is restricted to affine conformal contractions. Indeed, the macroscopically attainable deformations correspond exactly to the domain of the limit energy $\Ical_{\rm hom}$, which comprises all affine functions with zero mean value and whose gradients are suitable positive scalar multiples of rotation matrices. The latter implies that the Poisson's ratio of the composites under consideration is $-1$, which reflects their auxetic nature, see Remark~\ref{rem:Klambda} b). A comparison inspired by classical homogenization results like \cite{Bra85, Mul87} reveals that the homogenized density $\Whom$ coincides essentially (that is, up to maximal compressions) with the cell formula associated to the related model where the stiff tiles are fully rigid; we refer to Remark~\ref{rem:cell_formula} for more details. 
 
Further, two comments about the technical hypotheses in the previous theorem are in order.  
\begin{remark}
a) Note that the statement of Theorem~\ref{theo:homogenization_elastic_intro} is sensitive to the regularity of the admissible functions and fails for $p<2$.  The intuition is that the material can break up at the connecting joints between two stiff neighboring squares, when the deformations, here $W^{1,p}$-functions, can have discontinuities in isolated points, so that a large class of limit maps can be reached, cf.~Proposition~\ref{prop:anythingcanhappen}.
Interestingly, this observation about the critical role of the integrability parameter is in contrast to related homogenization results for materials with strict soft inclusions~\cite{BrG95, DuG16} or layered materials \cite{ChK17, ChK20}, which are valid for any $p>1$.

\smallskip

b) The condition $\beta>2p-2$ on the tuning parameter emerges naturally from our approach (see Section~\ref{subsec:approach}), but it is currently not clear whether this scaling regime is optimal. Answering this question remains an interesting open problem. 
In particular,  we are not aware of explicit constructions of bounded energy sequences converging to limit deformations other than affine conformal contractions when $\beta\in (0, 2p-2)$.
\end{remark}

\subsection{Approach and methodology} \label{subsec:approach}
The stepping stone for our analysis is a solid understanding of the related model with rigid elasticity on the stiff tiles, which results formally by replacing the density $\eWstiff$ in~\eqref{density} by
\begin{align}\label{Wrig2}
	\Wrig(F) = 
		\begin{cases}
			0 &\text{ for } F\in \SO(2),\\
			\infty &\text{ otherwise, } 
		\end{cases}\qquad \text{$F\in \R^{2\times 2}$.}
\end{align}
Due to this stricter assumption, the macroscopically attainable deformations are easier to characterize, since
the possible deformations even for structures at scale $\eps>0$, that is, $u_\eps \in \Acal$ with 
\begin{align}\label{exactinclusion_intro}
	\nabla u_\eps \in \SO(2)\qquad \text{ on $\Omega\cap\eps\Ystiff$,}
\end{align}
which by well-known rigidity results (e.g.,~\cite{Res67}) is equivalent to $\nabla u_\eps$ coinciding with a single rotation on each connected component of $\eps\Ystiff$,
are rather limited. Indeed, each such $u_\eps$ can be characterized as the sum of a function that is piecewise affine on the tiles with at most four different gradients, uniquely determined by two rotations, and local modulations on each soft tile with a Sobolev function with zero boundary values (Corollary~\ref{cor:preservation}). 
This follows from  basic geometric  considerations that allow us to determine the rigid motions acting on the boundaries of the rigid components, while accounting for orientation preservation. 
The next step of identifying the weak $W^{1,p}$-limits of sequences $(u_\eps)_\eps$ is then standard (see Proposition~\ref{prop:K_lambda}) and yields the affine deformations with gradients in the set $K$ in~\eqref{Klambda} - so exactly the finite-energy states of the homogenized functional~\eqref{Gamma_limit}.

Besides the insight into the case with fully rigid components, we wish to highlight two technical ingredients that are substantial for the proof of Theorem~\ref{theo:homogenization_elastic_intro}. They are both 
embedded in a general proof strategy of asymptotic rigidity results (cf.~\cite{ChK20, DFK21, EKR22} and also~\cite{FJM02})
essential for expanding the observations on the asymptotic behavior of sequences $(u_\eps)_\eps\subset \Acal$ when the exact differential inclusion~\eqref{exactinclusion_intro} is weakened to the approximate version
\begin{align}\label{exactinclusion_intro}
	\int_{\Omega\cap \eps\Ystiff}\dist^{p}(\nabla u_\eps,\SO(2)) \leq C\eps^\beta
\end{align}
with a constant $C>0$. 

The first key tool is a quantitative rigidity estimate in the spirit of  the seminal work by Friesecke, James \& M\"uller \cite{FJM02} applicable to cross structures, as stated in Lemma~\ref{lem:luftschloss}; by an (unscaled) cross structure $E'$, we understand a non-connected open set contained in $\Ystiff$ consisting of the four stiff neighboring squares of a single soft rectangle.
If one applies~\cite[Theorem~3.1]{FJM02} individually to a function $u\in W^{1, p}(E';\R^2)$ restricted to each of the connected components of $E'$, this yields four potentially different rotation matrices close to the gradient of $u$. Lemma~\ref{lem:luftschloss} states that only two rotations are in fact enough. We prove by careful geometric arguments in combination with an approximate version of the non-interpenetration condition (see~Lemma~\ref{lem:approx_ciarlet}) that the rotations on opposite squares can be chosen identical while preserving suitable control on the error terms. More precisely, the $L^p$-error between the rotations and the gradients of $u$ is given terms of $\delta^{1/2}$ with $\delta=\norm{\dist(\nabla u, \SO(2))}_{L^p(E')}$; note that the square root is due to our technical approach and comes in through Pythagoras' theorem. For the scaling analysis associated with Lemma~\ref{lem:luftschloss}, we refer to Remark~\ref{rem:scaling}. 
 
The second tool is a Poincar\'e-type inequality with uniform constants for checkerboard structures, which has  -  in contrast to Lemma~\ref{lem:luftschloss} - a global character. Roughly speaking, we show that a function $u\in W^{1,p}(\Omega;\R^2)$ with vanishing mean value on the stiff parts $\Omega\cap\eps \Ystiff$ and the property that the values of $u$ in interior of $\Omega$ control $u$ also in a boundary layer, then the $L^p$-norm of $u$ can be estimated by $\norm{\nabla u}_{L^p(\Omega\cap\eps \Ystiff;\R^2)}$ multiplied with a constant independent of $\epsilon$; the precise statement can be found in~Lemma~\ref{lem:poincare2}.
Our proof is inspired by a classical extension result in the literature. Based on~\cite{ACDP92} by Acerbi, Chiad\`o Piat, Dal Maso \& Percivale, we derive an approximate extension result tailored for our purposes, which then allows us to mimic the usual indirect proof of Poincar\'e's inequality.
To handle the technicalities around the joints, where the sets of stiff tiles does not have Lipschitz boundary, we proceed in two steps. We first extend our functions partially from the stiff to the soft parts by standard reflection arguments, leaving out small balls around the corners, and then fill them via an extension according to~\cite{ACDP92}.

\subsection{Outline} This paper is organized as follows. Section~\ref{sec:rigid} is concerned with the analysis of the auxiliary model with full rigid tiles.
After establishing  the deformation behavior of the individual soft components on the local level in Section~\ref{sec:rigid_aux}, we characterize in Section~\ref{sec:rigid_deform} the set of attainable macroscopic deformations in terms of affine conformal contractions. The corresponding homogenization result via variational convergence, which gives rise to the effective energy $\Ical_{\rm hom}$ as $\Gamma$-limit, is proven in Section \ref{sec:rigid_hom2}. We conclude this first part of the paper in Section~\ref{sec:rigid_assumptions} with a detailed discussion of our various modeling assumptions, including the effects of requiring orientation preservation, the Ciarlet-Ne\v{c}as condition and $p>2$.
The core of this work is Section \ref{sec:elastic}, where we investigate the model with diverging elastic energy contribution on the stiff parts as introduced in~Section~\ref{subsec:setup}.  
We provide the technical basis in Section~\ref{sec:elastic_aux} by proving the two technical key tools, a quantitative rigidity estimate for cross structures and a Poincar\'e-type inequality for checkerboard structures.   Section~\ref{sec:elastic_deform} then covers the proof of the compactness statement in Theorem \ref{theo:homogenization_elastic_intro} and determines the possible effective deformations through the weak closure of the admissible deformations of small-scale checkerboard structures.
Finally, the remaining parts of the proof of the main result Theorem \ref{theo:homogenization_elastic_intro} can be found in Section~\ref{sec:elastic_hom}.

\subsection{Notation}
The standard unit vectors in $\R^2$ are denoted by $e_1$ and $e_2$. For two vectors $a,b\in\R^2$, we write $a\cdot b$ for their scalar product. 
The one-dimensional unit sphere $\Scal^1$ consists of all vectors in $\R^2$ with unit length. 
For $a\in \R^{2}$, let $a^\perp := -a_2e_1 +a_1e_2$, while for $A\in \R^{2\times 2}$, we define $A^\perp=(Ae_2|-Ae_1)$. 
We equip $\R^{m\times n}$ for $m,n\in\{1,2\}$ with the standard Frobenius norm, that is, $|A|=\sqrt{\Tr(A^TA)}$ for $A\in\R^{m\times n}$ where $A^T$ is the transpose of $A$ and $\Tr$ denotes the trace operator.
We write $\Id$ for the identity matrix in $\R^{2\times 2}$ and $\SO(2)$ stands for the special orthogonal group of matrices in $\R^{2\times 2}$.

If $U,V\subset \R^2$, then $U+V:=\{u+v:u\in U,v\in V\}$ describes their Minkowski sum.
The notation $A\Subset B$ for two sets $A,B\subset \R^{2}$ means that $A$ is compactly contained in $B$.
We refer to a non-empty, open, connected set as a domain.
Given $x_0\in \R^2$ and $R,r>0$, we set $B(x_0,R) = \{x\in \R^2: |x-x_0|<R\}$ as the ball around $x_0$ with radius $R$, and 
\begin{align}\label{annulus}
	A(x_0,R,r)=\{x\in \R^2 : r<|x-x_0|<R\}
\end{align}
as the annulus around $x_0$ with outer radius $R$ and inner radius $r$.
We write $|\cdot|$ for the Lebesgue measure and use $\sharp(\cdot)$ for the counting measure.

For an open set $U\subset \R^2$ and $1\leq p \leq \infty$, we use the standard notation for Lebesgue and Sobolev spaces $L^p(U;\R^m)$, $W^{1,p}(U;\R^m)$ and $W^{1,p}_0(U;\R^m)$ with vanishing boundary values in the sense of traces, and define $L^p_0(U;\R^2):=\{u\in L^p(U;\R^2) : \int_U u(x) \dd x= 0\}$. 
For functions $f:\R^{2\times 2}\to [0,\infty]$, we briefly write $f(Ae_1|Ae_2)$ instead of $f((Ae_1|Ae_2))$ for $A=(Ae_1|Ae_2)$.
The indicator function $\one_U$ of a set $U\subset \R^2$ is identical to $1$ on $U$ and vanishes everywhere else.
Furthermore, we define
\begin{align}\label{qc_envelope}
	f^\qc(F) := \inf_{\ffi\in W^{1,\infty}_0(D;\R^2)} \dashint_D f(F+\nabla \ffi) \dd x,
\end{align}
where $D\subset \R^2$ is an arbitrary bounded open set and $\dashint$ describes the mean integral, as the quasiconvex envelope of $f$.
We say that $f$ is polyconvex if there exists a convex and lower semicontinuous function $g:\R^{2\times 2}\times \R\to [0,\infty]$ such that $f(F) = g(F,\det F)$ for all $F\in \R^{2\times 2}$.

Throughout the document, we use $C>0$ for generic constants which may differ from term to term; if we want to highlight the dependence of certain quantities, we include them in parentheses. 
Finally, families indexed with a continuous parameter $\eps>0$ refer to any sequence $(\eps_j)_j$ with $\eps_j\to 0$ as $j\to \infty$.

\section{Analysis of the model with fully rigid tiles}\label{sec:rigid}

\subsection{Auxiliary results}\label{sec:rigid_aux}
The following lemmas identify local restrictions on neighboring rotations of the stiff parts in the checkerboard structure 
and shows that the boundary values of a deformation of any single soft tile coincide with those of a piecewise affine function.

\begin{lemma}\label{lem:affine}
Let $E\subset \R^2$ be an open rectangle  
with two sides of length $l$ parallel to $e_1$, two sides of length $\mu l$ and $\partial_iE=\Gamma_i$ for $i=1, \ldots, 4$ the linear pieces of the boundary $\partial E$, numbered clockwise, starting in the lower left corner. If $u\in W^{1,p}(E;\R^2)$ with $p \geq 2$ is such that 
\begin{align} \label{boundary}
u|_{\Gamma_i} = R_i x+b_i\quad\text{ with $R_i\in \SO(2)$ and $b_i\in \R^2$ for $i=1, \ldots, 4$,}
\end{align}
 the following two statements hold:

a) There exist matrices $R, S \in \SO(2)$ depending only on $u|_{\partial E}$ as well as functions $F_{\pm}, G_{\pm}: (\SO(2))^{2} \times (0,1) \to \SO(2)$ being independent of $E$ and $u$ 
such that 
\begin{align*}
\bigcup_{i=1}^4 R_i
\subset \{R, S, F_{\pm}(R,S,\mu), G_{\pm}(R,S,\mu)\}
\subset \SO(2).
\end{align*} 
In particular, it holds that $F_{\pm}(R,S,1) = \pm R^\perp $, $G_{\pm}(R,S,1) = \pm S^\perp$.

b) There exist $\varphi\in W_0^{1,p}(E;\R^2)$ and a piecewise affine function $v:E\to \R^2$ with at most two different gradients in the set 
\begin{align*}
\{(Se_1|Re_2), (F_{+}(R,S,\mu)e_1|G_{-}(R,S,\mu)e_2), (F_{-}(R,S,\mu)e_1|G_{+}(R,S,\mu)e_2)\} \subset \R^{2\times 2}
\end{align*}
such that
\begin{align}\label{splitting}
u=v+\varphi.
\end{align}
\end{lemma}

\begin{proof}
Case 1: $E$ is an open square.
\\[-3mm]

After scaling and shifting, we may assume without loss of generality that $E=(0, 1)^{2}$ and $u(0)=0$. 
Due to $p \geq 2$, the trace of $u$ is continuous on $\partial E$. For $p >2$, this is a direct consequence of the fact that $W^{1,p}(\Omega;\R^2)$ embeds into the H\"older space $C^{0, 1-\frac{2}{p}}(\overline{\Omega};\R^2)$. For $p=2$, this follows from the fact that the boundary values \eqref{boundary} satisfy the assertion (c) of Theorem~1.5.2.3 in \cite{Gr85} if and only if   
they are continuous on $\partial E$.
Because of \eqref{boundary}, the continuity of $u(\partial E)$ and the fact that triangles which correspond in their three side lengths are congruent, it follows that $u(\partial E)$ has to be either
\begin{itemize}
\item[$i)$] the boundary of a rhombus with side length $1$ or 
\item[$ii)$] a straight line of length $2$ or
\item[$iii)$] a hook with two arms of length $1$ each or
\item[$iv)$] a straight line of length $1$,
\end{itemize}
cf.~Figure~\ref{fig2}.

\begin{figure}
\begin{tikzpicture}
	\begin{scope}[xshift = -6.5cm, yshift = 0cm, scale = 1]
		\filldraw[fill=Cerulean!50!white, draw=black]  (0, -1.5) --  (-1,0) -- (0,1.5)  -- (1,0) -- (0,-1.5); 
		\draw (-1,1) node {$i)$};
		\draw (-0.7,-0.9) node {$1$}; 
		\draw (0.7,-0.9) node {$1$};
	\end{scope}

	\begin{scope}[xshift = -3cm, yshift = 0cm, scale = 1]
		\filldraw[fill=Cerulean!50!white, draw=black]  (-1.5, 0) --  (1.5,0); 
		\draw[color=black] (0,-0.1)--(0, 0.1);
		\draw (-1,1) node {$ii)$};
		\draw (-0.7,-0.3) node {$1$}; 
		\draw (0.7,-0.3) node {$1$};
	\end{scope}

	\begin{scope}[xshift = 0.5cm, yshift = 0cm, scale = 1]
		\fill[Cerulean!50!white]  (0, 1) --  (0,-0.5) -- (1.5,-0.5)  -- (0,1); 
		\draw[color=black] (0, 1) -- (0, -0.5) --  (1.5,-0.5); 
		\draw (-1,1) node {$iii)$};
		\draw (-0.3,0.2) node {$1$}; 
		\draw (0.7,-0.8) node {$1$};
	\end{scope}

	\begin{scope}[xshift = 4cm, yshift = 0cm, scale = 1]
		\draw[color=black] (-0.5,0) -- (1,0) -- (-0.5,0); 
		\draw (-1,1) node {$iv)$};
		\draw (0.3,-0.3) node {$1$};
	\end{scope}
\end{tikzpicture}
\caption{ Illustration of the boundary deformations $u(\partial E)$ for $\mu =1$.}\label{fig2}
\end{figure}
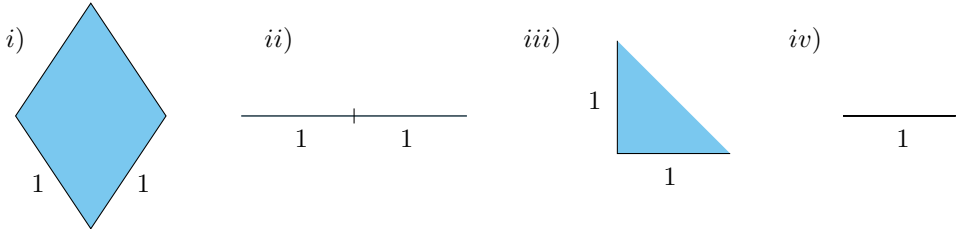

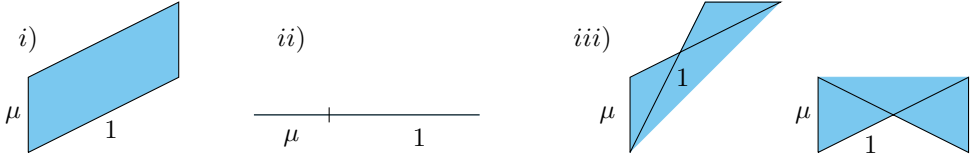
\begin{figure}
\begin{tikzpicture}
	\begin{scope}[xshift = -6.5cm, yshift = -4cm, scale = 1]
		\filldraw[fill=Cerulean!50!white, draw=black]  (-1, -0.5) --  (-1,0.5) -- (1,1.5)  -- (1,0.5) -- (-1,-0.5);
		\draw (-1,1) node {$i)$};
		\draw (-1.2,0) node {$\mu$}; 
		\draw (0.1,-0.2) node {$1$};
	\end{scope}

	\begin{scope}[xshift = -3cm, yshift = -4cm, scale = 1]
		\filldraw[fill=Cerulean!50!white, draw=black]  (-1.5, 0) --  (1.5,0); 
		\draw[color=black] (-0.5,-0.1)--(-0.5, 0.1);
		\draw (-1,1) node {$ii)$};
		\draw (-1,-0.3) node {$\mu$}; 
		\draw (0.7,-0.3) node {$1$};
	\end{scope}

	\begin{scope}[xshift = 1cm, yshift = -4cm, scale = 1]
		\filldraw[fill=Cerulean!50!white, draw=Cerulean!50!white] (-0.5, -0.5) -- (0.5, 1.5) --  (1.5, 1.5)  ; 
		\filldraw[fill=Cerulean!50!white, draw=black] (-0.5, -0.5) --  (-0.5,0.5) -- (1.5,1.5)  -- (0.5, 1.5) -- (-0.5, -0.5); 
		\draw (-1,1) node {$iii)$};
		\draw (-0.8,0) node {$\mu$}; 
		\draw (0.2,0.5) node {$1$};
	\end{scope}

	\begin{scope}[xshift = 4cm, yshift = -4cm, scale = 1]
		\filldraw[fill=Cerulean!50!white, draw=Cerulean!50!white] (-1, 0.5) -- (1,0.5) --  (1, -0.5)  ; 
		\filldraw[fill=Cerulean!50!white, draw=black]  (-1, -0.5) --  (-1,0.5) -- (1,-0.5)  -- (1,0.5) -- (-1,-0.5);
		\draw (-1.2,0) node {$\mu$}; 
		\draw (-0.3,-0.4) node {$1$};
	\end{scope}
\end{tikzpicture}
\caption{ Illustration of the boundary deformations $u(\partial E)$ for $0<\mu<1$.}\label{fig2mu}
\end{figure}

For $i)$ and $ii)$, we observe that $R_1=R_3=:R$ and $R_2=R_4=:S$. Hence, the affine map $v:E\to \R^2$ with $\nabla v=(Se_1|Re_2)$ and $v(0)=0$ satisfies $u(\partial E)=v(\partial E)$. 

The situations $iii)$ and $iv)$ imply that 
\begin{align}
 R_4 e_1&= R_1e_2 \;\text{ and }\; R_3e_2=R_2e_1 \label{case_square_a}
\end{align}
or
\begin{align} 	
R_2e_1&= -  R_1e_2 \;\text{  and }\; R_3e_2=- R_4e_1. \label{case_square_b}
\end{align}

For \eqref{case_square_a}, let $R:=R_1$, $S:=R_2$ and $E_{\rm nw}=\{x\in (0,1)^2: x_2>x_1\}$, $E_{\rm se}=\{x\in (0,1)^2:x_2<x_1\}$ be the open triangles that result from cutting $E$ at the diagonal. We define $v:E\to \R^2$ via $v(0)=0$ and 
\begin{align}\label{p1}
\nabla v=\begin{cases} (Se_1|Re_2) & \text{in $E_{\rm nw}$,}\\
(Re_2|Se_1) & \text{in $E_{\rm se}$.}
\end{cases}
\end{align} 
By construction, $v$ is compatible along the diagonal, hence $v\in W^{1, \infty}(E;\R^2)$ and $v(\partial E)=u(\partial E)$.

We argue similarly for \eqref{case_square_b}, setting $R:=R_1$, $S:=R_4$, $E_{\rm sw}=\{x\in (0,1)^2: x_2<1-x_1\}$ and $E_{\rm ne}=\{x\in (0,1)^2:x_2>1-x_1\}$. Defining a continuous function $v:E\to \R^2$  by
$v(0)=0$ and
\begin{align}\label{p2}
\nabla v=\begin{cases} (Se_1|Re_2) & \text{in $E_{\rm sw}$,}\\
- (Re_2|Se_1) & \text{in $E_{\rm ne}$}
\end{cases}
\end{align} 
yields a piecewise affine function with $v(\partial E) = u(\partial E)$. 

Hence, we obtain the statements of the lemma for $\mu=1$ by defining
\begin{align} \label{Fpm2}
F_{\pm}(R,S,1) & = \pm R^\perp   , \\ \label{Gpm2}
G_{\pm}(R,S,1) & = \pm S^\perp .
\end{align} 
\medskip
\\
Case 2: $E$ is an open rectangle and $0 < \mu <1$. 
\\[-3mm]

After scaling and shifting, we may assume without loss of generality that $E=(0, 1) \times (0,\mu)$ and that $u(0)=0$. 
Because of $p \geq 2$, the trace of $u$ is again continuous on $\partial E$. For the same reasons as above it follows that $u(\partial E)$ is either 
\begin{itemize}
\item[$i)$] the boundary of a parallelogram with side lengths $1$ and $\mu$ or 
\item[$ii)$] a straight line of length $1+\mu$ or
\item[$iii)$] the union of the two sides below or above one diagonal of a parallelogram  from $i)$ with the reflection of the two other sides on that diagonal,
\end{itemize}
cf.~Figure~\ref{fig2mu}. 

For $i)$ and $ii)$, we observe that $R_1=R_3=:R$ and $R_2=R_4=:S$. Hence, the affine map $v:E\to \R^2$ with $\nabla v=(Se_1|Re_2)$ and $v(0)=0$ satisfies $u(\partial E)=v(\partial E)$. 

The situation $iii)$ implies that 
\begin{align}
	R_4 e_1 &= F_{+}(R_1,R_2,\mu)e_1 \;\text{ and }\; R_3e_2= G_{-}(R_1,R_2,\mu)e_2,\label{case_rectangle_a}
\end{align}
or
\begin{align}
	R_2 e_1 &= F_{-}(R_1,R_4,\mu)e_1 \;\text{ and }\; R_3e_2=G_{+}(R_1,R_4,\mu)e_2,\label{case_rectangle_b}
\end{align}
where $F_{\pm}, G_{\pm}$ are given by 
\begin{align} \label{Fpm1}
	 F_{\pm}(R,S,\mu) & = \frac{\pm 2 \mu + 2\mu^{2}\, Se_1 \cdot Re_2 }{1+ \mu^{2} \pm 2 \mu \, Se_1 \cdot Re_2} \, R^\perp +
		\frac{1-\mu^{2}}{1+\mu^{2} \pm 2 \mu\, Se_1 \cdot Re_2} \, S , \\[2mm] \label{Gpm1}
	G_{\pm}(R,S,\mu) & = \frac{\pm 2 \mu - 2\, Se_1 \cdot Re_2 }{1+ \mu^{2} \mp 2 \mu \, Se_1 \cdot Re_2} \, S^\perp -
		\frac{1-\mu^{2}}{1+\mu^{2} \mp 2 \mu\, Se_1 \cdot Re_2} \, R , 
\end{align}
which can be directly computed by using the facts that the reflection $\mathcal{R}_u v$ of a vector $v$ across a line $\{\lambda u: \lambda \in \mathbb{R}\}$ is determined by
\begin{align*}
\mathcal{R}_u v = 2\, \frac{v \cdot u}{u \cdot u}\, u - v
\end{align*}
and that $|Se_1|=|Re_2|=1$.
We remark that the denominators in \eqref{Fpm1}--\eqref{Gpm1} cannot be equal to $0$ for $0 < \mu <1$ and that in the case of $\mu=1$, \eqref{Fpm1}--\eqref{Gpm1} coincide with \eqref{Fpm2}--\eqref{Gpm2} if $Se_1 \cdot Re_2 \neq \mp 1$. 

For \eqref{case_rectangle_a}, let $R:=R_1$, $S:=R_2$ and $E_{\rm nw}=\{x\in (0,1)\times (0,\mu): x_2>\mu x_1\}$, $E_{\rm se}=\{x\in (0,1)^2:x_2<\mu x_1\}$ be the open triangles that result from cutting $E$ at the diagonal. We define $v:E\to \R^2$ via $v(0)=0$ and 
\begin{align}\label{p1mu}
\nabla v=\begin{cases} (Se_1|Re_2) & \text{in $E_{\rm nw}$,}\\
(F_{+}(R,S,\mu)e_1|G_{-}(R,S,\mu)e_2) & \text{in $E_{\rm se}$.}
\end{cases}
\end{align} 
By construction, $v$ is compatible along the diagonal, hence $v\in W^{1, \infty}(E;\R^2)$, and $v(\partial E)=u(\partial E)$.

We argue similarly for \eqref{case_rectangle_b}, setting $R:=R_1$, $S:=R_4$, $E_{\rm sw}=\{x\in (0,1)\times (0,\mu): x_2<\mu(1-x_1)\}$ and $E_{\rm ne}=\{x\in (0,1)\times (0,\mu):x_2>\mu(1-x_1)\}$. Defining a continuous function $v:E\to \R^2$ by
$v(0)=0$ and
\begin{align}\label{p2mu}
\nabla v=\begin{cases} (Se_1|Re_2) & \text{in $E_{\rm sw}$,}\\
 (F_{-}(R,S,\mu)e_1|G_{+}(R,S,\mu)e_2) & \text{in $E_{\rm ne}$}
\end{cases}
\end{align} 
yields a piecewise affine function with $v(\partial E) = u(\partial E)$. 
\end{proof}

The next result specializes the previous two lemmas to the case of orientation and locally volume-preserving maps.

\begin{corollary}[Decomposition on a single soft tile]\label{cor:preservation}
Let $E$ and $u\in W^{1,p}(E;\R^2)$ be as in Lemma~\ref{lem:affine}.

a) If $\det \nabla u>0$ a.e.~in $E$, then there exist $R, S\in \SO(2)$ with $\det(Se_1|Re_2) = Se_1\cdot Re_1>0$, $b\in \R^2$ and $\varphi\in W_0^{1,p}(E;\R^2)$ such that
\begin{align}\label{two_rot}
	u(x) = (Se_1|Re_2)x+ b + \varphi(x) \quad \text{for a.e.~$x\in E$.}
\end{align} 

b) If additionally, $\det \nabla u=1$ a.e.~in $E$, then there exists $R\in \SO(2)$,
such that 
\begin{align*}
	u(x) = Rx+ b + \varphi(x) \quad \text{for a.e.~$x\in E$.}
\end{align*} 
\end{corollary}
\begin{proof}
The statement a) follows from the observation that the situations $iii)$ and $iv)$ in the cases $1$ and $2$ in the proof of Lemma~\ref{lem:affine} can be ruled out since $u$ is orientation preserving. 
Indeed, assume to the contrary that there exists $p$ as constructed in~\eqref{p1} or~\eqref{p1mu}. In these cases, the identity~\eqref{splitting} and the Null-Lagrangian property of the determinant yield the contradictions
\begin{align*}
	0<\int_{E} \det \nabla u\dd x= \int_E \det (\nabla u-\nabla \varphi)\dd x= \frac{1}{2}|E|\det(Se_1|Re_2) + \frac{1}{2}|E|\det(Re_2|Se_1)=0
\end{align*}
or
\begin{align*}
	 0<\int_{E} \det \nabla u\dd x &= \int_E \det (\nabla u-\nabla \varphi)\dd x \\
		&= \frac{1}{2}|E|\det(Se_1|Re_2) + \frac{1}{2}|E|\det(F_{\pm}(R,S,\mu)e_1|G_{\mp}(R,S,\mu)e_2) \\
		 &= \frac{1}{2}|E|\det(Se_1|Re_2) - \frac{1}{2}|E|\det(Se_1|Re_2) =0\,,
\end{align*}
where the second to last equality follows by \eqref{Fpm2}--\eqref{Gpm2}, \eqref{Fpm1}--\eqref{Gpm1} and basic algebraic properties of the determinant.
The cases~\eqref{p2} and~\eqref{p2mu} can be handled analogously.
 
The only remaining possible boundary values of $u$ are described by the situations $i)$ and $ii)$ in the proof of Lemma \ref{lem:affine}. 
Hence, there exist two rotations $R,S\in \SO(2)$ and $\ffi\in W^{1,p}_0(E;\R^2)$, such that \eqref{two_rot} is satisfied.
Note that $\det (Se_1|Re_2) = Se_1\cdot Re_1$ and distinguish three cases: If $\det(Se_1|Re_2) >0$, then there is nothing to prove; otherwise $\det(Se_1|Re_2)\leq 0$ (equality corresponds to the case $ii)$) and it holds that
\begin{align}\label{det_scalar_product}
	0<\int_E \det(\nabla u)\dd x = \int_E \det (\nabla u-\nabla \varphi)\dd x = |E|\det (Se_1|Re_2)  \leq 0,
\end{align}
which produces a contradiction.
 
\medskip
 
b) In case $\det \nabla u=1$ a.e.~in $E$, then we obtain analogously to \eqref{det_scalar_product} the identity
\begin{align*}
	|E|=\int_E \det(\nabla u)\dd x =|E|Re_1\cdot Se_1,
\end{align*}
from which we conclude that $Re_1$ is identical to $Se_1$. The desired equality then follows from \eqref{two_rot}.
\end{proof}

\subsection{Macroscopic deformation behavior}\label{sec:rigid_deform}
In this section, we focus on maps $u_\eps\in\Acal$ with $\nabla u_\eps\in \SO(2)$ on $\Omega\cap \eYstiff$.
We prove via Corollary \ref{cor:preservation} that $\nabla u_\eps$ can essentially only attain two different values $S_\eps,R_\eps\in\SO(2)$ with $S_\eps e_1\cdot R_\eps e_1>0$ on $\Omega\cap \eYstiff$, which suggests an affine limit with gradient in the set $K$ as in \eqref{Klambda}.
The next proposition proves this statement (on any compactly contained subset) and serves as the compactness result for the homogenization in Theorem \ref{theo:homogenization_rigid} later in this section.

\begin{proposition}[Characterization of limit deformations]\label{prop:K_lambda}
Let $p \geq 2$.

a) If a sequence $(u_\eps)_\eps\subset \Acal$ (recall \eqref{Acal}) satisfies
\begin{align}\label{diffinclusion}
	\nabla u_\eps \in \SO(2)\  \text{a.e.~in $\Omega\cap\eYstiff$}
\end{align} 
and $u_\eps\weakly u$ in $W^{1,p}(\Omega;\R^2)$, then $u$ is affine with
\begin{align*}
	\nabla u = F \in  K:=\{\lambda S + (1-\lambda)R : R,S\in \SO(2), Re_1\cdot Se_1 \geq 0\}.
\end{align*}

b) For every affine function $u:\Omega\to \R^2$ with $\nabla u \in K$ there exists a sequence of piecewise affine functions $(u_\eps)_\eps\subset \Acal$ satisfying~\eqref{diffinclusion} and $\int_\Omega u_\eps \dd{x} = \int_\Omega u\dd{x}$ such that $u_\eps\weakly u$ in $W^{1,p}(\Omega;\R^2)$.
\end{proposition}

\begin{proof} 
a) \textit{Step 1: Local-global rigidity effects.} 
First, we prove that the set of rotation matrices which emerge from Reshetnyak's rigidity theorem on all the connected components of $\Omega\cap\eYstiff$ has at most two different elements.

Let $\Omega'\Subset\Omega$ and set
\begin{align}\label{J_eps}
	J_\eps'=\{k\in \Z^2: \Omega'\cap\eps(k+Y)\neq \emptyset\}.
\end{align}
Then, it holds that $\Omega'\subset\bigcup_{k\in J_\eps'}\eps(k+Y)\subset\Omega$ for $\eps$ sufficiently small.
By Reshetnyak's rigidity theorem, we conclude that for each $k\in J_\eps$ there exist rotations $S_\eps^k, R_\eps^k \in \SO(2)$ such that $\nabla u_\eps=S_\eps^k$ on $\eps(k+Y_1)$ and $\nabla u_\eps=R_\eps^k$ on $\eps(k+Y_{3})$, respectively. 
Applying Corollary~\ref{cor:preservation}\;a) on each rectangle $\eps(k+(0, \lambda)\times (\lambda, 1))$ and $\eps(k+(\lambda, 1)\times (0, \lambda))$ yields that $S_\eps^k = S_\eps^{l}$ and $R_\eps^k = R_\eps^{l}$ for all $k, l\in J_\eps'$. 
Thus, $\nabla u_\eps$ attains at most two different values, say $S_\eps\in \SO(2)$ on $\Omega'\cap\eps(k+Y_1)$ and $R_\eps\in \SO(2)$ on $\Omega'\cap\eps(k+ Y_3)$ with $R_\eps e_1 \cdot S_\eps e_1 >0$ for all $k\in \Z^2$.
\medskip

\textit{Step 2: Characterization of the weak limit.}
In light of Step 1 and \eqref{two_rot}, we can now write $u_\eps\restrict{\Omega'}$ in the form
\begin{align}\label{splitting1}
	u_\eps = v_\eps + \varphi_\eps\quad \text{on $\Omega'$,}
\end{align} 
where $v_\eps:\R^2\to\R^2$ is a $\eps Y$-periodic continuous and piecewise affine function with gradients
\begin{align}\label{nablav_eps}
	\nabla v_\eps =	 \begin{cases}
						S_\eps & \text{on $\eps Y_1$,}\\
						R_\eps &  \text{on $\eps Y_3$,}\\
						(S_\eps e_1 | R_\eps e_2) &  \text{on $\eps Y_2$,}\\
						(R_\eps e_1 | S_\eps e_2) & \text{on $\eps Y_4$,}
					\end{cases}
\end{align}
and $\varphi_\eps\in W^{1, p}(\Omega';\R^2)$ with $\varphi_\eps = 0$ on $\Omega'\cap\eYstiff$.

In the following, we show that
\begin{align}\label{weaklimit_varphi}
	\nabla \varphi_\eps \weakly 0 \quad \text{in $L^p(\Omega';\R^{2\times 2})$.}
\end{align}
We first observe that $(\nabla \ffi_\eps)_\eps$ is bounded in $L^p(\Omega';\R^{2\times 2})$ since $(u_\eps)_\eps$ is weakly convergent in $W^{1,p}(\Omega;\R^2)$ and $|\nabla v_\eps| = \sqrt{2}$ for every $\eps>0$ and a.e.~on $\Omega'$. Since piecewise constant functions on a grid are dense in $L^q(\Omega';\R^{2\times 2})$ with $\frac{1}{p} + \frac{1}{q} =1$, it suffices to test the weak convergence with characteristic functions of squares.
Hence, we set for an arbitrary open square $Q\subset \Omega'$ the set $Q_\eps:=\bigcup_{k\in I_\eps^{\partial Q}} \eps(k + \Ysoft)$ with $I_\eps^{\partial Q} = \{k\in \Z^2: \partial Q\cap \eps(k+\Ysoft) \neq \emptyset \}$, and obtain with the help of the Gauss-Green theorem and H\"older's inequality that
\begin{align*}
	\Bigl|\int_Q \nabla \varphi_\eps \dd{x}\Bigr|\leq \Bigl|\int_{Q_\eps \cap Q} \nabla \varphi_\eps\dd{x}\Bigr| 
		\leq  \|\nabla \ffi_\eps\|_{L^p(\Omega')} |Q_\eps|^{1-\frac{1}{p}} \leq C\|\nabla \ffi_\eps\|_{L^p(\Omega')} \eps^{1-\frac{1}{p}}.
\end{align*}
In the last line, we used the fact that $\# I_\eps^{\partial Q} \leq C \frac{1}{\eps}$ for a constant $C>0$ independent of $\eps$, and $|\eps(k+\Ysoft)| = 2\lambda(1-\lambda)\eps^2$. This proves the desired convergence \eqref{weaklimit_varphi}.

We now address the the weak convergence of $(\nabla v_\eps)_\eps$. First, we find $R,S\in \SO(2)$ and (non-relabeled) subsequences of $(S_\eps)_\eps$ and $(R_\eps)_\eps$ such that $S_\eps \to S$ and $R_\eps \to R$ as $\eps\to 0$; the limits then satisfy $Re_1 \cdot Se_1\geq 0$ since $R_\eps e_1 \cdot S_\eps e_1 >0$ for all $\eps$.
We now show that
\begin{align}\label{weaklimit_v}
	\nabla v_\eps\weakly \lambda S + (1-\lambda)R \qquad \text{in $L^p(\Omega;\R^{2\times 2})$.}
\end{align}
To see this, we define for $\eps$ the auxiliary functions $w_\eps(x) = \eps w(\frac{x}{\eps})$ for $x\in \R^2$, where $w:\R^2\to \R^2$ is continuous and piecewise affine with the $Y$-periodic arrangement of gradients
\begin{align}\label{nablaw}
	\nabla w  = 	\begin{cases}
						S & \text{on $Y_1$,}\\
						R &  \text{on $Y_3$,}\\
						(Se_1 | R e_2) &  \text{on $Y_2$,}\\
						(Re_1 | S e_2) & \text{on $Y_4$.}
					\end{cases}
\end{align} 
It follows with the help of the Riemann-Lebesgue lemma that 
\begin{align}\label{weaklimit_w}
	\nabla w_\eps\weakly \int_Y \nabla w \dd{x} = \lambda^2S +(1-\lambda)^2R + \lambda(1-\lambda)(S+R) = \lambda S + (1-\lambda)R \quad\text{in }L^p(\Omega;\R^{2\times 2}).
\end{align}
Moreover, it holds that
\begin{align}\label{v_eps-w_eps}
	\|\nabla v_\eps -\nabla w_\eps\|_{L^p(\Omega;\R^{2\times 2})} \leq C(|R_\eps-R| + |S_\eps - S|)
\end{align}
for a constant $C>0$ independent of $\epsilon$ and $\Omega'$. Combining \eqref{weaklimit_w} with \eqref{v_eps-w_eps} then produces \eqref{weaklimit_v}.

In view of \eqref{splitting1}, we then obtain that
\begin{align}\label{affine_limit}
	\nabla u=  \lambda S + (1-\lambda)R \quad \text{a.e.~on $\Omega'$,}
\end{align}
by the uniqueness of weak limits. 
The arbitrariness of $\Omega'\Subset\Omega$ implies that~\eqref{affine_limit} is true on all of $\Omega$.

\medskip 

b) For the proof of the approximation result we use an explicit construction of continuous and piecewise affine functions, ensuring first the orientation preservation and \eqref{diffinclusion}.
Let $u:\Omega\to \R^2$ be affine such that
\begin{align*}
	\nabla u = \lambda S+(1-\lambda)R
\end{align*} 
with $S, R\in \SO(2)$ satisfying $Re_1 \cdot Se_1\geq 0$. 
If the latter is an equality, then we choose a sequence $(\hat{S}_\eps)_\eps\subset \SO(2)$ such that $Re_1\cdot \hat{S}_\eps e_1 >0$ and $\hat{S}_\eps \to S$ as $\eps\to 0$.
We then define the continuous and piecewise affine approximating sequence $(u_\eps)_\eps$ as
\begin{align}\label{sequence1}
	u_\eps(x) =	v_\eps(x) + \dashint_\Omega u(y) - v_\eps(y)\dd{y},\quad x\in\R^2
\end{align}
where $v_\eps:\R^2\to\R^2$ is chosen as in \eqref{nablav_eps} with
\begin{align}\label{rot_replacement}
	\begin{cases}
		R_\eps=R\text{ and } S_\eps = \hat{S}_\eps\quad &\text{if } Re_1\cdot S_\eps e_1 =0\\
		R_\eps=R\text{ and } S_\eps = S\quad &\text{if } Re_1\cdot S_\eps e_1 >0
	\end{cases}
\end{align}
By design, the sequence $(u_\eps)_\eps$ has the same mean value as $u$, satisfies \eqref{diffinclusion}, and converges to $u$ in $W^{1,p}(\Omega;\R^2)$ due to \eqref{weaklimit_v}.
It remains to prove that this sequence also satisfies the Ciarlet-Ne\v{c}as condition on $\Omega$ so that $(u_\eps)_\eps\subset \Acal$, cf.~\eqref{Acal}.
Since each $u_\eps$ fulfills $\det \nabla u_\eps>0$ a.e.~in $\Omega$ this task is equivalent to establishing the injectivity of $u_\eps$, see e.g.,~\cite[Proposition 4.2]{GiP08}.
As $u_\eps$ and $v_\eps$ differ differ only by a global translation, it suffices to show that $v_\eps$ is injective. 
This can be seen directly by considering the explicit construction
\begin{align*}
	v_\eps(\eps k + x) = d_\eps + \eps(\lambda S_\eps + (1-\lambda)R_\eps)k +
		\begin{cases}
			S_\eps x &\text{ if } x\in \eps Y_1,\\
			R_\eps x + \eps\lambda (S_\eps-R_\eps)(e_1+e_2) &\text{ if } x\in \eps Y_3,\\
			(S_\eps e_1|R_\eps e_2) x + \eps\lambda (S_\eps-R_\eps)e_2 &\text{ if } x\in \eps Y_2,\\
			(R_\eps e_1|S_\eps e_2) x + \eps\lambda (S_\eps-R_\eps)e_1 &\text{ if } x\in \eps Y_4,\\
		\end{cases}
\end{align*}
for a suitable global translation $d_\eps\in\R^2$, and that $\lambda S_\eps +(1-\lambda) R_\eps$, $(S_\eps e_1|R_\eps e_2)$, $(R_\eps e_1|S_\eps e_2)$ have positive determinants, since $S_\eps e_1 \cdot R_\eps e_1 >0$ and
\begin{align*}
	\det(\lambda S_\eps + (1-\lambda) R_\eps) = \lambda^2 + (1-\lambda)^2 + 2\lambda(1-\lambda) S_\eps e_1 \cdot R_\eps e_1 = |\Ystiff| + |\Ysoft|S_\eps e_1 \cdot R_\eps e_1 > |\Ystiff|.
\end{align*}
\end{proof}

\begin{remark}[Discussion of \boldmath{$K$}]\label{rem:Klambda}
	
	a) In Proposition \ref{prop:K_lambda}, we established that weak limits of sequences in $\Acal$ that satisfy \eqref{diffinclusion} are characterized by affine functions with gradient in 
	$K$ defined as in \eqref{Klambda}. 
	In the following, we shall prove the second identity in this equation.
	We first observe that 
	\begin{align*}
		\SO(2) \subset K \subset \lambda \SO(2) + (1-\lambda)\SO(2) & \subset \bigcup_{\mu \in [0,1]} \mu \SO(2) + (1-\mu) \SO(2)\\
		& =  \{F\in \R^{2\times 2}: |Fe_1|\leq 1, Fe_2=(Fe_1)^\perp\} = \SO(2)^{\rm c},
	\end{align*}
	which shows that every $F\in K$ is a conformal contraction.
	Furthermore, the set can be simplified to
	\begin{align*}
		K = \{\alpha Q: |\Ystiff|\leq \alpha^2\leq 1, Q\in \SO(2)\},
	\end{align*}
	since for every $R,S\in \SO(2)$ with $Se_1\cdot Re_1\geq 0$ it holds that
	\begin{align}\label{det_Klamba}
		\det(\lambda S + (1-\lambda) R) = |\Ystiff| + |\Ysoft|Se_1 \cdot Re_1 \geq |\Ystiff|.
	\end{align}	
	
	\smallskip
	
	b) The Poisson's ratio $\nu$ corresponding to every non-trivial affine deformation with gradient $\alpha Q$ for $\sqrt{|\Ystiff|}\leq \alpha< 1$ and $Q\in\SO(2)$ satisfies
	\begin{align*}
		\nu = - \frac{\alpha - 1}{\alpha - 1} = -1.
	\end{align*}
	This is a confirmation of the calculations in \cite{GrE00, KoV13} via a variational perspective.
\end{remark}

\subsection{Homogenization}\label{sec:rigid_hom2}
Now that the set of admissible limit deformations in the fully rigid setting is characterized, we are in the position to prove a corresponding $\Gamma$-convergence result.
Here, we consider energy functionals of integral type with integrand $W_\eps$ as in \eqref{density} with $\Wsoft$ as in \eqref{Wsoft} and $\eWstiff$ replaced by \eqref{Wrig2}.
Note that in this scenario $W_\eps$ does, in fact, not depend on $\eps$, which is why we write $W$ instead of $W_\eps$ throughout this section.

\begin{theorem}[Homogenization of rigid checkerboard structures]\label{theo:homogenization_rigid}
Let $\Omega\subset \R^2$ be a bounded Lipschitz domain, $p \geq 2$, and let $\Ical_\eps$ for $\eps>0$ as in~\eqref{energy}, \eqref{Acal}, and \eqref{density} with $\eWstiff$ replaced by \eqref{Wrig2} and $\Wsoft$ as in \eqref{Wsoft} such that $\Wsoft^\qc$ is polyconvex. 
Then, the family of functionals $(\Ical_\eps)_\eps$ $\Gamma$-converges for $\eps\to 0$ with respect to the strong $L^p(\Omega;\R^2)$-topology to $\Ical_{\rm hom}$ as in \eqref{Gamma_limit}-\eqref{Whom}.

Moreover, any sequence $(u_\eps)_\eps\subset L^p_0(\Omega;\R^2)$ with $\sup_{\eps}\Ical_\eps(u_\eps)<\infty$ has a subsequence that converges weakly in $W^{1,p}(\Omega;\R^2)$ to some affine function $u:\Omega\to \R^2$ with vanishing mean value and $\nabla u\in K$, cf.~\eqref{Klambda}.
\end{theorem}
\begin{proof}
\textit{Step 1: The lower bound.}
Let $(u_\eps)_\eps\subset L^p_0(\Omega;\R^2)$ be strongly convergent with limit $u\in L^p_0(\Omega;\R^2)$ and
\begin{align*}
	\lim_{\eps\to 0}\Ical_\eps(u_\eps) = \liminf_{\eps\to 0}\Ical_\eps(u_\eps)<\infty.
\end{align*} 
In particular, it holds that $(u_\eps)_\eps\subset \Acal$, the sequence satisfies \eqref{diffinclusion}, and has a (non-relabeled) subsequence with $u_\eps \weakly u$ in $W^{1,p}(\Omega;\R^2)$ for some $u\in W^{1,p}(\Omega;\R^2)$ due to \eqref{Wsoft} and the specific choice \eqref{Wrig2} for $\Wrig$.
In view of Proposition~\ref{prop:K_lambda}\,a), we find that $u$ is affine with $\nabla u=F\in K$.

To show the liminf-inequality, let $\Omega'\Subset \Omega$ be an arbitrary subset and let $\eps$ be sufficiently small. 
Exploiting the non-negativity of $\Wsoft$, the splitting~\eqref{splitting1} together with \eqref{nablav_eps}, and the fact that $W^\qc$ is $W^{1,p}$-quasiconvex as a polyconvex function (see \cite[Lemma 2.5]{MVGSN20}) then produce
\begin{align*}
	\int_{\Omega} W(\tfrac{x}{\eps},\nabla u_\eps) \dd{x} &\geq \sum_{k\in J_\eps'}\int_{\eps(k+Y_2)} \Wsoft\big((S_\eps e_1|R_\eps e_2) + \nabla \varphi_\eps\big) \dd x +  \int_{\eps(k+Y_4)} \Wsoft\big((R_\eps e_1|S_\eps e_2) + \nabla \varphi_\eps\big) \dd x\\ 
	&\geq \sum_{k\in J_\eps'}  \int_{\eps(k+Y_2)} \Wsoft^\qc\big((S_\eps e_1|R_\eps e_2) + \nabla \varphi_\eps\big)\dd x +  \int_{\eps(k+Y_4)} \Wsoft^\qc\big((R_\eps e_1|S_\eps e_2) + \nabla \varphi_\eps\big) \dd x \\
	&\geq \sum_{k\in J_\eps'}  \lambda(1-\lambda)\eps^2\big(\Wsoft^\qc(S_\eps e_1|R_\eps e_2) + \Wsoft^\qc(R_\eps e_1|S_\eps e_2)\big) \\
	&\geq \frac{1}{2}|\Ysoft| |\Omega'|\big(\Wsoft^\qc(S_\eps e_1|R_\eps e_2) + \Wsoft^\qc(R_\eps e_1|S_\eps e_2)\big),
\end{align*}
where $J_\eps'$ is taken as in \eqref{J_eps}; recall also that $\varphi_\eps\in W^{1,p}_0(\eps(k+Y_i);\R^2)$ for $i\in\{2,4\}$ and every $k\in J_\eps'$.

Now, let $S, R\in \SO(2)$ be the limits of $(S_\eps)_\eps$ and $(R_\eps)_\eps$ (up to a subsequence) as in the proof of Proposition \ref{prop:K_lambda} a), respectively.
Since $W^{\qc}$ is polyconvex and therefore lower semicontinuous by definition, we may pass to the limit $\epsilon\to0$ and obtain 
\begin{align*}
	\liminf_{\eps\to 0} \Ical_\eps(u_\eps)  \geq \frac{1}{2}|\Ysoft| |\Omega'| \bigl(\Wsoft^\qc(Se_1|Re_2) + \Wsoft^\qc(Re_1|Se_2)\bigr)\geq |\Omega'|\Whom(F).
\end{align*} 
Upon taking the supremum over all compactly contained $\Omega'\Subset\Omega$, we obtain the desired lower bound.

\medskip

\textit{Step 2: The upper bound.} The idea is to use the approximating sequence of Proposition~\ref{prop:K_lambda}\,b) and augment it with a suitable perturbation on the softer part to enforce optimal energy. Preserving orientation during this construction requires a subtle construction due to Conti \& Dolzmann~\cite{CoD15}.  

To be precise, let $u$ be affine with $\nabla u=F\in K$ and choose the energetically optimal $R,S\in\SO(2)$ with $Re_1\cdot Se_1\geq 0$ such that $F=\lambda S+(1-\lambda)R$, and for $\eps>0$ let $u_\eps:\R^2\to\R^2$ as in~\eqref{sequence1}, see also \eqref{rot_replacement} and \eqref{nablav_eps}.
For any $\eps$ and $k\in\R^2$, let $(\hat{u}^k_{\eps,j})_j\subset  W^{1,p}(\eps (k + Y_2);\R^2)$ be the orientation preserving sequences as in \cite[Theorem~2.1]{CoD15} such that 
\begin{align*}
	\hat{u}^k_{\eps,j}\weakly u_\eps\quad &\text{ in $W^{1,p}(\eps(k+Y_2);\R^2)$ as $j\to \infty$,}\qquad\text{and}\qquad \hat{u}^k_{\eps,j}=u_\eps\quad \text{ on }\partial(\eps k+\eps Y_2),
\end{align*}
as well as
\begin{align}\label{energy_conv}
	\limsup_{j\to \infty} \int_{\eps(k+Y_2)} \Wsoft(\nabla \hat{u}^k_{\eps,j})\dd{x}\leq \int_{\eps(k+Y_2)}\Wsoft^\qc(\nabla u_\eps) \dd{x}.
\end{align}
Analogously, we introduce $(\check{u}^k_{\eps,j})_j\subset W^{1,p}(\eps k +\eps Y_4;\R^2)$.

Let $\widetilde{\Omega}\subset\R^2$ an open set with $\Omega\Subset \widetilde{\Omega}$, and let $\tilde{J}_\eps=\{k\in\R^2 : \eps(k+Y)\subset \widetilde{\Omega}\}$. 
For sufficiently small $\eps>0$, it then holds that
\begin{align}\label{cover_Omega}
	\Omega\subset\bigcup_{k\in \tilde{J}_\eps} \eps(k+Y)\subset\widetilde{\Omega},
\end{align}
and we define for $j\in \N$ the functions
\begin{align*}
	u_{\eps,j} = \sum_{k\in \tilde{J}_\eps} \hat{u}^k_{\eps,j}\mathbbm{1}_{\eps(k+Y_2)} + \check{u}^k_{\eps,j}\mathbbm{1}_{\eps(k+Y_4)} + u_\eps\mathbbm{1}_{\eps(k+Y_1\cup Y_3)} \quad\text{on $\widetilde{\Omega}$.}
\end{align*}
Each $u_{\eps,j}$ is, by design, orientation preserving, and $u_{\eps,j}\weakly u_\eps$ in $W^{1,p}(\Omega;\R^2)$. 
Moreover, every $u_{\eps,j}$ satisfies the Ciarlet-Ne\v{c}as condition \eqref{ciarlet_necas} on every subset of $\R^2$ since $u_\eps\in\Acal$ and the perturbations $\hat{u}_{\eps,j}^k, \check{u}^k_{\eps,j}$ have a positive determinant and coincide with $u_\eps$ on the boundary of the soft parts. In light of \cite[Theorem 1]{Bal81}, the functions $u_{\eps,j}$ are globally injective and thus satisfy \eqref{ciarlet_necas} on every subset of $\R^2$, cf.~\cite[Proposition 4.2]{GiP08}.

Now, combining \eqref{cover_Omega} with the non-negativity of $\Wsoft$, $\Wrig=0$ on $\SO(2)$, with \eqref{energy_conv} produces the energy estimate
\begin{align*}
	\limsup_{j\to \infty}\int_{\Omega} &W(\nabla u_{\eps,j})\dd{x} \leq \limsup_{j\to \infty}\sum_{k\in \tilde{J}_\eps} \int_{\eps(k+\Ysoft)} \Wsoft(\nabla u_{\eps,j}) \dd{x}=\sum_{k\in \tilde{J}_\eps} \limsup_{j\to \infty}\int_{\eps(k+\Ysoft)} \Wsoft(\nabla u_{\eps,j}) \dd{x}\\
	&\leq\sum_{k\in \tilde{J}_\eps} \int_{\eps(k+\Ysoft)} \Wsoft^\qc(\nabla u_\eps)\dd{x} 
	= \lambda(1-\lambda)\sum_{k\in \tilde{J}_\eps} \eps^2 \bigl(\Wsoft^\qc(S_\eps e_1| R_\eps e_2) + \Wsoft^\qc(R_\eps e_1|S_\eps e_2)\bigr)
\end{align*}
and hence
\begin{align*}
	\limsup_{\eps\to 0}\limsup_{j\to \infty}\int_{\Omega} &W(\nabla u_{\eps,j})\dd{x} \leq \lambda(1-\lambda)|\widetilde{\Omega}| \Whom(F).
\end{align*}
Finally, we exploit that this estimate holds for arbitrary $\widetilde{\Omega}\Supset\Omega$, and we use a diagonalization argument in the sense of Attouch \cite{Att84} to select a diagonal sequence $(\bar{u}_\eps)_\eps$ with $\bar{u}_\eps=u_{\eps,j(\eps)}$ such that
\begin{align*}
	\limsup_{\eps\to 0}\int_{\Omega} &W(\nabla \bar{u}_{\eps})\dd{x} \leq \lambda(1-\lambda)|\Omega| \Whom(\nabla u)
\end{align*}
and $\bar{u}_\eps\weakly u$ in $W^{1,p}(\Omega;\R^2)$. 
Note that the uniform bounds (with respect to the index parameters) of $u_{\eps,j}$ are obtained via the coercivity of $\Wsoft$ as in \eqref{Wsoft} and the triviality of $\Wrig$ defined in \eqref{Wrig2}.
\end{proof}

\begin{remark}[Properties of \boldmath{$\Whom$}]\label{rem:limit_density}
a) The representation of $F\in K$ into $F=\lambda S + (1-\lambda)R$ for $R,S\in\SO(2)$ with $Re_1\cdot Se_1\geq 0$ is not unique.
A direct calculation based on the intersection of two circles with radii $\lambda$ and $1-\lambda$ shows that
\begin{align*}
	Se_1 = \frac{1}{2\lambda|Fe_1|^2}\big((|Fe_1|^2 + 2\lambda-1)Fe_1 \pm \sqrt{4\lambda^2|Fe_1|^2 - (|Fe_1|^2 + 2\lambda - 1)^2}Fe_2\big)
\end{align*}
and $Re_1=\frac{1}{1-\lambda}(Fe_1 - \lambda Se_1)$.

In fact, there exist exactly two choices for $R$ and $S$ if $|Fe_1|<1$, and the representation is unique if $|Fe_1|=1$.
For $\lambda=\frac{1}{2}$, this formula reduces to
\begin{align}\label{SandR}
	Se_1 = Fe_1 \pm \frac{\sqrt{1-|Fe_1|^2}}{|Fe_1|} Fe_2\qand Re_1 = Fe_1 \mp \frac{\sqrt{1-|Fe_1|^2}}{|Fe_1|}Fe_2.
\end{align}
\medskip

b) Note that if $\Wsoft$ is frame-indifferent or isotropic, i.e., $\Wsoft(QF)=\Wsoft(F)$ or $\Wsoft(FQ)=\Wsoft(F)$ for all $F\in \R^{2\times 2}$ and $Q\in \SO(2)$, then it is immediate that the quasiconvex envelope $\Wsoft^\qc$ (cf.~\eqref{qc_envelope}) is frame-indifferent or isotropic as well. 

In case $\Wsoft$ has one of these two properties then the limit density simplifies to
\begin{align*}
	\Whom(F)=\Whom( |Fe_1|\Id)
\end{align*} 
for $F\in K$. Moreover, if $\Wsoft$ is both frame-indifferent and isotropic, then 
\begin{align*}
	\Whom(F) = |\Ysoft|\min_{R, S\in \SO(2), \lambda S+(1-\lambda)R=|Fe_1|\Id, Re_1\cdot Se_1\geq 0}\Wsoft^\qc(Se_1|Re_2),
\end{align*}
since $(Re_1|Se_2) = R_{-\frac{\pi}{2}}(Se_1|Re_2)R_{\frac{\pi}{2}}$, where $R_\theta\in\SO(2)$ describes a rotation matrix by the angle $\theta\in\R$. 
The expression on the right-hand side reduces even further in the case $\lambda=\frac{1}{2}$, where we obtain the explicit formula
\begin{align*}
	\Whom(F) = |\Ysoft|\Wsoft^\qc\big(\big(|Fe_1|+ \sqrt{1-|Fe_1|^2}\big)\Id\big)
\end{align*}
with the help of \eqref{SandR}.
\end{remark}

\begin{remark}[Comparison with cell formula]\label{rem:cell_formula} 
Homogenization for integral-type functionals commonly gives rise to homogenized integrands that are defined by a (multi-)cell formula \cite{Bra85,Mul87}. 
In this remark, we explicitly compute the (multi-)cell formula corresponding to $W$ and compare the result with $\Whom$ as in \eqref{Whom}.

In the following, we consider the density
\begin{align}\label{cell}
	\Wcell(F) = \inf_{\psi\in W^{1, p}_{\#}(Y;\R^2)} \int_Y W(y, F+\nabla \psi) \dd{y} \quad \text{for $F\in \R^{2\times 2}$}
\end{align} 
taken from \cite[Equation (1.7)]{Mul87}, and prove that
\begin{align}\label{cell2}
	\Wcell =\begin{cases} \Whom & \text{on $K\setminus \sqrt{|\Ystiff|}\SO(2)$,} \\ \infty & \text{otherwise.}\end{cases} 
\end{align}
We shall point out that Theorem \ref{theo:homogenization_rigid} also holds if the Ciarlet-Ne\v{c}as condition is dropped, see Remark \ref{rem:other_constraints} a) later on. 
The identity \eqref{cell2} shows, in particular, that the two densities $\Wcell$ and $\Whom$ coincide on $K\setminus \sqrt{|\Ystiff|}\SO(2)$, but differ on $\sqrt{|\Ystiff|}\SO(2)$.
This observation stands in contrast to other homogenization results in the context of asymptotic rigidity, see \cite[Section 6]{ChK17} and \cite[Remark 5.5]{ChK20}, where the homogenized density and the cell formula coincide everywhere.

To prove \eqref{cell2}, let $F\in\R^{2\times 2}$ such that $\Wcell(F)<\infty$, which implies that there exist $\psi\in W^{1,p}_{\#}(Y;\R^2)$ and $S, R \in \SO(2)$ such that 
\begin{align*}
	F+\nabla \psi=S\quad\text{ on $Y_1$}\quad\text{and}\quad F+\nabla \psi=R\quad \text{ on $Y_3$.}
\end{align*}
By exploiting the periodicity of the boundary values of $\psi$, we can apply Corollary~\ref{cor:preservation} a) to $u(x)=Fx+\psi(x)$ for $x\in Y_2$, which produces
\begin{align*}
	F+\nabla \psi = (Se_1|Re_2) + \nabla \ffi_2\quad\text{on } Y_2
\end{align*} 
with $\ffi_2\in W_0^{1,p}(Y_2;\R^2)$, and $Se_1\cdot Re_1>0$. 
Similarly, we find $\varphi_4\in W_0^{1,p}(Y_4;\R^2)$ such that $F+\nabla \psi = (Re_1| Se_2) + \nabla \ffi_4$ on $Y_4$.

By choosing $\hat \psi =\psi-\ffi_2-\ffi_4\in W^{1,p}_{\#}(\Omega;\R^2)$, we obtain that 
\begin{align*}
	F+\nabla \hat \psi = \begin{cases}
								S & \text{on $Y_1$,}\\
								R &  \text{on $Y_3$,}\\
								(Se_1|Re_2) &  \text{on $Y_2$,}\\
								(Re_1|Se_2) & \text{on $Y_4$.}
							\end{cases}
\end{align*}
Then the periodicity of $\hat{\psi}$ yields that
\begin{align*}
	\lambda Fe_1 = \int_{Y_1\cup Y_4} Fe_1 + \partial_1\hat \psi \dd{x} = \lambda^2 Se_1 + \lambda(1-\lambda) Re_1,
\end{align*}
and hence, $Fe_1=\lambda Se_1 + (1-\lambda) Re_1$. Similarly, one can show that $Fe_2=\lambda Se_2 + (1-\lambda) Re_2$, which implies that $F\in K \setminus \sqrt{|\Ystiff|}\SO(2)$ since $Se_1\cdot Re_1>0$.

It remains to compare the values of the two functions in \eqref{cell} and \eqref{cell2} for $F\in K\setminus \sqrt{|\Ystiff|}\SO(2)$. 
Indeed, let $F=\lambda S + (1-\lambda)R$ for $S, R\in \SO(2)$ with $Se_1\cdot Re_1 > 0$ and $\Whom(F)=\Wsoft^\qc(Se_1|Re_2) + \Wsoft^\qc(Re_1|Se_2)$, then the previous calculations show that
\begin{align*}
	\Wcell(F) & = \inf_{\psi\in W_{\#}^{1,p}(Y;\R^2)} \int_{Y_{\rm soft}} \Wsoft(F+\nabla \psi) \dd{x} \\ 
		& = |Y_2|\inf_{\varphi\in W_{0}^{1,p}(Y_2;\R^2)} \dashint_{Y_2} \Wsoft\bigl((Se_1|Re_2) + \nabla \ffi\bigr) \dd{x}  \\
		& \qquad \qquad \qquad + |Y_4|  \inf_{\varphi\in W_{0}^{1,p}(Y_4;\R^2)} \dashint_{Y_4} \Wsoft\bigl((Re_1|Se_2) + \nabla \ffi\bigr) \dd{x}\\ 
		& =\frac{1}{2}|\Ysoft|\bigl( \Wsoft^\qc(Se_1|Re_2) + \Wsoft^\qc(Re_1|Se_2)\bigr) = \Whom(F).
\end{align*}
This concludes the proof of \eqref{cell2}.
The results presented above dot not change if $\Wcell$ is replaced by the multi-cell formula
\begin{align*}
	W_{\rm multi-cell}(F) := \inf_{k\in \N} \inf_{\psi\in W_{\#}^{1,p}(kY;\R^2)} \dashint_{kY} W_\eps(F+\nabla \psi) \dd{x},\quad F\in\R^{2\times 2},
\end{align*}
cf.~\cite[Equation (2.7)]{Mul87}.
\end{remark}

\subsection{Discussion of the assumptions}\label{sec:rigid_assumptions}
In this chapter, we present a critical discussion of the necessity of several model assumptions made in Section \ref{sec:rigid}.
First, we address the set of admissible functions $\Acal$, cf.~\eqref{Acal}, which consists of all Sobolev functions satisfying the Ciarlet-Ne\v cas condition \eqref{ciarlet_necas} and orientation preservation. While the macroscopic deformation behavior stays intact when dropping either of the two assumptions, see Remark \ref{rem:other_constraints} a) and b), the material can undergo infinite compression if both conditions are dropped, see Proposition \ref{prop:without_orientation}.
Second, we prove that the elastic material becomes much more flexible in the case $p<2$ due to the occurrence of microfractures at the hinges. 
This section is then concluded with two final remarks about the geometric setup of the model:  the porous case and the case of rigid rectangles.

\begin{remark}[Orientation preservation and Ciarlet-Ne\v cas]\label{rem:other_constraints}
	a) Theorem \ref{theo:homogenization_rigid} and Proposition \ref{prop:K_lambda} remain true if the Ciarlet-Ne\v{c}as condition \eqref{ciarlet_necas} on $\Omega$ in the definition of $\Acal$, see \eqref{Acal}, is dropped.
	Indeed, the compactness and lower bound do not require non-interpenetration of matter at all, while the recovery sequences sequence designed in Proposition \ref{prop:K_lambda} b) and in Step 2 of the proof of Theorem \ref{theo:homogenization_rigid} satisfy this constraint automatically.
	
	\medskip
	
	b) For $p>2$ we shall also point out that Proposition \ref{prop:K_lambda} is true if the orientation preservation $\det \nabla u>0$ a.e.~in $\Omega$ is dropped instead of the Ciarlet-Ne\v{c}as condition \eqref{ciarlet_rigid} on $\Omega$.
	In fact, we merely need to replace Corollary \ref{cor:preservation} by a variant that also considers the full neighboring stiff squares, see Proposition \ref{lem:luftschloss} later in Section \ref{sec:elastic}.
	The result essentially stays the same with the minor adjustment, that the rotations $S,R\in\SO(2)$ in Corollary \ref{cor:preservation} a) satisfy $Se_1\cdot Re_1 \geq 0$ instead of $Se_1\cdot Re_1 >0$.	
	
	\medskip
	
	c) Exchanging the constraint of orientation preservation $\det \nabla u>0$ a.e.~in $\Omega$ in the definitions of $\Acal$ and $\Wsoft$, cf.~\eqref{Acal} and \eqref{Wsoft}, by incompressibility, that is,
	\begin{align*}
		\det \nabla u =1 \ \text{a.e.~in $\Omega$},
	\end{align*}  
	results in a fully rigid limit set and a trivial $\Gamma$-limit. 
	Precisely, the energy sequence $(\Ical_\eps)_\eps$ $\Gamma$-converges with respect to the strong topology in $L^p_0(\Omega;\R^2)$ to
	\begin{align*}
		\Ical_{\rm hom}: L^p_0(\Omega;\R^2) \to [0,\infty],\ u\mapsto \begin{cases} 
						0 & \text{if $\nabla u=R\in\SO(2)$,}\\
						\infty & \text{otherwise.}
					\end{cases}
	\end{align*}
	This is a direct consequence of Corollary \ref{cor:preservation} b) and the proof of Proposition \ref{prop:K_lambda} a).
	In fact, for any sequence $(u_\eps)_\eps$ of bounded energy there is $R\in\SO(2)$ such that $\nabla u_\eps \weakly \lambda R+(1-\lambda)R =R$.
	As for the energetic adjustment of the approximating sequence in Proposition \ref{prop:K_lambda} b), we invoke \cite[Theorem 2.4]{CoD15} instead of Theorem \cite[Theorem 2.1]{CoD15}.
\end{remark}

As discussed in Remark \ref{rem:other_constraints} a) and b), the macroscopic deformation behavior (see Proposition \ref{prop:K_lambda}) still holds true if either the Ciarlet-Ne\v{c}as condition \eqref{ciarlet_necas} or the orientation preservation is dropped. 
We shall now discuss the scenario where all functions in $W^{1,p}(\Omega;\R^2)$ are admissible. In this setting, the set of admissible limit deformations can become larger, even allowing for infinite conformal compression, as the following result proves.

\begin{proposition}\label{prop:without_orientation}
Let $\lambda=\tfrac{1}{2}$ and $p\geq 2$.

a) If $(u_\eps)_\eps\subset W^{1,p}(\Omega;\R^2)$ converges weakly in $W^{1,p}(\Omega;\R^2)$ to some $u\in W^{1,p}(\Omega;\R^2)$, and satisfies the inhomogeneous constraint \eqref{diffinclusion}, then $u$ is affine with
\begin{align}\label{Klambda_without}
	\nabla u\in \lambda\SO(2) + (1-\lambda)\SO(2) = \{\alpha Q:  0=|\Ystiff| - |\Ysoft|\leq \alpha \leq 1,\, Q\in \SO(2)\}.
\end{align}

b) For every affine $u:\Omega\to \R^2$ with gradient in $\lambda\SO(2) + (1-\lambda)\SO(2)$ there exists a sequence of piecewise affine functions $(u_\eps)_\eps\subset  W^{1,p}(\Omega;\R^2)$ satisfying~\eqref{diffinclusion} and $\int_\Omega u_\eps \dd{x} = \int_\Omega u\dd{x}$ such that $u_\eps\weakly u$ in $W^{1,p}(\Omega;\R^2)$.

\end{proposition}

\begin{proof}
a) Let $\Omega'\Subset\Omega$ be arbitrary. For the characterization of limit deformations, one needs to understand the large scale compatibilities between the basic building blocks resulting from the proof of Lemma~\ref{lem:affine}. The latter can be classified in three classes: 
\begin{center}
	$\circled{1} \ S \to R \to S \to  R$, $\circled{2}\ R^\perp\to R \to S \to  -S^\perp$, $\circled{3}\ S\to R \to -R^\perp \to S^\perp$, 
\end{center} 
see Figure~\ref{fig3} (rotations listed in clockwise direction, starting at the bottom). 
\begin{figure}
	\centering
	\begin{tikzpicture}
		\begin{scope}[xshift = -6.5cm, yshift = 0cm, scale = 1]
			\draw[fill=yellow]  (0, 1) --  (0,0) -- (1,0)  -- (1,1) -- (0,1); 
			\draw[fill=ForestGreen!70!white]  (1, 1) --  (1,2) -- (2,2)  -- (2,1) -- (1,1); 
			\draw[fill=yellow]  (2, 1) --  (3,1) -- (3,0)  -- (2,0) -- (2,1); 
			\draw[fill=ForestGreen!70!white]  (1, 0) --  (2,0) -- (2,-1)  -- (1,-1) -- (1,0); 
			
			\draw (0.5,0.5) node {$R$}; 
			\draw (1.5,1.5) node {$S$};
			\draw (2.5,0.5) node {$R$}; 
			\draw (1.5,-0.5) node {$S$};
			
			\draw (0,2) node {$\circled{1}$};
		\end{scope}
		
		\begin{scope}[xshift = -2cm, yshift = 0cm, scale = 1]
			\draw[fill=yellow]  (0, 1) --  (0,0) -- (1,0)  -- (1,1) -- (0,1); 
			\draw[fill=ForestGreen!70!white]  (1, 1) --  (1,2) -- (2,2)  -- (2,1) -- (1,1); 
			\draw[fill=ForestGreen!40!white]  (2, 1) --  (3,1) -- (3,0)  -- (2,0) -- (2,1); 
			\draw[fill=yellow!70!orange]  (1, 0) --  (2,0) -- (2,-1)  -- (1,-1) -- (1,0); 
			
			\draw (0.5,0.5) node {$R$}; 
			\draw (1.5,1.5) node {$S$};
			\draw (2.5,0.5) node {$-S^\perp$}; 
			\draw (1.5,-0.5) node {$R^\perp$};
			
			\draw (0,2) node {$\circled{2}$};
		\end{scope}
		
		\begin{scope}[xshift = 2.5cm, yshift = 0cm, scale = 1]
			\draw[fill=yellow]  (0, 1) --  (0,0) -- (1,0)  -- (1,1) -- (0,1); 
			\draw[fill=yellow!40!orange]  (1, 1) --  (1,2) -- (2,2)  -- (2,1) -- (1,1); 
			\draw[fill=ForestGreen!100!white]  (2, 1) --  (3,1) -- (3,0)  -- (2,0) -- (2,1); 
			\draw[fill=ForestGreen!70!white]  (1, 0) --  (2,0) -- (2,-1)  -- (1,-1) -- (1,0); 
			
			\draw (0.5,0.5) node {$R$}; 
			\draw (1.5,1.5) node {$-R^\perp$};
			\draw (2.5,0.5) node {$S^\perp$}; 
			\draw (1.5,-0.5) node {$S$};
			\draw (0,2) node {$\circled{3}$};
		\end{scope}
	\end{tikzpicture}
	\caption{Possible configurations of rotation matrices arranged around a square for $\lambda=\frac{1}{2}$}\label{fig3}
\end{figure}
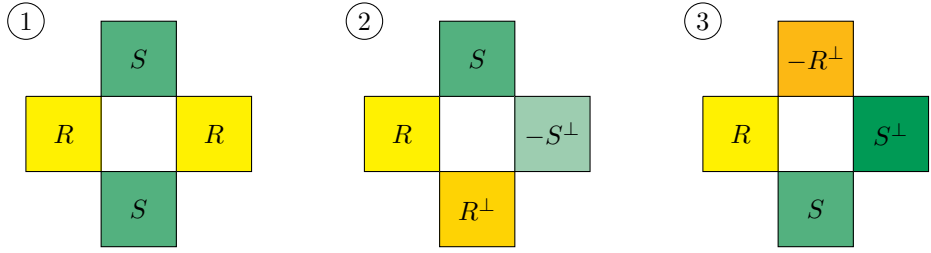

Hence, we have that for any $\eps(k+Y_2)$ and $\eps(k+Y_4)$ with $k\in J_\eps'$ (see \eqref{J_eps}), that $\nabla u_\eps$ restricted to $\eps(k+Y)$
fits in one of the three scenarios described above, where $[k+Y_2:i]$ means that the rotations on neighboring squares of $\eps(k+Y_2)$ behave like $\circled{i}$ for $i\in \{1,2,3\}$, and analogously for $[k+Y_4:i]$ on $\eps(k+Y_4)$. 
In the following, we use the notation $S_\eps^k, R_\eps^k$ for the rotation matrices satisfying $\nabla u_\eps= R_\eps^k$ on $\eps(k+Y_3)$ and $\nabla u_\eps = S_\eps^k$ on $\eps(k+Y_1)$ , respectively.

\begin{itemize}
	\item \textit{Class I:} $[k+Y_2:1]$ \& $[k+Y_4:1]$. In this case, $S_\eps^{k} = S_{\eps}^{k+e_1} = S_{\eps}^{k+e_2}$ and $R_\eps^{k} = R_\eps^{k-e_1} = R_\eps^{k-e_2}$.  
	\item \textit{Class II:} $[k+Y_2:1]$ \& $[k+Y_4:2]$, $[k+Y_2:2]$ \& $[k+Y_4:1]$, $[k+Y_2:3]$ \& $[k+Y_4:1]$, $[k+Y_2:1]$ \& $[k+Y_4:3]$, $[k+Y_2:2]$ \& $[k+Y_4:3]$, $[k+Y_2:3]$ \& $[k+Y_4:2]$.
	
	In the following, we provide a detailed explanation of the first of the above-mentioned cases. All other scenarios can be handled analogously.
	We assume that $[k+Y_2:1]$ \& $[k+Y_4:2]$, i.e.,
	\begin{align}\label{puzzle}
		R_\eps^{k-e_2} = (S_\eps^k)^\perp\qand S_\eps^{k+e_1} = -(R_\eps^k)^\perp,
	\end{align}		
	and perform a case study. 
	Suppose first that $[k-e_2+Y_2:1]$, then the first equation in \eqref{puzzle} produces
	\begin{align}\label{puzzle0}
		S_\eps^{k-e_2} = S_{\eps}^k = S_\eps^{k+e_2}\qand R_\eps^{k-(1,1)} = R_\eps^{k-e_2} = (S_\eps^{k})^\perp.
	\end{align}
	If $[k-e_1+Y_4:1]$, we derive from \eqref{puzzle} and \eqref{puzzle0} the equations
	\begin{align}\label{puzzle1}
		(S_\eps^{k})^\perp = R_\eps^k = R_\eps^{k-e_1} = R_\eps^{k-(1,1)} = R_\eps^{k-e_2} \qand (-R_\eps^k)^\perp = S_\eps^{k} = S_\eps^{k+e_2} = S_\eps^{k-e_2}.
	\end{align}
	In the case $[k-e_1+Y_4:2]$ the identities in \eqref{puzzle1} also hold true. In these two cases, we obtain the class I case $[k+Y_2:1]$ \& $[k+Y_4:1]$ with $R_\eps^k = (S_\eps^k)^\perp$.
	If $[k-e_1+Y_4:3]$, then \eqref{puzzle0} yields that
	\begin{align*}
		S_\eps^k = (R_\eps^{k-(1,1)})^\perp = ((S_\eps^{k})^\perp)^\perp = -S_\eps^k,
	\end{align*}
	which is a contradiction.
	Suppose second that $[k-e_2+Y_2:2]$, then \eqref{puzzle0} generates the next contradiction
	\begin{align*}
		S_\eps^k = (R_\eps^{k-(1,1)})^\perp = ((S_\eps^k)^\perp)^\perp = - S.
	\end{align*}
	Finally, assume that $[k-e_2+Y_2:3]$, then it holds that
	\begin{align*}
		R_\eps^{k-(1,1)} = (S_\eps^k)^\perp\qand S_\eps^{k-e_2} = - (R_\eps^{k-e_2})^\perp,
	\end{align*}
	which we combine with \eqref{puzzle} to produce \eqref{puzzle1} again, since $[k-e_1+Y_2:1]$ is automatically satisfied.
	
	All the other scenarios mentioned above can be handled analogously and reduced to the special case of class I with one of the two additional relations $S_\eps^k = (R_\eps^k)^{\perp}$ or $R_\eps^k=(S_\eps^k)^{\perp}$.
	\item \textit{Class III:} $[k+Y_2:2]$ \& $[k+Y_4:2]$, $[k+Y_2:3]$ \& $[k+Y_4:3]$. Here, checking the different combinations of $\circled{i}$ for the restriction of $u_\eps$ to $\eps(k - e_2 +Y_2)$  and $\eps(k-e_1+Y_4)$ yields a contradiction in each case.
\end{itemize}

In summary, the only relevant class to consider is class I, and applying the implications for any $k\in J_\eps'$ yields that $S_\eps^k =S_\eps$ and $R_\eps^k=R_\eps$ for all $k\in J_\eps'$ and suitable $S_\eps, R_\eps\in \SO(2)$. 
We can now proceed as in Step 2 of the proof of Proposition \ref{prop:K_lambda} a).

The identity in \eqref{Klambda_without} can be shown as in \eqref{det_Klamba} considering that $Se_1 \cdot Re_1 \in [-1,1]$ for any $S,R\in \SO(2)$.
\medskip

b) Since we merely need to recover affine functions with gradient in $\lambda\SO(2) + (1-\lambda)\SO(2)$, the proof is almost identical with that of Proposition~\ref{prop:K_lambda} b). 
The only difference is that we can omit the scalar product $S_\eps e_1 \cdot R_\eps e_1>0$ since it does not appear in this context without orientation preservation.
\end{proof}

Whereas the deformations in the case $p \geq 2$ are strongly restricted, one observes, in accordance with intuition, much softer material behavior, as soon as microfracture in the form of discontinuities in the joints occur. 
For this next proposition, we require a suitable extension result, which we state and prove directly after.
 
\begin{proposition}[Affine limit deformations for \boldmath{$p<2$}]\label{prop:anythingcanhappen}
	Let $1<p<2$, then any affine map $u:\Omega\to \R^2$ can be approximated weakly in $W^{1,p}(\Omega;\R^2)$ by a sequence $(u_\eps)_\eps\subset W^{1,p}(\Omega;\R^2)$ in such a way that $\int_{\Omega} u_\eps\dd{x}=\int_\Omega u\dd{x}$ and
 	\begin{align*}
 		\nabla u_\eps \in \SO(2)  \text{a.e.~in $\Omega\cap\eYstiff$.}
 	\end{align*}
\end{proposition}
 
\begin{proof} Let $\nabla u=F$ with $F\in \R^{2\times 2}$.
The idea is to work here with a classical Sobolev extension result \cite[Lemma 2.5]{ACDP92}, bearing in mind that in contrast to \cite[Theorem~2.1]{CaS11}, the functions we wish to extend are defined on different connected components, which makes a pure estimate of the gradients impossible.
First, we define $v$ on the stiff components via
\begin{align}\label{def_v_p<2}
	v(x) = x - k + Fk \quad\text{ if } x\in k+ Y_1\cup Y_3
\end{align}
for some $k\in \Z^2$, see e.g., Figure \ref{fig:microcrack}.
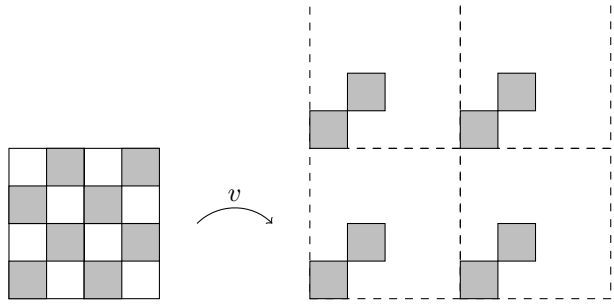
\begin{figure}
	\centering
	\begin{tikzpicture}[scale=.5]
		\foreach \x in {0,2}{
			\foreach \y in {0,2}{
				\draw (\x,\y) rectangle ($(2,2)+(\x,\y)$);
				\draw [fill=black!25!white]($(0,0)+(\x,\y)$) rectangle ($(1,1)+(\x,\y)$);
				\draw [fill=black!25!white]($(1,1)+(\x,\y)$) rectangle ($(2,2)+(\x,\y)$);
		}}
		\draw[->] (5,2) to[out=45,in=135] (7,2);
		\draw (6,2.75) node {$v$};
		\begin{scope}[shift={(8,0)}]
			\foreach \x in {0,2}{
			\foreach \y in {0,2}{
				\begin{scope}[shift={(\x,\y)}]
					\draw [dashed](\x,\y) rectangle ($(4,4)+(\x,\y)$);
					\draw [fill=black!25!white]($(0,0)+(\x,\y)$) rectangle ($(1,1)+(\x,\y)$);
					\draw [fill=black!25!white]($(1,1)+(\x,\y)$) rectangle ($(2,2)+(\x,\y)$);
				\end{scope}
		}}		
		\end{scope}
	\end{tikzpicture}
	\caption{An illustration of the microcracks induced by the deformation $v$ as in \eqref{def_v_p<2} for $F=2\Id$ and $\lambda=\frac{1}{2}$. The stiff components (and their image) are colored in grey.}\label{fig:microcrack}
\end{figure}

Let $\tilde{\Omega}\Supset\Omega$ be an open set covering $\Omega$ and let $L: W^{1,p}(\tilde{\Omega}\cap\eYstiff;\R^2)\to W^{1,p}(\Omega;\R^2)$ be the operator from Lemma \ref{lem:extension2} for $U=\tilde{\Omega}$ and $U'=\Omega$.
For $\eps>0$, let us then define the Sobolev function 
\begin{align*}
	u_\eps(x)=\eps L(v)(\tfrac{x}{\eps}) +\dashint_{\Omega} u(y) - \eps L(v)(\tfrac{y}{\eps})\dd{y},\quad x\in \Omega,
\end{align*}
which satisfies $\nabla u_\eps = \nabla \big(L(v)\big)(\frac{\cdot}{\eps}) = \Id\in\SO(2)$ a.e.~in $\Omega\cap\eYstiff$ by design. 
Moreover, Riemann-Lebesgue's Lemma yields that
\begin{align}\label{536}
	\nabla u_\eps\weakly|\Ystiff|\Id + \int_{Y_4}\nabla v \dd{x} + \int_{Y_2} \nabla v \dd{x}\quad \text{in $L^p(\Omega;\R^{2\times 2})$},
\end{align}
where the last two integrals can be calculated using Gau{\ss}-Green's theorem,
\begin{align}\label{Gaus_green_calc}
	\int_{Y_4}\nabla v\dd{x} &+ \int_{Y_2} \nabla v \dd{x} = \int_{\partial Y_4} v \otimes \nu \dd{x} + \int_{\partial Y_2} v \otimes \nu \dd{x}\nonumber\\
	& = 2\lambda(1-\lambda)\Id + \big(\lambda (Fe_1 - e_1)  | (1-\lambda) (Fe_2-e_2)\big) + \big((1-\lambda)(Fe_1-e_1)|\lambda (Fe_2-e_2)\big)\nonumber\\ 
	& = (|\Ysoft|-1)\Id + F,
\end{align}
where $\nu$ denotes the outer unit normal, cf.~also Figure \ref{fig:trace_p<2} for the boundary values of $v$ in the sense of traces.
Hence, the weak limit in~\eqref{536} is $F$. 
With Poincar\'e's inequality in mind, we then finally conclude that $u_\eps\weakly u$ in $W^{1,p}(\Omega;\R^2)$, as desired.
\begin{figure}
	\centering
	\begin{tikzpicture}[scale=1]
		\foreach \x in {0,1,2}{
			\draw [fill=black!25!white]($(0,0)+(\x,-\x)$) rectangle ($(1,1)+(\x,-\x)$);
			\draw [fill=black!25!white]($(1,1)+(\x,-\x)$) rectangle ($(2,2)+(\x,-\x)$);
		}
		\draw (1.5,.5) node {$Y_2$};
		\draw (2.5,-.5) node {$Y_4$};
		\draw [<-] (.5,.5) -- (-.5,.5) node [anchor=east]{$v(x)=x+e_1-Fe_1$};
		\draw [<-] (1.5,-.5) -- (.5,-.5) node [anchor=east]{$v(x)=x$};
		\draw [<-] (2.5,-1.5) -- (1.5,-1.5) node [anchor=east]{$v(x)=x+e_2-Fe_2$};
		
		\draw [<-] (1.5,1.5) -- (2.5,1.5) node [anchor=west]{$v(x)=x-e_2+Fe_2$};
		\draw [<-] (2.5,.5) -- (3.5,.5) node [anchor=west]{$v(x)=x$};
		\draw [<-] (3.5,-.5) -- (4.5,-.5) node [anchor=west]{$v(x)=x-e_1+Fe_1$};
	\end{tikzpicture}
	\caption{An illustration of the $Y$-periodic deformation $v$ as in \eqref{def_v_p<2}. The values of $v$ of the stiff components (colored in gray) can be used to calculate the line integral in \eqref{Gaus_green_calc} in the sense of traces.}\label{fig:trace_p<2}
\end{figure}
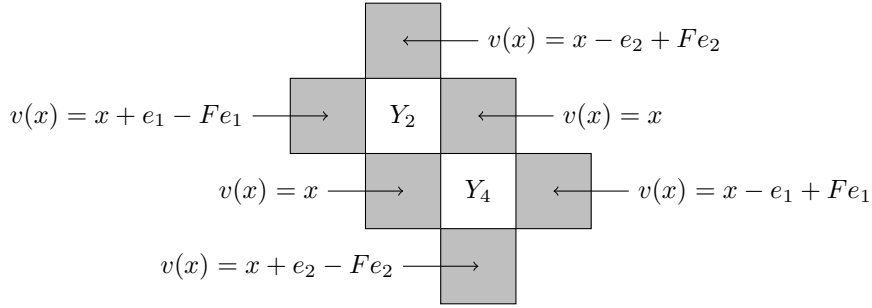
\end{proof}

This next lemma is needed to prove Proposition \ref{prop:anythingcanhappen} and derive suitable energy estimates later in Section \ref{sec:elastic_hom}.

\begin{lemma}[Extension result for checkerboard structures]\label{lem:extension2}
	Let $U'\Subset U\subset \R^2$ be bounded open sets and $\eps>0$ sufficiently small.
	
	a) If $p>2$, then there is a linear operator $L: W^{1,p}(U\cap\eYstiff;\R^2)\cap C^0(\overline{U\cap\eYstiff};\R^2) \to W^{1,p}(U';\R^2)$ such that
	$L u = u$ a.e.~in $U'\cap \eYstiff$ and
	\begin{align*}
		\norm{L u}_{W^{1,p}(U';\R^2)} \leq C\norm{u}_{W^{1,p}(U\cap\eYstiff;\R^2)}
	\end{align*}
	for a constant $C>0$ independent of $\eps,U',U$, and for every $u\in W^{1,p}(U\cap\eYstiff;\R^2)\cap C^0(\overline{U\cap\eYstiff};\R^2)$.
	
	b) If $p\in(1,2)$, then the operator $L$ in a) is defined on all of $W^{1,p}(U\cap\eYstiff;\R^2)$.
\end{lemma}
\begin{proof}
	We first cover the continuous case $p>2$.

	\textit{Step 1: A preliminary construction on the first unit cell.}
	We set $Z$ to be the union of $Y$ and its eight neighbors, i.e.,
	\begin{align}\label{cellZ}
		Z :=\bigcup_{e\in I} (e+Y)\quad\text{ with } I=\{0,\pm e_1, \pm e_2, (\pm 1, \pm 1), (\pm 1, \mp 1)\};
	\end{align}
	we analogously set $\Zstiff$ as the union of $\Ystiff$ and the stiff components all its eight neighboring cells.
	Moreover, consider the space
	\begin{align}\label{trace_space}
		B = \Big\{(g_i)_i\in \prod_{i=1}^4W^{1-\tfrac{1}{p},p}(\Gamma_i;\R^2) : g_{i-1}(x_{i})=g_{i}(x_i) \text{ for all } i\in\{1,\ldots,4\}\Big\},
	\end{align}
	where $g_0 = g_4$, and $\Gamma_1,\ldots,\Gamma_4\subset \partial Y_2$ are the four straight boundary pieces of the polygon $Y_2$ and $x_i\in\partial Y_2$ are the four vertices of $Y_2$, all numbered clockwise, starting in the lower left corner.
	The space $B$ is exactly the trace space of $Y_2$ as can be seen in \cite[Theorem 1.5.2.3 b)]{Gr85}.
	Let $T: W^{1,p}(Y_2;\R^2)\to B$ be the trace operator on the domain $Y_2$ and let
	\begin{align*}
		T_1&: W^{1,p}(-e_1 + Y_3;\R^2)\to W^{1-\tfrac{1}{p},p}(\Gamma_1;\R^2),\quad T_2: W^{1,p}(e_2 + Y_1;\R^2) \to W^{1-\tfrac{1}{p},p}(\Gamma_2;\R^2)\\
		T_3&: W^{1,p}(Y_3;\R^2)\to W^{1-\tfrac{1}{p},p}(\Gamma_3;\R^2),\quad T_4: W^{1,p}(Y_1;\R^2) \to W^{1-\tfrac{1}{p},p}(\Gamma_4;\R^2)
	\end{align*}
	be the projections of the trace operators of the neighboring stiff components onto $\Gamma_1,\ldots,\Gamma_4$.
	In light of \cite[Theorem 4.2]{LaP20}, there exists a linear and continuous right inverse $S$ of $T$.
	By composing $S$ with $(T_1,\ldots,T_4)$ and arguing similarly on $Y_4$, we find a linear and continuous operator $L^{(1)}: W^{1,p}(\Zstiff;\R^2)\cap C^0(\overline{\Zstiff};\R^2)\to W^{1,p}(Y;\R^2)$ such that $L^{(1)} u = u$ a.e.~in $\Ystiff$ and
	\begin{align*}
		\norm{L^{(1)} u}_{W^{1,p}(Y;\R^2)} &\leq C(\lambda,p)\norm{u}_{W^{1,p}(\Zstiff;\R^2)}
	\end{align*}
	for every $u\in W^{1,p}(U\cap\eYstiff;\R^2)$.
	
	\medskip
	
	\textit{Step 2: Extension on large domains.} Now, let $V'\Subset V\subset \R^2$ and $\eps$ sufficiently small.
	Then, there exists an operator $L^{(2)}: W^{1,p}(V\cap \Ystiff;\R^2)\cap C^0(\overline{V\cap\Ystiff};\R^2)\to W^{1,p}(V';\R^2)$ such that $L^{(2)} u = u$ a.e.~in $V'\cap \Ystiff$ and
	\begin{align}\label{est_big}
		\begin{split}
			\norm{L^{(2)} u}_{W^{1,p}(V';\R^2)} &\leq C\norm{u}_{W^{1,p}(V\cap\Ystiff;\R^2)},
		\end{split}
	\end{align}
	where $\Ystiff$ now denotes the $Y$-periodic extension in this step.
	Indeed, with $J'=\{k\in\Z: V'\cap(k+Y) \neq \emptyset\}$ we obtain
	\begin{align}\label{coverYZ}
		V' \subset \bigcup_{k\in J'}k+Y\subset \bigcup_{k\in J'}\eps(k+Z) \subset V,
	\end{align}
	which then allows us to work cell-wise. 
	With $\pi^\xi(x):= x+\xi$ for $x,\xi\in\R^2$, we find for fixed $k\in J$ and $u\in W^{1,p}(V\cap \Ystiff)\cap C^0(\overline{V\cap\Ystiff};\R^2)$ the function
	\begin{align*}
		u_k:= L^{(1)}(u\restrict{k+\Zstiff} \circ \pi^{k})\circ \pi^{-k} \in W^{1,p}(k+(\Zstiff\cup Y);\R^2)
	\end{align*}
	with $L^{(1)}$ as in Step 2.
	Since $u_k=u$ on $k+\Zstiff$, we obtain that
	\begin{align*}
		L^{(2)}: W^{1,p}(V\cap \Ystiff;\R^2)\cap C^0(\overline{V\cap\Ystiff};\R^2)\to W^{1,p}(V';\R^2), (L^{(2)}u)(x)=u_k(x)\text{ if }x\in k + Y
	\end{align*}
	is well-defined and satisfies $L^{(2)}u=u$ on $V'\cap\Ystiff$. On each $k+Y$, it holds that
	\begin{align*}
		\norm{u_k}_{W^{1,p}(k+Y;\R^2)} &= \norm{L_0(u\restrict{k+\Zstiff}\circ \pi^k)\circ \pi^{-k}}_{W^{1,p}(k+Y;\R^2)} = \norm{L_0(u\restrict{k+\Zstiff}\circ \pi^k)}_{W^{1,p}(Y;\R^2)}\\
		&\leq C(\lambda,p)\norm{u\restrict{k+\Zstiff}\circ \pi^k}_{W^{1,p}(\Zstiff;\R^2)} = C(\lambda,p)\norm{u\restrict{k+\Zstiff}}_{W^{1,p}(k+\Zstiff;\R^2)}.
	\end{align*}
	Summing this estimate over all $k\in J'$ and exploiting \eqref{coverYZ} then yields \eqref{est_big}.

	\medskip

	\textit{Step 3: Scaling analysis.} The desired extension operator follows immediately from a scaling analysis as in the first step of the proof of \cite[Theorem 2.1]{ACDP92}.	
	
	\medskip
	
	\textit{Step 4:} To obtain the desired result for $p\in(1,2)$, we merely need to add the fact that the trace space $B$ corresponding to $Y_2$ as in \eqref{trace_space} is now simply
	\begin{align*}
		B = \prod_{i=1}^4W^{1-\tfrac{1}{p},p}(\Gamma_i;\R^2)
	\end{align*}
	in light of \cite[Theorem 1.5.2.3 a)]{Gr85}; one works analogously on $Y_4$.
	Omitting the intersection with a suitable space of continuous functions, the rest of the proof can be handled exactly as in the three steps before.
\end{proof}

\begin{remark}[Porous checkerboard structures]
So far, we have dealt with checkerboard structures composed of elastically stiff squares $\eYstiff$ and soft rectangles $\eYsoft$, so that the entire reference configuration $\Omega$ consists of an elastic material. While this model is relevant, for example, in the production of waterproof or airtight auxetic materials, the porous counterpart, where $\eYsoft$ is replaced by void, is also of significance.
	
	To model this scenario, we choose a bounded Lipschitz domain $\Omega'\Subset\Omega$ and work with energies defined on the set $\Acal_\eps$ of all functions $u\in W^{1,p}(\Omega\cap\eYstiff;\R^2)\cap C^0(\overline{\Omega\cap\eYstiff};\R^2)$ with $\int_{\Omega'\cap \eYstiff} u \dd x =0$ and $\norm{u}_{L^p(\Omega\cap\eYstiff;\R^2)}\leq M\norm{u}_{L^p(\Omega'\cap\eYstiff;\R^2)}$ for a fixed constant $M>0$; the latter condition serves to avoid concentration effects near the boundary of $\Omega$.
	Precisely, the energies are defined as
	\begin{align*}
		\Ical_\eps : \Acal_\eps\to [0,\infty],\ u\mapsto \int_{\Omega\cap\eYstiff} \Wrig(\nabla u) \dd x,
	\end{align*}
	with $\Wrig$ as in \eqref{Wrig2} and $p>2$.
	Since $\Ical_\eps$ is defined on $\eps$-dependent spaces, it is necessary to explain the underlying topology of a corresponding $\Gamma$-convergence (and compactness) result.
	In light of Lemma \ref{lem:extension2}, every $u\in\Acal$ can be extended to a function $Lu$ in $W^{1,p}(\Omega';\R^2)$ with estimates of the $W^{1,p}$-norms, which allows us to use the weak topology in $W^{1,p}(\Omega';\R^2)$ for the $\Gamma$-convergence of $(\Ical_\eps)_\eps$.
	In particular, we say that a sequence $(u_\eps)_\eps$ with $u_\eps\in \Acal_\eps$ converges to $u\in W^{1,p}(\Omega';\R^2)$ in $W^{1,p}(\Omega';\R^2)$ if the sequence $(Lu_\eps)_\eps\subset W^{1,p}(\Omega';\R^2)$ does so. 
	
	With this notion of convergence, it is straightforward to show that $(\Ical_\eps)_\eps$ $\Gamma$-converges to the constant zero function defined on the set of all affine deformations with vanishing mean value and gradient in $K$, cf.~\eqref{Klambda}, considering that compactness follows in view of the continuity of $L$ the Poincar\'e's inequality as in Lemma \ref{lem:poincare2} below.
\end{remark}

\begin{remark}[Checkerboard structures with rigid rectangles]\label{rem:result_rectangle}
By an analogous argumentation as in the proofs of Proposition \ref{prop:K_lambda}, periodic high-contrast geometries with stiff parts consisting of rectangles can be handled as well. In this situation, we set 
$$Y_1=(0,\lambda] \times (0, \mu],\quad Y_2=(0, \lambda]\times (\mu, 1],\quad Y_3 = (\lambda,1] \times (\mu,1],\quad Y_4=(\lambda, 1]\times (0, \mu]$$
for given $\lambda,\mu \in (0,1)$. Instead of the weak limit \eqref{weaklimit_w}, we now obtain
\begin{align*}
	\nabla w_\eps\weakly \int_Y \nabla w \dd{x} &= \lambda \mu S +(1-\lambda)(1-\mu)R + \lambda (1-\mu)(Se_1|Re_2) + \mu (1-\lambda) (Re_1|Se_2)\\ 
	&= ((\lambda S + (1-\lambda)R)e_1|(\mu S + (1-\mu)R)e_2)  \quad\text{in }L^p(\Omega;\R^{2\times 2}).
\end{align*}
This yields that admissible limit deformations are affine with gradient in
\begin{align*}
	K=\{((\lambda S + (1-\lambda)R)e_1|(\mu S + (1-\mu)R)e_2): R,S\in \SO(2), Re_1\cdot Se_1 \geq 0\}.
\end{align*}
For every $R,S\in \SO(2)$ with $Re_1\cdot Se_1 \geq 0$, we have
\begin{align*}
		\det\big(((\lambda S + (1-\lambda)R)e_1|(\mu S + (1-\mu)R)e_2)\big) 
&= \lambda \mu +(1-\lambda)(1-\mu) + (\lambda(1-\mu)+(1-\lambda)\mu)Se_1\cdot Re_1	\\
&= |\Ystiff| + |\Ysoft|Se_1 \cdot Re_1 \geq |\Ystiff|,
	\end{align*}	 
as well as $$|(\lambda S + (1-\lambda)R)e_1| \leq 1, \quad|(\mu S + (1-\mu)R)e_2| \leq 1,$$ but 
$$(\lambda S + (1-\lambda)R)e_1 \cdot (\mu S + (1-\mu)R)e_2 = (\mu- \lambda)
Re_1\cdot Se_2.$$ Hence, $F \in K$ is not necessarily a conformal contraction any more but the Poisson's ratio corresponding to $F$ is still negative.  
\end{remark}

\section{Analysis of the model with stiff tiles}\label{sec:elastic}
\subsection{Technical tools}\label{sec:elastic_aux}
We begin the analysis of the model with diverging elastic energy by establishing a replacement for the local results Lemma \ref{lem:affine} and Corollary \ref{cor:preservation}.
In contrast to Section \ref{sec:rigid_aux}, where we merely needed to consider the boundary values at a single soft rectangle, our analysis now requires the four neighboring rigid squares as well.
In this section, we consider for $\mu\in (0,1]$ the following cross-like structure
\begin{align}\label{cross-structure}
	 \begin{split}
	 	E &= \bigcup_{i=0}^4 E_i,\quad E' = E\setminus E_0 \text{ with } \\
	 	E_0 &= (0,1]\times (0,\mu],\ E_1 = (0,1]\times (-1,0],\ E_2 = (-\mu,0]\times (0,\mu],\\
	 	E_3 &= E_1 + (1+\mu)e_2,\ E_4=E_2 + (1+\mu)e_1,
	 \end{split}
\end{align}
see also Figure \ref{fig:cross}.
\begin{figure}[h!]
	\centering
	\begin{tikzpicture}
		\draw (0,0) rectangle (1.5,1);
		\draw (-1,0) rectangle (0,1);
		\draw (1.5,0) rectangle (2.5,1);
		\draw (0,-1.5) rectangle (1.5,0);
		\draw (0,1) rectangle (1.5,2.5);
		\draw (.75,0.5) node {$E_0$};
		\draw (-0.5,0.5) node {$E_2$};
		\draw (2,0.5) node {$E_4$};
		\draw (.75,-.75) node {$E_1$};
		\draw (.75,1.75) node {$E_3$};	
		\draw (2.5,2) node {$E$};
		
		\draw [<->] (-1.2,0) --++ (0,0.5) node [anchor=east] {$\mu$} --++ (0,0.5);
		\draw [<->] (0,-1.7) --++ (.75,0) node [anchor=north] {$1$} --++ (.75,0);
		
		\fill (0,0) circle (1.5pt);
		\draw (0,0) node [anchor = north east] {$\small x_1$};
		\fill (0,1) circle (1.5pt);
		\draw (0,1) node [anchor = south east] {$\small x_2$};
		\fill (1.5,1) circle (1.5pt);
		\draw (1.5,1) node [anchor = south west] {$\small x_3$};
		\fill (1.5,0) circle (1.5pt);
		\draw (1.5,0) node [anchor = north west] {$\small x_4$};
	\end{tikzpicture}
	\caption{An illustration of the cross structure defined in \eqref{cross-structure}}\label{fig:cross}
\end{figure}
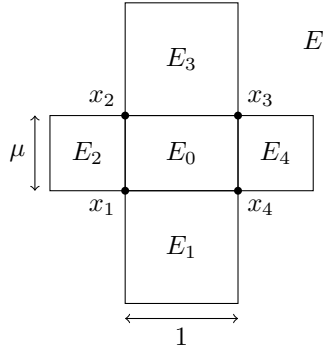 

We begin with a brief lemma about transferring the Ciarlet-Ne\v{c}as condition from one function to one that is sufficiently close with respect to the $W^{1,p}$-norm.

\begin{lemma}[Approximate Ciarlet-Ne\v{c}as condition]\label{lem:approx_ciarlet}
Let $p>2$, $M\subset \R^2$ be the union of finitely many bounded Lipschitz-domains. If $u\in W^{1,p}(M;\R^2)$ satisfies the Ciarlet-Ne\v{c}as condition \eqref{ciarlet_necas} for $\Omega=M$
and there is $v\in W^{1,p}(M;\R^2)$ with
\begin{align}\label{u-v_Sobolev_estimate}
	\norm{u-v}_{W^{1,p}(M;\R^2)}\leq h
\end{align}
for some  $h\in(0,1)$ sufficiently small, then there exists a constant $C=C(M,p)>0$ such that
\begin{align*}
	\int_{M}|\det \nabla v| \dd x \leq |v(M)| + C\big(1+\norm{\nabla v}_{L^2(M;\R^{2\times 2})}\big)h.
\end{align*}
\end{lemma}
\begin{proof}
	Let $M_1,\ldots, M_n\subset \R^2$ for $n\in\N$ be the finitely many bounded Lipschitz domains that comprise $M$, i.e., $M= \bigcup_{i=1}^nM_i$.
	In light of the Sobolev embeddings applied to each $M_i$, the bound \eqref{u-v_Sobolev_estimate} is (up to a constant $C_1=C_1(M,p)>0$) also uniform on $M_i$. 
	We therefore obtain
	\begin{align*}
		u(M_i) \subset v(M_i) + \overline{B(0,C_1h)},
	\end{align*}
	which, after taking the union $i=1,\ldots, n$, leads to the estimate
	\begin{align}\label{ciarlet_upper}
		|u(M)| &\leq |v(M)| + C_2h
	\end{align}
	for a constant $C_2=C_2(M,p)>0$.
	On the other hand, the estimate \eqref{u-v_Sobolev_estimate} yields that
	\begin{align}\label{ciarlet_lower}
		\int_{M} &|\det \nabla u - \det \nabla v| \dd x = \int_{M} |(\partial_1 u)^\perp\cdot \partial_2 u - (\partial_1 v)^\perp\cdot \partial_2 v| \dd x\nonumber\\
		&= \int_{M} |(\partial_1 u)^\perp(\partial_2 u - \partial_2 v) + (\partial_1 u - \partial_1 v)^\perp\cdot\partial_2 v| \dd x\nonumber\\
		&\leq \int_{M} |\partial_1 u||\partial_2 u - \partial_2 v| + |\partial_1 u - \partial_1 v||\partial_2 v| \dd x\nonumber\\
		&\leq (\norm{\partial_1 u}_{L^2(M;\R^2)} + \norm{\partial_2 v}_{L^2(M;\R^2)})\norm{u-v}_{W^{1,2}(M;\R^2)}\nonumber\\
		&\leq (\norm{u-v}_{W^{1,2}(M;\R^2)}+2\norm{\nabla v}_{L^2(M;\R^{2\times 2})})\norm{u-v}_{W^{1,2}(M;\R^2)}\nonumber\\
		&\leq (C_1h+2\norm{\nabla v}_{L^2(M;\R^{2\times 2})})C_1h \leq 2C_1\big(1+\norm{\nabla v}_{L^2(M;\R^{2\times 2})}\big)h
	\end{align}
	if $C_1h<2$.
	Now, we combine \eqref{ciarlet_necas} with the estimates \eqref{ciarlet_lower} and \eqref{ciarlet_upper} to conclude that
	\begin{align*}
		\int_M |\det \nabla v| \dd x \leq |v(M)| + (2C_1+C_2)\big(1+\norm{\nabla v}_{L^2(M;\R^{2\times 2})}\big)h.
	\end{align*}
\end{proof}

The next lemma, which is substantial for characterizing the set of admissible limit deformations, is a quantitative rigidity estimate in the spirit of \cite{FJM02} for cross structures $E'$ as in \eqref{cross-structure}.
By combining Lemma \ref{lem:approx_ciarlet} with careful geometric arguments, we show that the rotations on opposite squares can be selected identical while controlling the error terms.
This result demonstrates, in particular, that Corollary \ref{cor:preservation} a) holds true if orientation preservation is replaced by non-self-interpenetration of matter.

\begin{lemma}[Quantitative rigidity estimate for cross structures]\label{lem:luftschloss}
Let $p>2$ and $E,E',E_0,\ldots, E_4$ be as in \eqref{cross-structure}.
There is a constant $C=C(p)>0$ and $\delta_0 = \delta_0(p)$ with the following property: For every $u\in W^{1,p}(E;\R^2)$ satisfying the Ciarlet-Ne\v{c}as condition \eqref{ciarlet_necas} on $E'$ and for which $\norm{\dist(\nabla u, \SO(2))}_{L^p(E')}=:\delta < \delta_0$, there exist $R,S\in \SO(2)$ such that
\begin{align}\label{luftschloss_absch}
	\norm{\nabla u - S}_{L^p(E_1\cup E_3;\R^{2\times 2})} + \norm{\nabla u - R}_{L^p(E_2\cup E_4;\R^{2\times 2})} \leq C\delta^\frac{1}{2}
\end{align}
and 
\begin{align}\label{approx_scalar}
	Re_1 \cdot Se_1 \geq - C \delta^\frac{1}{2}.
\end{align}
\end{lemma}
\begin{proof}
	This proof concentrates on the more delicate scenario $\mu=1$, while the case $\mu\in (0,1)$ shall be discussed at the end.
	We first establish \eqref{luftschloss_absch} and \eqref{approx_scalar} with the right-hand side $C\delta^{\frac{1}{3}}$ and improve the estimate later on.
	
	\medskip
	
	\textit{Step 1: Geometric setup.} Due to the quantitative geometric rigidity estimate by Friesecke, James \& M\"uller \cite[Theorem 3.1]{FJM02} there exist four matrices $S_1,S_3,R_2,R_4\in \SO(2)$ such that
	\begin{align*}
		\norm{\nabla u - S_i}_{L^p(E_i,\R^{2\times 2})}, \norm{\nabla u - R_j}_{L^p(E_j,\R^{2\times 2})} \leq C\norm{\dist(\nabla u, \SO(2))}_{L^p(E')}
	\end{align*}
	for all $i\in\{1,3\}$ and $j\in\{2,4\}$.
	The reversed triangle inequality then yields that
	\begin{align}\label{fjm_inital_estimate}
		\begin{split}
			\norm{\nabla u - S_1}_{L^p(E_1\cup E_3;\R^{2\times 2})} &\leq C(\norm{\dist(\nabla u, \SO(2))}_{L^p(E')} + |S_1 - S_3|),\\
			\norm{\nabla u - R_2}_{L^p(E_2\cup E_4;\R^{2\times 2})} &\leq C(\norm{\dist(\nabla u, \SO(2))}_{L^p(E')} + |R_2 - R_4|).
		\end{split}
	\end{align}
	Our primary task is to obtain an estimate for the quantities $|S_1-S_3|$ and $|R_2-R_4|$ in terms of powers of $\norm{\dist(\nabla u, \SO(2))}_{L^p(E')}$.
	For $i\in\{1,3\}$, $j\in \{2,4\}$ we set $s_i = \int_{E_i} u(x) - S_i x \dd x$, $r_j= \int_{E_j} u(x) - R_j x \dd x$ and introduce the auxiliary functions
	\begin{align*}
		v_k: \overline{E_k} \to \R^2,\ x\mapsto \begin{cases}
										S_i x + s_i &\text{ if } k = i,\\
										R_j x + r_j	&\text{ if } k = j,
									 \end{cases}
									 \quad\text{ for } k\in\{1,\ldots,4\}.
	\end{align*}
	From Poincar\'e's inequality and the Sobolev embeddings, we then obtain for all $k\in\{1,\ldots, 4\}$ the estimates
	\begin{align}\label{uniform_eta}
		\norm{u-v_k}_{C^0(\overline{E_k};\R^2)} \leq C \norm{u-v_k}_{W^{1,p}(E_k;\R^2)} \leq C \norm{\dist(\nabla u, \SO(2)}_{L^p(E')}=:\eta.
	\end{align}
	From this uniform estimate, we infer that $u(\overline{E_k})\subset v_k(\overline{E_k}) + \overline{B(0,\eta)}$; in particular, it holds that
	\begin{align*}
		u(\partial E_0) \subset \bigcup_{k=1}^4 v_k(\partial E_k \cap \partial E_0) + \overline{B(0,\eta)},
	\end{align*}
	see also Figure \ref{fig:boundary}. To shorten the notation, we set
	\begin{align}\label{abcd}
		\begin{split}
			a=v_1(x_1),\ b=v_3(x_2),\ c=v_3(x_3),\ d=v_1(x_4),\\
			a'=v_2(x_1),\ b'=v_2(x_2),\ c'=v_4(x_3),\ d'=v_4(x_4),
		\end{split}
	\end{align}
	and find that \eqref{uniform_eta} and the continuity of $u$ yields that
	\begin{align}\label{a-a'}
		|a-a'|,|b-b'|,|c-c'|,|d-d'|\leq 2\eta.
	\end{align}
	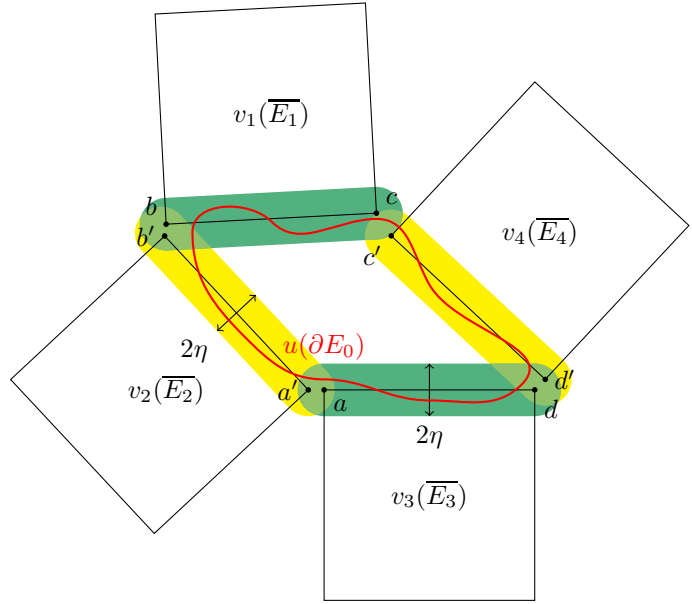
\begin{figure}
		\centering
		\begin{tikzpicture}[scale=.7]
			\fill[yellow] ($(-0.3,0) + (43:.5)$) arc (43:-137:.5) --++(133:4) arc (233:43:.5) -- cycle;
			\fill[yellow] ($(4.2,0.2)+(227:.5)$) arc (-133:47:.5) --++ (137:4) arc (47:227:.5) -- cycle;
			\fill[ForestGreen!70!white] (0,-0.5) arc (270:90:0.5) -- (4,0.5) arc (90:-90:0.5) -- cycle;
			\fill[ForestGreen!70!white] ($(-3,3.15) + (273:.5)$) arc (273:93:.5) --++ (3:4) arc (93:-87:.5) -- cycle;			
			
			\begin{scope}
				\clip ($(-0.3,0) + (43:.5)$) arc (43:-137:.5) --++(133:4) arc (233:43:.5) -- cycle;
				\fill[ForestGreen!60!yellow!80!white] (0,-0.5) arc (270:90:0.5) -- (4,0.5) arc (90:-90:0.5) -- cycle;
			\end{scope}
			\begin{scope}
				\clip ($(4.2,0.2)+(227:.5)$) arc (-133:47:.5) --++ (137:4) arc (47:227:.5) -- cycle;
				\fill[ForestGreen!60!yellow!80!white] (0,-0.5) arc (270:90:0.5) -- (4,0.5) arc (90:-90:0.5) -- cycle;
			\end{scope}
			\begin{scope}
				\clip ($(-0.3,0) + (43:.5)$) arc (43:-137:.5) --++(133:4) arc (233:43:.5) -- cycle;
				\fill[ForestGreen!60!yellow!80!white] ($(-3,3.15) + (273:.5)$) arc (273:93:.5) --++ (3:4) arc (93:-87:.5) -- cycle;
			\end{scope}
			\begin{scope}
				\clip ($(4.2,0.2)+(227:.5)$) arc (-133:47:.5) --++ (137:4) arc (47:227:.5) -- cycle;
				\fill[ForestGreen!60!yellow!80!white] ($(-3,3.15) + (273:.5)$) arc (273:93:.5) --++ (3:4) arc (93:-87:.5) -- cycle;
			\end{scope}
			\draw (-0.3,0) --++ (133:4) --++ (223:4)--++(313:4)--cycle;
			\draw (0,0) --++ (0:4) --++ (-90:4) --++(-180:4) --cycle;
			\draw (-3,3.15) --++ (3:4) --++ (93:4) --++(183:4) --cycle;
			\draw (4.2,0.2) --++ (137:4) --++ (47:4) --++ (-43:4)--cycle;
			\draw [<->] (2,.5) -- (2,-.5) node [anchor = north]{$2\eta$}; 
			\draw [<->] ($(-0.3,0) + (133:2) + (43:.5)$) --++ (223:1) node [anchor=north east] {$2\eta$};
			
			\fill (0,0) circle (1.5pt) node [anchor = north west] {$a$};
			\fill (4,0) circle (1.5pt) node [anchor = north west] {$d$};
			\fill (-3,3.15) circle (1.5pt) node [anchor = south east] {$b$};
			\fill ($(-3,3.15)+(3:4)$) circle (1.5pt) node [anchor = south west] {$c$};
			
			\fill (-0.3,0) circle (1.5pt) node [anchor = east] {$a'$};
			\fill ($(-0.3,0)+(133:4)$) circle (1.5pt) node [anchor = east] {$b'$};
			\fill (4.2,0.2) circle (1.5pt) node [anchor = west] {$d'$};
			\fill ($(4.2,0.2) + (137:4)$) circle (1.5pt) node [anchor = north east] {$c'$};
			
			\draw [thick,red] (0,.2) to[out=0,in=180] (2,-.2) to [out=0,in=225] (3.75,.1) to[out=45,in=-60] (2,2) to[out=120,in=0] (1,3.25) to[out=180,in=-45] (-1,3.2) to[out=135,in=90] (-2.5,3) to[out=-90,in=135] (-1.5,1) to[out=-45,in=180] (0,.2);
			
			\draw (0,.4) node [red, anchor=south] {$u(\partial E_0)$};
			
			\draw (-3,0) node {$v_2(\overline{E_2})$};
			\draw (2,-2) node {$v_3(\overline{E_3})$};
			\draw (4.1,3) node {$v_4(\overline{E_4})$};
			\draw (-1,5.25) node {$v_1(\overline{E_1})$};			
		\end{tikzpicture}
		\caption{The four rotated squares $v_1(\overline{E_1}),\ldots, v_4(\overline{E_4})$ (which have side length $1$). The connected image $u(\partial E_0)$ (colored in red) is contained in the four closed tubes (colored in yellow and green) of thickness $2\eta$. The points $a,b,c,d,a',b',c',d'$ are defined as in \eqref{abcd}.}\label{fig:boundary}
	\end{figure}	
	
	The goal for the remainder of this proof is to show that the polygons $abcd$ and $a'b'c'd'$ are close to a parallelogram with a small error in terms of powers of $\eta$; note that
	$S_1e_1 =S_3e_1$ (or $R_2e_2 = R_4e_2$) if $abcd$ (or $a'b'c'd'$) is a parallelogram.
	First, we focus on the polygon $abcd$ and estimate the deviation of $c-b=v_3(x_3) - v_3(x_2) = S_3 e_1$ from $d-a = v_1(x_4) - v_1(x_1) = S_1e_1$.
	In light of \eqref{abcd}, \eqref{a-a'} and the fact that $v_2,v_4$ are a rigid body motions, we find that 
	\begin{align}\label{bc_annulus}
		b\in \overline{A(a,1-4\eta,1+4\eta)} \qand c\in \overline{A(d,1-4\eta, 1+4\eta)}
	\end{align}
	if $\eta$ is sufficiently small, cf.~also for the notation of the annuli \eqref{annulus}. 
	Moreover, it holds that 
	\begin{align*}
		|d-a|=1 \qand |b-c|=1
	\end{align*} 
	since $v_1$ and $v_3$ are rigid body motions.

	\medskip	
	
	\textit{Step 2: Auxiliary function.} In this step, we show that there exists a continuous and piecewise affine function $\bar{v}:E\to \R^2$ such that
	$\bar{v}\restrict{E_i}$ is a rigid body motion, and
	\begin{align}\label{v-barv}
		\norm{v_i-\bar{v}}_{W^{1,p}(E_i;\R^2)}\leq C \eta^{\frac{1}{3}}.
	\end{align}
	for every $i\in\{1,\ldots,4\}$. Such a function is uniquely determined on $E'$ by the vertices $\bar{v}(x_1),\ldots, \bar{v}(x_4)$. Finding such suitable points is the goal of this next step.
	
	\smallskip
	
	\textit{Step 2a: Auxiliary points.}
	We show that there exist $\bar{b}\in a+ \Scal^1$ and $\bar{c} \in d+\Scal^1$ such that $|\bar{b} -\bar{c}|=1$ and
	\begin{align}\label{quadratloesung}
		|\bar{b}-b| + |\bar{c}-c| \leq C \eta^\frac{1}{3}.
	\end{align}
	\begin{figure}
		\centering
		\begin{tikzpicture}[scale=.7]
			\coordinate (a) at (0,0);
			\coordinate (b) at (129:4.1);
			\coordinate (bb) at (135:4);
			\coordinate (cc) at (0,0);
			\coordinate (d) at (4,0);
			\def\ringa{(a) circle (3.5) (a) circle (4.5)}
			\def\ringb{(d) circle (3.5) (d) circle (4.5)}
			
			\path [draw=none,fill=Cerulean!60!white,even odd rule] (a) circle (4.5) (a) circle (3.5);
			\path [draw=none,fill=Dandelion!80!white,even odd rule] (d) circle (4.5) (d) circle (3.5);
			
			\begin{scope}[even odd rule]
		        \clip \ringa;
		        \fill[fill=Dandelion!50!Cerulean!80!white] \ringb;
		    \end{scope}
		    
			\draw (a) circle (4);
			\fill (a) circle (2pt);
			\draw[name path = circle1] (d) circle (4);
			\path [name path = bigcircle1] (d) circle (4.5);
			\path [name path = smallcircle1] (d) circle (3.5);
			\fill (d) circle (2pt);
			\draw (d) node [anchor = north west]{$d$};
			
			\draw [<->] (-4.5,0) --++(1,0);
			\draw (-3.95,0) node[anchor = south east] {$8\eta$};
			\draw [<->] (7.5,0) --++(1,0);
			\draw (7.95,0) node[anchor = north west] {$8\eta$};

			\fill (b) circle (2pt);
			\draw [name path = circle3] (b) circle (4);
			\draw (b) node [anchor = south west]{$b$};
			\begin{scope}[even odd rule]
				\clip \ringb;
				\draw[ultra thick,red] (b) circle (4);
			\end{scope}
			\draw[name path = circle2, blue] (bb) circle (4);
			\fill[blue] (bb) circle (2pt);
			\draw (bb) node [anchor = east]{$\bar{b}$};
			\path [name intersections={of=circle1 and circle2}];
			\fill[blue] (intersection-2) circle (2pt);
			\draw [dashed] (0,0) -- (4,0) -- (intersection-2) -- (135:4)--cycle;
			\draw ($(intersection-2)-(0.1,0)$) node [anchor = south east] {$\bar{c}$};
			
			\path [name intersections={of=bigcircle1 and circle3}];
			\draw [thick, black,shorten >=-90] (intersection-1) -- (intersection-2);
			\draw [thick, black,shorten >=-60] (intersection-2) -- (intersection-1);
			\path [name intersections={of=smallcircle1 and circle3}];
			\draw [thick, black,shorten >=-90] (intersection-1) -- (intersection-2);
			\draw [thick, black,shorten >=-115] (intersection-2) -- (intersection-1);
			
			\draw [<->] (-2,-4.8) --++ (157:.7) node [anchor=north east] {$C\eta^{\frac{2}{3}}$};
			
			\draw ($(a)+(.1,0)$) node [anchor = north west]{$a$};
		\end{tikzpicture}
		\caption{The two colored annuli describe the regions in which the points $b,c$ can lie once $a$ and $d$ are fixed.
		If $b$ is sufficiently far from $d$, then we choose $\bar{b}\in a + \Scal^1$ such that $|b-\bar{b}|$ is small and compare the intersection of $\bar{b}+\Scal^1$ and $d + \Scal^1$ with the point $c$, which lies in the intersection of the orange annulus with $b+\Scal^1$ (colored in red).}\label{fig:kandinsky}
	\end{figure}
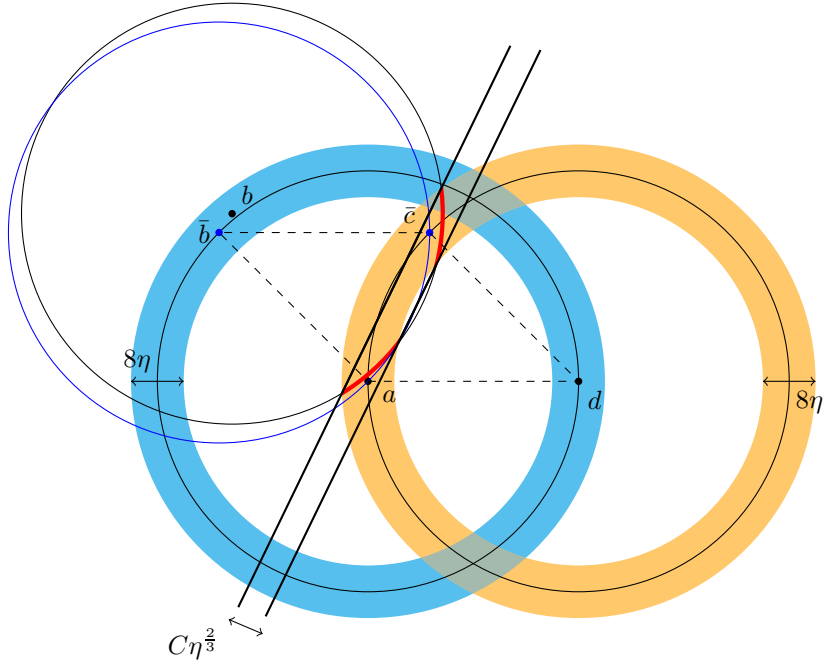
	If $|b-d|\leq \eta^\frac{1}{3}$, then we choose $\bar{b} = d$ and an arbitrary $\bar{c}\in d + \Scal^1$ such that $|\bar{c} - c| \leq 4 \eta$.
	In this case, \eqref{quadratloesung} holds if $\eta\ll 1$.
	
	Now, let $|b-d|>\eta^{\frac{1}{3}}$, then we choose an arbitrary $\bar{b}\in a + \Scal^1$ with $|b-\bar{b}| \leq 4\eta$ and search for a point $\bar{c}\in d +\Scal^1$ that satisfies $|\bar{b}-\bar{c}|=1$; note that there exist at least one but at most two options. 
	In any case, the following system of equations has to be satisfied by $\bar{c}$ and $c$:
	\begin{align*}
		\begin{cases}
			|\bar{c}|^2 - 2\bar{c}\cdot d + |d|^2 = |\bar{c}-d|^2=1,\\
			|\bar{c}|^2 - 2\bar{c}\cdot \bar{b} + |\bar{b}|^2 = |\bar{c}-\bar{b}|^2=1,
		\end{cases}\qand
		\begin{cases}
			|c|^2 - 2c\cdot d + |d|^2 = |c-d|^2=r^2,\\
			|c|^2 - 2c\cdot b + |b|^2 = |c-b|^2=1,
		\end{cases}
	\end{align*}
	for some $r\in [1-4\eta,1+4\eta]$.
	By suitably combining these equations, we obtain for the difference $c-\bar{c}$ in the direction $d-b$ that
	\begin{align}\label{one_direction_est1}
		2(c-\bar{c})\cdot(d-b) &=2c\cdot(d-b) - 2\bar{c}\cdot (d-\bar{b}) + 2\bar{c}\cdot (b - \bar{b})\nonumber\\
			&\leq 1-r^2 + |d|^2 - |b|^2  - |d|^2 + |\bar{b}|^2  + 2|\bar{c}||b-\bar{b}|\nonumber\\
			&\leq 1-r^2 + 4(|\bar{b}|-|b|)(|\bar{b}|+|b|) + 8\eta|\bar{c}|\nonumber\\
			&\leq 8\eta + 16\eta^2 + 4\eta(|b-a|+|\bar{b}-a| + 2|a|) + 8\eta(|\bar{c} - d| + |d-a| +|a|)\nonumber\\
			&\leq 24\eta + 4\eta(3+2|a|) + 8\eta(3+|a|)
	\end{align}
	since $\eta\ll 1$. Without loss of generality, we may now assume that $a=0$, otherwise we move the coordinate system. 
	Since $|d-b|>\eta^{\frac{1}{3}}$ we conclude that
	\begin{align}\label{one_direction_est2}
		\left|(c-\bar{c})\cdot\frac{d-b}{|d-b|}\right| \leq C \frac{\eta}{\eta^{\frac{1}{3}}} = C \eta^{\frac{2}{3}}.
	\end{align}
	Essentially, this estimate ensures that $\bar{c}$ lies in an infinitely long tube in the direction $(d-b)^\perp$ with thickness $C\eta^{\frac{2}{3}}$ around $c$, see also Figure \ref{fig:kandinsky}.
	The intersection of this tube with the annulus $\overline{A(d,1-4\eta, 1+4\eta)}$ has at most two connected components $T_1,T_2\subset \R^2$.
	If $T_1\neq T_2$, then we find one of the two possible choices for $\bar{c}$ in each of the two sets. 
	Naturally, we select $\bar{c}$ to be in the same component as $c$.
	While the width (measured in the direction $\frac{b-d}{|b-d|}$) of $T_1\cup T_2$ is at most $C\eta^{\frac{2}{3}}$ due to \eqref{one_direction_est2}, its height (measured in the direction $\frac{(b-d)^\perp}{|b-d|}$) becomes largest as soon as the two sets touch. In this case, it holds that $T_1=T_2$ and we may select any of the two choices for $\bar{c}$.
	We then estimate via Pythagoras that
	\begin{align}\label{one_direction_est3}
		\left|(c-\bar{c})\cdot\frac{(d-b)^\perp}{|d-b|}\right| \leq C\sqrt{(1+4\eta)^2 - (1-4\eta-C\eta^{\frac{2}{3}})^2} \leq C \sqrt{\eta^{\frac{2}{3}}} = C\eta^{\frac{1}{3}}
	\end{align}
	since $\eta \ll 1$.	This yields the desired estimate \eqref{quadratloesung}.
	
	\smallskip	
	
	\textit{Step 2b: Construction of the auxiliary function $\bar{v}$.} 
	Let $\bar{v}: \overline{E'} \to \R^2$ be continuous such that $\bar{v}\restrict{E_i}$ is a rigid body motion for every $i\in\{1,\ldots, 4\}$, and
	\begin{align}\label{design_barv}
		\bar{v}(x_1) = a,\quad \bar{v}(x_2) = \bar{b},\quad \bar{v}(x_3) =\bar{c},\quad\text{and } \bar{v}(x_4) = d.
	\end{align}
	Exactly as in the proof of Lemma \ref{lem:affine} (cf.~Case 1), we can continuously extend $\bar{v}$ to a piecewise affine function defined on all of $E$. 
	We now aim to prove the estimate \eqref{v-barv}.
	Indeed, we first observe that $\bar{v}=v_1$ on $E_1$ by design. 
	We then consider the case $i=3$. Under consideration of \eqref{abcd}, \eqref{design_barv} and \eqref{quadratloesung}, it holds that
	\begin{align*}
		\big|\big(\bar{v}(x_3)-\bar{v}(x_2)\big) - \big(v_3(x_3) - v_3(x_2)\big)\big| = |(\bar{c}-\bar{b}) - (c-b)| \leq |\bar{c}-c| + |\bar{b}-b| \leq C\eta^{\frac{1}{3}}.
	\end{align*}
	Since $v_3$ and $\bar{v}$ are both rigid body motions on the bounded set $E_3$, and $|Q-Q'|=\sqrt{2}|Qe_1-Q'e_1|$ for all $Q,Q'\in \SO(2)$, we conclude that
	\begin{align*}
		\norm{v_3-\bar{v}}_{W^{1,p}(E_3;\R^2)}\leq C \eta^{\frac{1}{3}}.
	\end{align*}
	As for $i\in\{2,4\}$, we repeat the same strategy and recall the estimates \eqref{a-a'}. Now, we have proven the desired estimate \eqref{v-barv}.
	
	We also point out that
	\begin{align}\label{u-barv}
		\norm{u-\bar{v}}_{W^{1,p}(E_i;\R^2)}\leq C \eta^{\frac{1}{3}}\quad\text{for every } i\in\{1,\ldots, 4\}
	\end{align}
	in view of \eqref{uniform_eta} and \eqref{v-barv}.
	
	\medskip
	
	\textit{Step 3: Estimating $|S_1-S_3|$ and $|R_2-R_4|$.} In the following, we differentiate between the different possible geometric of outcomes for $\bar{v}(\partial E_0)$ (in other words the polygon $a\bar{b}\bar{c}d$). 
	While some geometries (such as the case that $a\bar{b}\bar{c}d$ is a parallelogram) provide the desired estimates for $|S_1-S_3|$ and $|R_2-R_4|$, others will be excluded via the non-interpenetration of $u$.
	
	\smallskip
	
	\textit{Step 3a: The $\bar{v}(\partial E_0)$ is a parallelogram.}
	In this case, it holds that
	\begin{align}\label{parallelogram}
		d-a = \bar{c}-\bar{b}\qand \bar{b}-a = \bar{c}-d.
	\end{align}
	It is then straightforward to derive 
	\begin{align*}
		|S_1 - S_3| = \sqrt{2}|(S_3 - S_1)e_1| = \sqrt{2}|(c-b) - (d-a)| \leq\sqrt{2}(|(\bar{c}-\bar{b}) - (d-a)| + |\bar{c}-c| + |\bar{b}-b|) \leq  C\eta^{\frac{1}{3}}
	\end{align*}
	from \eqref{abcd}, and \eqref{quadratloesung}.
	We analogously conclude, under additional consideration of \eqref{a-a'}, that $|R_2-R_4| \leq C\eta^{\frac{1}{3}}$.
	Together with \eqref{fjm_inital_estimate}, these inequalities already prove the desired the estimate \eqref{luftschloss_absch} if $\eta\ll 1$.
	
	\smallskip	
	
	\textit{Step 3b: The $\bar{v}(\partial E_0)$ is not a parallelogram.}
	In this case, it holds that $\bar{b}=d$ or $\bar{c}=a$. 
	We may, due to symmetry reasons, assume without loss of generality that $\bar{b}=d$.
	Moreover, let 
	\begin{align}\label{rotated_branch}
		\bar{c}-\bar{b} = R_\ffi(a-d)
	\end{align}
	for some $\ffi\in[0,\pi]$ (the case $\ffi\in [0,-\pi]$ can be handled analogously), cf.~Figure \ref{fig:exclusion}.
	If $\ffi=\pi$, then \eqref{parallelogram} also holds and we obtain \eqref{luftschloss_absch} exactly as in Step 3a.

	What follows is a discussion of the cases $\ffi\in[0,\pi)$ where an overlap of the deformed squares $\bar{v}(E_i)$ occurs, see Figure \ref{fig:exclusion}.
	In light of Lemma \ref{lem:approx_ciarlet} applied to $M=\mathrm{int}\, E'$, $u,p$ as given, and $v=\bar{v}$, we obtain the inequality
	\begin{align*}
		|E'| \leq |\bar{v}(E')| + C_0\eta^{\frac{1}{3}},
	\end{align*}
	for a constant $C_0>0$; here we used \eqref{u-barv} and that $\bar{v}$ is a rigid body motion on each $E_i$, $i\in\{1,\ldots,4\}$.
	We further simplify this estimate to
	\begin{align}\label{ciarlet_rigid}
		4 - C_0\eta^{\frac{1}{3}} \leq |\bar{v}(E')| \leq 4.
	\end{align}
	If $\ffi = 0$, then the continuity and the design of $\bar{v}$ yields that $|\bar{v}(E')| = 2$, so that \eqref{ciarlet_rigid} yields a contradiction if $\eta\ll 1$.
	Due to monotonicity reasons, we can also exclude all cases $\ffi\in[0,\frac{\pi}{2}]$ since $|\bar{v}(E')|=3$ for $\ffi=\frac{\pi}{2}$. We shall thus assume from now on that $\ffi\in(\frac{\pi}{2},\pi)$.
	
	\begin{figure}
		\centering
		\begin{tikzpicture}[scale=1.4]
			\draw [orange](0,-1) rectangle (1,0);
			\draw [blue](-1,0) rectangle (0,1);
			\draw [ForestGreen](1,0) rectangle (2,1);
			\draw [red](0,1) rectangle (1,2);
			\draw (0.5,0.5) node {$E_0$};
			\draw [orange](0.5,-0.5) node {$E_1$};
			\draw [blue](-0.5,0.5) node {$E_2$};
			\draw [red](0.5,1.5) node {$E_3$};	
			\draw [ForestGreen](1.5,0.5) node {$E_4$};
			
			\fill (0,0) circle (1.5pt);
			\draw (0,0) node [anchor = north east] {$\small x_1$};
			\fill (0,1) circle (1.5pt);
			\draw (0,1) node [anchor = south east] {$\small x_2$};
			\fill (1,1) circle (1.5pt);
			\draw (1,1) node [anchor = south west] {$\small x_3$};
			\fill (1,0) circle (1.5pt);
			\draw (1,0) node [anchor = north west] {$\small x_4$};
			
			\draw[->] (2.5,0.5) to[out=45,in=135] (4,0.5);
			\draw (3.25,1) node {$\bar{v}$};
			
			\begin{scope}[shift={(4.75,.5)}]	
				\draw[blue] (0,0) rectangle (1,1);
				\draw[orange] (0,0) rectangle (1,-1);
				
				\draw[red] (1,0) --++(-45:1) --++(45:1) --++(135:1)--cycle;
				\draw[ForestGreen] (1,0) --++(-45:1) --++(-135:1) --++(135:1)--cycle;
				
  				\draw [->,black,domain=180:315] plot ({1+ 0.4*cos(\x)}, {0.4*sin(\x)}) node [anchor = south west]{$\ffi$};
  				\draw [pattern=north west lines, pattern color=purple!50!white] (1,0) --++ (225:1) -- (1,-1) -- cycle;
  				\draw [->,black,domain=225:270] plot ({1+ 0.7*cos(\x)}, {0.7*sin(\x)}) node [anchor = west]{$\theta$};
  				
  				\fill (0,0) circle (1.5pt);
				\draw (0,0) node [anchor = east] {$a$};
				\fill (1,0) circle (1.5pt);
				\draw (1.05,0) node [anchor = west] {$\bar{b}$};
				\draw (1,0) node [anchor = south east] {$d$};
				\fill ($(1,0)+(-45:1)$) circle (1.5pt);
				\draw ($(1,0)+(-45:1)$) node [anchor = north west] {$\bar{c}$};
				
				\draw [orange](0.5,-1.5) node {$\bar{v}(E_1)$};
				\draw [blue](0.5,1.5) node {$\bar{v}(E_2)$};
				\draw [red] (2.3,.7) node {$\bar{v}(E_3)$};
				\draw [ForestGreen] (1.7,-1.3) node {$\bar{v}(E_4)$};
			\end{scope}
		\end{tikzpicture}
		\caption{An illustration of the reference configuration $E$ and its deformed configuration under the continuous map $\bar{v}$. Here, the points $a\bar{b}\bar{c}d$ do not form a parallelogram, thus leading to an overlap of the deformed squares $\bar{v}(E_1)$ and $\bar{v}(E_4)$ of at least the hatched triangle with $\theta = \pi - \ffi$.}\label{fig:exclusion}
	\end{figure}
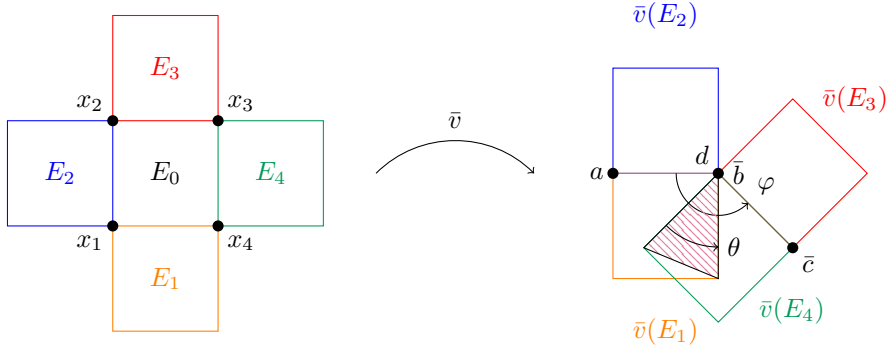
	
	The region $D=\bar{v}(E_1)\cap\bar{v}(E_4)$ in which we observe an overlap has at least measure $\frac{1}{2}\sin (\pi-\ffi) = \frac{1}{2}\sin\ffi$; in other words, $|\bar{v}(E')|\leq 4 - \frac{1}{2}\sin \ffi$.
	Let $\eta$ be small enough that $C_0\eta^{\frac{1}{3}}< \frac{1}{2}$, then there exists $\ffi_0\in (\frac{\pi}{2},\pi)$ such that $\frac{1}{2}\sin\ffi_0 = C_0\eta^{\frac{1}{3}}$. 
	We then derive from \eqref{ciarlet_rigid} the contradiction
	\begin{align}\label{exclusion}
		4 - C_0\eta^{\frac{1}{3}} \leq |\bar{v}(E')| \leq 4-\frac{1}{2}\sin \ffi < 4 - \frac{1}{2}\sin \ffi_0 = 4 - C_0\eta^{\frac{1}{3}},
	\end{align}
	for every $\ffi\in (\frac{\pi}{2},\ffi_0)$.
	In the cases $\ffi\in[\ffi_0, \pi)$, it holds that $0<\sin\ffi \leq \sin \ffi_0$ and thus,
	\begin{align}\label{S3-S1}
		|S_3 - S_1| &= \sqrt{2}|(S_3-S_1)e_1| \leq \sqrt{2}(|(\bar{c}-\bar{b}) - (d-a) | + |\bar{b}-b| + |\bar{c}-c|)\nonumber \\
		&= \sqrt{2}|(-R_\ffi - \Id)(d-a)|  + C\eta^{\frac{1}{3}}\leq C\big(|-R_\ffi-\Id| + \eta^{\frac{1}{3}}\big) \leq C \big(|\sin \ffi| + \big| -1 - \cos\ffi\big| + \eta^{\frac{1}{3}}\big)\nonumber \\
		&\leq C \big(|\sin \ffi| + \big| 1 - |\cos\ffi|\big| + \eta^{\frac{1}{3}}\big) \leq  C \big(|\sin \ffi| + \big| 1- \sqrt{1-\sin^2\ffi}\big| + \eta^{\frac{1}{3}}\big)\nonumber \\
		&\leq  C \big(|\sin \ffi| + \sqrt{| 1- 1+\sin^2\ffi|} + \eta^{\frac{1}{3}}\big) \leq C \big(|\sin \ffi_0| +\eta^{\frac{1}{3}}\big) \leq C\eta^{\frac{1}{3}}
	\end{align}
	under consideration of \eqref{rotated_branch}, \eqref{quadratloesung}, and the H\"older-continuity of the square-root.
	Analogously, obtain the same estimate for $|R_2-R_4|$ if we take \eqref{a-a'} into account.
	We now set $S:=S_1$ and $R:=R_2$ and thus, obtain the the desired estimate \eqref{luftschloss_absch} in light of \eqref{fjm_inital_estimate}, \eqref{uniform_eta} for $\eta\ll 1$.
	
	\medskip
	
	\textit{Step 4: The scalar product estimate.} We now prove \eqref{approx_scalar} with exponent $\frac{1}{3}$ on the right-hand side. 
	Indeed, if $a\bar{b}\bar{c}d$ is not a parallelogram, we may assume (as in Step 3b) that $\bar{b} = d$ and obtain that
	\begin{align*}
		Se_1 \cdot Re_1 &= Se_1\cdot (a'-b')^\perp =  Se_1\cdot(a'-a)^\perp + Se_1\cdot(b-b')^\perp + Se_1\cdot(\bar{b}-b)^\perp + Se_1 \cdot (a-\bar{b})^\perp\\
		&\geq - C\eta^{\frac{1}{3}} + Se_1 \cdot (a-\bar{b})^\perp = -C \eta^{\frac{1}{3}} - (d-a)\cdot (d-a)^\perp = -C \eta^{\frac{1}{3}}
	\end{align*}
	due to \eqref{a-a'}, \eqref{quadratloesung}.
	If $a\bar{b}\bar{c}d$ is a parallelogram and 
	\begin{align}\label{barv_pos_det}
		Se_1 \cdot (a-\bar{b})^\perp =\nabla \bar{v}\restrict{E_1} e_1 \cdot \nabla \bar{v} \restrict{E_2}e_1 >0,
	\end{align}
	then it holds that
	\begin{align}\label{barv_parallelogram}
		\begin{split}
			Se_1\cdot Re_1 \geq - C\eta^{\frac{1}{3}} + Se_1 \cdot (a-\bar{b})^\perp \geq -C\eta^{\frac{1}{3}} + \nabla \bar{v}\restrict{E_1} e_1 \cdot \nabla \bar{v} \restrict{E_2}e_1 > -C\eta^{\frac{1}{3}}.
		\end{split}
	\end{align}
	
	Lastly, we deal with the case that \eqref{barv_pos_det} is not satisfied. If $Se_1 \cdot (a-\bar{b})^\perp=0$, then there is nothing to prove. 
	Otherwise, let $\ffi\in(-\frac{\pi}{2},\frac{\pi}{2})$ be such that $\bar{R}e_2 := \bar{v}\restrict{E_2}e_2 = -R_{\ffi}\bar{v}\restrict{E_1}e_2$ and note that for $\ffi=0$ it holds that $|\bar{v}(E')|=1$, which causes a contradiction to \eqref{ciarlet_rigid} for $\eta\ll 1$. Henceforth, we shall only cover $\ffi\in(0,\frac{\pi}{2})$ due to symmetry reasons.
	In these cases, we observe an overlap of $\bar{v}(E_1)$ and $\bar{v}(E_4)$ with at least measure $\frac{1}{2}\sin(\frac{\pi}{2}-\ffi) = \frac{1}{2}\cos \ffi$.
	We now proceed analogously to Step 3b to find some $\ffi_0\in (0,\frac{\pi}{2})$ such that $\frac{1}{2}\cos\ffi_0 = C_0\eta^{\frac{1}{3}}$ with $C_0$ as in \eqref{ciarlet_rigid}.
	Every geometry resulting from $\ffi\in(0,\ffi_0)$ can then be excluded as in \eqref{exclusion}. On the other hand, for $\ffi\in(\ffi_0,\frac{\pi}{2})$, it holds that
	\begin{align*}
		Se_1 \cdot (a-\bar{b})^\perp = \det(Se_1|\bar{R}e_2)  = - \sin (\frac{\pi}{2}-\ffi) = - \cos \ffi \geq -\cos \ffi_0 = - 2C_0\eta^{\frac{1}{3}},
	\end{align*}
	and hence $Se_1\cdot Re_1 \geq - C\eta^{\frac{1}{3}}$ similarly to \eqref{barv_parallelogram}.

	\medskip
	
	\textit{Step 5: Improving the estimate.} In light of Step 3, we find that either $\bar{v}(\partial E_0)$ forms a parallelogram, or it holds that $\bar{b}=d$ or $\bar{c}=a$ together with \eqref{rotated_branch} for $\ffi\geq\frac{\pi}{2}$.
	We find in either case that $|\bar{c}-a|\geq\sqrt{2}$ or $|\bar{b}-d|\geq\sqrt{2}$ due to the parallelogram identity or the choice $\ffi\geq\frac{\pi}{2}$.
	In light of \eqref{quadratloesung}, it holds that $|c-a|\geq l$ or $|b-d|\geq l$ for some constant length $l>0$ if $\eta$ is sufficiently small. 
	Let us assume that the latter inequality is true.
	We then repeat the procedure in Steps 2 - 4 to improve the estimates \eqref{luftschloss_absch} and \eqref{approx_scalar}.
	Precisely, we first construct a new auxiliary function $\tilde{v}: E\to \R^2$ similar to $\bar{v}$ as in Step 2, for which we first need to find two points $\tilde{b}\in a+\Scal^1$ and $\tilde{c}\in d+\Scal^1$ such that $|\tilde{b}-\tilde{c}|=1$ and 
	\begin{align}\label{quadratloesung2}
		|\tilde{b} - b| + |\tilde{c}-c| \leq C\eta^{\frac{1}{2}}.
	\end{align}
	To this end, first choose any $\tilde{c}\in d+\Scal^1$ such that $|\tilde{c} - c|<4\eta$ and continue as in the proof of \eqref{one_direction_est1}-\eqref{one_direction_est3}, exploiting the new estimate $|b-d|>l$ for some constant $l>0$ independent of $\eta$.
	The rest of the proof works analogously to the Steps 2b - 4.
	
	\medskip
	
	\textit{Step 6: The case $\mu \in (0,1)$.} The general strategy is quite similar. 
	Step 1 stays essentially the same with the difference being $b\in \overline{A(a,\mu-4\eta,\mu+4\eta})$ and $c\in \overline{A(d,\mu-4\eta,\mu+4\eta})$ instead of \eqref{bc_annulus}.
	
	Our next task is to find auxiliary points $\bar{b}\in a+\mu \Scal^1$ and $\bar{c}\in d+\mu\Scal^1$ such that $|\bar{b}-\bar{c}|=1$ and \eqref{quadratloesung2} similar to Step 2a; 
	the construction of the auxiliary function $\bar{v}:E\to \R^2$ with the help of these points works exactly as before.
	Since $\mu<1$ we always find that $|b-d|>l$ for some $l>0$ independent of $\eta$ if $\eta\ll 1$ is small enough. We may hence argue as in Step 5 to establish \eqref{quadratloesung2}.
	Setting $S:=S_1$ and $R:=R_2$, we then need to prove estimates for $|S - S_3|$, $|R - R_4|$ and the scalar product $Se_1\cdot Re_1$ similarly to Steps 3 - 4.
	
	To this end, we seek to exclude invalid geometries for $\bar{v}(\partial E_0)$ with the help of the approximate Ciarlet-Ne{\v c}as condition from Lemma \ref{lem:approx_ciarlet}. 
	The estimate \eqref{ciarlet_rigid} changes in this case to
	\begin{align}\label{ciarlet_rigid2}
		2(1+\mu^2) - C_0\eta^{\frac{1}{2}} \leq |\bar{v}(E')| \leq 2(1+\mu^2).
	\end{align}
	The procedure to produce a contradiction to \eqref{ciarlet_rigid2} for $\eta\ll 1$ is now very similar to what we presented in Steps 3 - 4 and is based again on finding suitable rectangles that emerge from overlapping two neighboring rigid squares. 
	Since the general methodology is virtually the same, we shall only explain one scenario in detail for illustration, see also Figure \ref{fig:exclusion2}.
	Let us assume that $\bar{b} - a = R_\ffi (d-a)^\perp$ for $\ffi\in [0,\frac{\pi}{2})$, then we find that the intersection of $\bar{v}(E_3)$ and $\bar{v}(E_4)$ contains at least a triangle of measure $\frac{1}{2}\mu^2\sin(\theta)$ with $\theta=\frac{\pi}{2}-\ffi$. 
	Choose now for sufficiently small $\eta$ an angle $\ffi_0$ in such a way that $\frac{1}{2}\mu^2\cos\ffi_0 = C_0\eta^{\frac{1}{2}}$ so that the monotonicity of the cosine on $[0,\frac{\pi}{2}]$ generates a contradiction to \eqref{ciarlet_rigid2} for every $\ffi\in (0,\ffi_0)$. 
	In the remaining cases, a direct calculation in the spirit of \eqref{S3-S1} shows that
	\begin{align*}
		|Se_1 - S_3e_1| \leq C(\cos \ffi +\eta^{\frac{1}{2}})\leq C(\cos \ffi_0 +\eta^{\frac{1}{2}}) \leq C\eta^{\frac{1}{2}}
	\end{align*}
	and similarly for $|Re_2-R_4e_2|$, which proves the desired estimate \eqref{luftschloss_absch} when combined with \eqref{fjm_inital_estimate}. 
	The scalar product estimate \eqref{approx_scalar}, cf.~Step 4, can be handled analogously.
	
	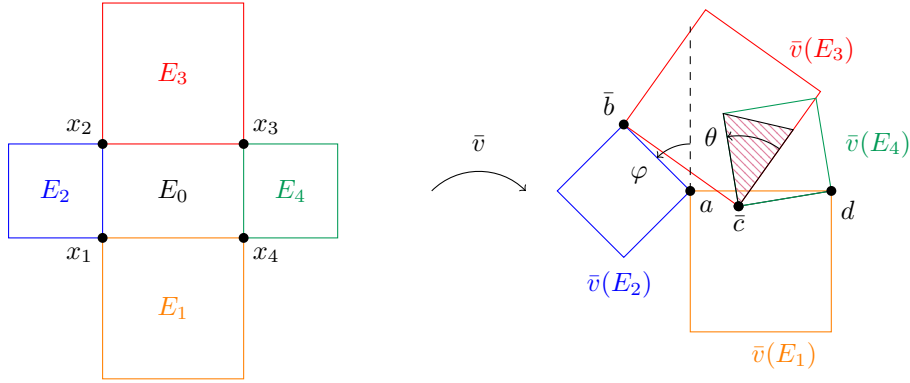
\begin{figure}
		\centering
		\begin{tikzpicture}[scale=1.25]
			\draw [orange](0,-1.5) rectangle (1.5,0);
			\draw [blue](-1,0) rectangle (0,1);
			\draw [ForestGreen](1.5,0) rectangle (2.5,1);
			\draw [red](0,1) rectangle (1.5,2.5);
			\draw (.75,0.5) node {$E_0$};
			\draw [orange](.75,-0.75) node {$E_1$};
			\draw [blue](-0.5,0.5) node {$E_2$};
			\draw [red](.75,1.75) node {$E_3$};	
			\draw [ForestGreen](2,0.5) node {$E_4$};
			
			\fill (0,0) circle (1.5pt);
			\draw (0,0) node [anchor = north east] {$\small x_1$};
			\fill (0,1) circle (1.5pt);
			\draw (0,1) node [anchor = south east] {$\small x_2$};
			\fill (1.5,1) circle (1.5pt);
			\draw (1.5,1) node [anchor = south west] {$\small x_3$};
			\fill (1.5,0) circle (1.5pt);
			\draw (1.5,0) node [anchor = north west] {$\small x_4$};
			
			\draw[->] (3.5,0.5) to[out=45,in=135] (4.5,0.5);
			\draw (4,1) node {$\bar{v}$};
			
			\begin{scope}[shift={(6.25,.5)}]	
				\draw[orange] (0,0) rectangle (1.5,-1.5);				
				\draw[blue] (0,0) --++ (135:1) --++ (225:1) --++ (-45:1) -- cycle;
				\coordinate (b) at (135:1);
				\coordinate (d) at (1.5,0);
				
				\path [name path = circle1] (b) circle (1.5);
				\path [name path = circle2] (d) circle (1);
				\path [name intersections={of=circle1 and circle2}];
				\draw (d) -- (intersection-2);
				\fill[black] (intersection-2) circle (1.5pt);
				\pgfmathanglebetweenpoints{\pgfpointanchor{b}{center}}{\pgfpointanchor{intersection-2}{center}}
				\edef\angleBC{\pgfmathresult}
				\pgfmathanglebetweenpoints{\pgfpointanchor{intersection-2}{center}}{\pgfpointanchor{d}{center}}
				\edef\angleCD{\pgfmathresult}
				\draw [red] (135:1) --++({\angleBC+90}:1.5)--++(\angleBC:1.5) -- (intersection-2) -- cycle;
				\draw [ForestGreen] (intersection-2) --++ (\angleCD+90:1) --++ (\angleCD:1) -- (d) -- cycle;
				
				\draw [dashed] (0,0) -- (0,1.75);
				\draw [->,black,domain=90:135] plot ({0.5*cos(\x)}, {0.5*sin(\x)}) node [anchor = north east] {$\ffi$};
				
				\draw[pattern=north west lines, pattern color=purple!50!white] (intersection-2) --++ (\angleBC+90:1) -- ($(intersection-2)+(\angleCD+90:1)$) -- cycle;
				\begin{scope}[shift={(intersection-2)}]
					\draw [->,black,domain={\angleBC-270}:{\angleCD+90}] plot ({0.75*cos(\x)}, {0.75*sin(\x)}) node [anchor=east] {$\theta$};
				\end{scope}
				
				\fill (0,0) circle (1.5pt) node [anchor = north west] {$a$};
				\fill (d)  circle (1.5pt) node [anchor = north west] {$d$};
				\fill (intersection-2) circle (1.5pt) node [anchor = north] {$\bar{c}$};
				\fill (135:1) circle (1.5pt) node[anchor=south east] {$\bar{b}$};
				\draw [orange](1,-1.75) node {$\bar{v}(E_1)$};
				\draw [blue](-.75,-1) node {$\bar{v}(E_2)$};
				\draw [red] (1.4,1.5) node {$\bar{v}(E_3)$};
				\draw [ForestGreen] (2,.5) node {$\bar{v}(E_4)$};
			\end{scope}
		\end{tikzpicture}
		\caption{An illustration of the reference configuration $E$ and its deformed configuration under the continuous map $\bar{v}$. Here, the points $a\bar{b}\bar{c}d$ do not form a parallelogram, thus leading to an overlap of the deformed squares $\bar{v}(E_3)$ and $\bar{v}(E_4)$ of at least the hatched triangle with $\theta=\frac{\pi}{2}-\ffi$.}\label{fig:exclusion2}
	\end{figure}
	
	Finally, we shall point out that quantity $\delta$ in the formulation of the statement is simply a multiple (dependent only on $p$) of $\eta$.
\end{proof}

In the following, we shall briefly remark on the dependence of the constants and rotations in Lemma \ref{lem:luftschloss} on uniform scalings and translations.

\begin{remark}[Scaling analysis]\label{rem:scaling}
	Let $\rho>0$ be arbitrary, then for every $u\in W^{1,p}(\rho E;\R^2)$ that satisfies the Ciarlet-Ne\v{c}as condition \eqref{ciarlet_necas} on $\rho E'$ and for which
	$\norm{\dist(\nabla u,\SO(2))}_{L^p(\rho E')}<\delta_0 \rho^{\frac{2}{p}}$, there exist $R,S\in \SO(2)$ such that
	\begin{align*}
		\norm{\nabla u - S}_{L^p(\rho E_1\cup \rho E_3,\R^{2\times 2})} + \norm{\nabla u - R}_{L^p(\rho E_2\cup \rho E_4,\R^{2\times 2})} \leq C\rho^{\frac{1}{p}} \norm{\dist(\nabla u,\SO(2))}_{L^p(\rho E')}^\frac{1}{2}
	\end{align*}
	and $Re_1 \cdot Se_1 \geq - C \rho^{-\frac{1}{p}}\norm{\dist(\nabla u,\SO(2))}_{L^p(\rho E)}^\frac{1}{2}$. 
	Here, the constants $\delta_0$ and $C$ (which depend only on $p$) and the rotations $S,R$ are exactly the same as in the case $\rho=1$ in Lemma \ref{lem:luftschloss} and are invariant under translations of the domain. 
	This is a direct consequence of a change of variables in the occurring integrals.
\end{remark}

Next, we prove a Poincar\'e inequality for open sets of checkerboard structure. 
The prove is essentially the same as in the case of connected open sets and is based on an argument via contradiction.
This way, however, we do not obtain an explicit dependence of the emerging constant.

\begin{lemma}[Poincar\'e-inequality for sets with path-connected closure]\label{lem:poincare}
	Let $p>2$, $N\in\N$ and let $U_1,\ldots,U_N$ be bounded Lipschitz domains such that the closure of $U:=U_1\cup\ldots\cup U_N$ is path-connected. 
	Then there exists a constant $C>0$ with the following property: 
	For every $u\in W^{1,p}(U;\R^2)\cap C^0(\overline{U};\R^2)\cap L^p_0(U;\R^2)$, it holds that
	\begin{align}\label{poincare0}
		\norm{u}_{L^p(U;\R^2)} \leq C \norm{\nabla u}_{L^p(U;\R^{2\times 2})}.
	\end{align}
\end{lemma}
\begin{proof}
	We argue via contradiction and assume that for every $j\in\N$ there exists a sequence $(\tilde{u}_j)_{j}\subset W^{1,p}(U;\R^2)\cap C^0(\overline{U};\R^2)\cap L^p_0(U;\R^2)$ and 
	\begin{align}\label{u_j_contra}
		\norm{\tilde{u}_j}_{L^p(U;\R^2)} > j\norm{\nabla \tilde{u}_j}_{L^p(U;\R^{2\times 2})}.
	\end{align}
	We now define $u_j:=\norm{\tilde{u}_j}_{L^p(U;\R^2)}^{-1}\tilde{u}_j$ and obtain that $(u_j)_{j}$ is bounded in $W^{1,p}(U;\R^2)$. 
	In particular, each $u^n_j:=u_j\restrict{U_n}$ for $n\in \{1,\ldots,N\}$ is bounded in $W^{1,p}(U_n;\R^2)$. 
	We can thus find $u^{n}\in W^{1,p}(U_n;\R^2)$ such that 
	\begin{align}\label{limits_u_j}
		\begin{split}
			u^{n}_j\weakly u^{n} &\text{ in } W^{1,p}(U_n;\R^2),\\
			u^{n}_j\to u^{n} &\text{ in } C^0(\overline{U_n};\R^2),
		\end{split}
	\end{align}
	and we set
	\begin{align*}
		u: U\to \R^2, x\mapsto u^{n}(x)\text{ if } x\in U_n.
	\end{align*}
	In light of the weak convergence of $(u_j^{n})_j$, \eqref{u_j_contra} and the lower-semicontinuity of the norm, we find that $u^{k,i}$ is constant on $U_n$ with value, say $d^{n}\in\R^2$.
	Since each $u_j$ is continuous on $\overline{U_n}$, and $\overline{U}$ is path-connected, we find that $d^{n}=d^{m}=d\in\R^2$ for every $n,m\in\{1,\ldots,N\}$.
	Hence, the vanishing mean value of $u$ on $U$ yields that
	\begin{align*}
		0 = \int_{U} u(x) \dd x = \sum_{n=1}^N |U_n| d,
	\end{align*}
	which implies that $d=0$ and $u=0$ on all of $U$.
	On the other hand, we then find a contradiction to $\norm{u}_{L^p(U;\R^2)}=1$ as $u$ is the strong limit of $(u_j)_j$ on $L^p(U;\R^2)$, recall \eqref{limits_u_j} and $\norm{u_j}_{L^p(U;\R^2)}=1$. 
\end{proof}

As we have pointed out above, it is not clear if or how the constant $C$ in \eqref{poincare0} depends on the domain.
We address this issue in the next lemma under additional assumptions and slight change in the domains of integration.
This result serves as the second key ingredient in the characterization of the macroscopic deformation behavior.

\begin{lemma}[Poincar\'e estimate for checkerboard structures]\label{lem:poincare2}
	Let $p>2$, $U'\Subset U$ be bounded Lipschitz domains, and $M>0$. 
	There exists a constant $C>0$ independent of $\eps$ such that for every $u \in W^{1,p}(U;\R^2)$ with
	\begin{align}\label{mean_value_rig}
		\int_{U'\cap\eYstiff} u \dd x=0
	\end{align}
	and
	\begin{align}\label{domain_est}
		\norm{u}_{L^p(U\cap\eYstiff;\R^2)}\leq M \norm{u}_{L^p(U'\cap\eYstiff;\R^2)}
	\end{align}
	it holds that
	\begin{align*}
		\norm{u}_{L^p(U'\cap\eYstiff;\R^2)}\leq C \norm{\nabla u}_{L^p(U\cap\eYstiff;\R^{2\times 2})}.
	\end{align*}
\end{lemma}

Before we prove this result, we first cover an alternative auxiliary extension-type result, in which the deformations and the gradients can be estimated separately, cf.~\cite{ACDP92}.

\begin{lemma}[Approximate extension result for checkerboard structures]\label{lem:extension}
	Let $p>2$, $U'\Subset U\subset \R^2$ be a bounded open sets and $\eps>0$ sufficiently small.
	Moreover, let
	\begin{align}\label{B_r}
		B_r:=\bigcup_{e\in I} Y\cap(e + B(0,r))\text{ with } I=\{0,\lambda,1\}^2\text{ and } r<\frac{1}{4}\min\{\lambda,1-\lambda\},
	\end{align}
	as well as its $Y$-periodic extension.
	There exists a constant $C>0$ independent of $\eps, U, U'$ and a linear and continuous operator 
	$L_r: W^{1,p}(U\cap\eYstiff;\R^2)\cap C^0(\overline{U\cap\eYstiff};\R^2) \to W^{1,p}(U';\R^2)$ such that $L_r u = u$ a.e.~in $U'\cap \eYstiff\setminus \eps B_r$ and
	\begin{align*}
		\norm{L_r u}_{L^p(U';\R^2)} &\leq C\norm{u}_{L^p(U\cap \eYstiff;\R^2)},\\
		\norm{\nabla(L_r u)}_{L^p(U';\R^{2\times 2})} &\leq C\norm{\nabla u}_{L^p(U\cap \eYstiff;\R^{2\times 2})}
	\end{align*}
	for every $u\in W^{1,p}(U\cap\eYstiff;\R^2)\cap C^0(\overline{U\cap\eYstiff};\R^2)$.
\end{lemma}
\begin{proof}
	This proof is subdivided into two main arguments. For $V\subset\R^2$ with $U'\Subset V\Subset U$ we first find an operator 
	$\hat{L}_r: W^{1,p}(U\cap\eYstiff;\R^2)\cap C^0(\overline{U\cap\eYstiff};\R^2) \to W^{1,p}(V\setminus (\eps B_r\cap \eYsoft);\R^2)$ such that $\hat{L}_r u = u$ a.e.~in $V\cap \eYstiff$ and
	\begin{align*}
		\norm{\hat{L}_r u}_{L^p(V\setminus (\eYsoft\cap \eps B_r);\R^2)} &\leq C\norm{u}_{L^p(U\cap \eYstiff;\R^2)},\\
		\norm{\nabla(\hat{L}_r u)}_{L^p(V\setminus (\eYsoft \cap \eps B_r);\R^{2\times 2})} &\leq C\norm{\nabla u}_{L^p(U\cap \eYstiff;\R^{2\times 2})}
	\end{align*}
	for every $u\in W^{1,p}(U\cap\eYstiff;\R^2)\cap C^0(\overline{U\cap\eYstiff};\R^2)$.
	
	For such functions $u$, we then define the desired operator as
	$$L_r(u):= \tilde{L}_r\big(\hat{L}_r(u)\restrict{V\cap \eps (Y\setminus B_r)}\big)$$
	where the linear and continuous map $\tilde{L}_r: W^{1,p}(V\cap \eps (Y\setminus B_r);\R^2) \to W^{1,p}(U';\R^2)$ is taken as in \cite[Theorem 2.1]{ACDP92} for $E=Y\setminus B_r$. 
	We shall now detail the construction of $\hat{L}$.\medskip
	
	\textit{Step 1: A preliminary construction on the first unit cell.}
	Let $Z$ be as in \eqref{cellZ}, and recall the analogous definition of $\Zstiff$.
	Moreover, we define $\Ysoftr:=\Ysoft\setminus B_r$.
	The next task is to choose suitable sets $V_1,\ldots,V_N$ for $N\geq 8$ that cover the compact set $\overline{\Ysoft}$ in a particular way. 
	Each of these sets shall overlap with at most one of the eight straight components of $\partial \Ysoft$, see e.g., Figure \ref{fig:cover_Ysoft}.
	\begin{figure}
		\centering
		\begin{tikzpicture}[scale=2]
			\def\length{.6}
			\pgfmathsetmacro\lengthb{1-\length}
			\begin{scope}
				\clip (0,\length) rectangle (\length,1);
	  			\fill [red](\length,\length) circle (0.1);
	  			\fill [red](\length,1) circle (0.1);
	  			\fill [red](0,1) circle (0.1);
	  			\fill [red](0,\length) circle (0.1);
	  		\end{scope}
	  		\begin{scope}
				\clip (\length,0) rectangle (1,\length);
	  			\fill [red](\length,\length) circle (0.1);
	  			\fill [red](\length,0) circle (0.1);
	  			\fill [red](1,\length) circle (0.1);
	  			\fill [red](1,0) circle (0.1);
	  		\end{scope}

			\draw (0,0) rectangle (1,1);
			\draw [fill=black!25!white] (\length,\length) rectangle (1,1);
			\draw [fill=black!25!white] (0,0) rectangle (\length,\length);
			\draw [fill=black!25!white] ($(\length,\length)-(1,0)$) rectangle (0,1);
			\draw [fill=black!25!white] ($(\length,\length)-(0,1)$) rectangle (1,0);
			\draw [fill=black!25!white] (0,1) rectangle (\length,\length+1);
			\draw [fill=black!25!white] (1,0) rectangle (\length+1,\length);

	  		\draw (\length/2,\length) ellipse (\length/2 and \length/2);
	  		\draw ($(0,\length)+(0,\lengthb/2)$) ellipse (\lengthb/2 and \lengthb/2);
	  		\draw (\length/2,1) ellipse (\length/2 and \length/2);
	  		\draw ($(\length,\length)+(0,\lengthb/2)$) ellipse (\lengthb/2 and \lengthb/2);
	  
	  		\draw (\length,\length/2) ellipse (\length/2 and \length/2);
	  		\draw ($(\length,0)+(\lengthb/2,0)$) ellipse (\lengthb/2 and \lengthb/2);
	  		\draw (1,\length/2) ellipse (\length/2 and \length/2);
	  		\draw ($(\length,\length)+(\lengthb/2,0)$) ellipse (\lengthb/2 and \lengthb/2);
	  		
			\draw [thick,blue] (0,0) rectangle (1,1);
			\draw (1,1) node [blue, anchor=south west] {$Y$};
		\end{tikzpicture}
		\caption{The unit cell $Y$ (indicated in blue) and its immediate neighboring stiff components. The eight ellipses cover the soft part $\Ysoft$ in such a way that they overlap with exactly one straight piece of $\partial \Ysoft$. The red quarter-circles describe the set $\Ysoft\cap B_r$ in the unit cell with $B_r$ as in \eqref{B_r}.}\label{fig:cover_Ysoft}
	\end{figure}
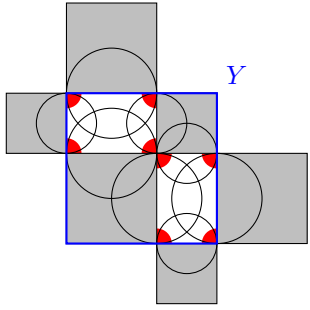	
	Then, there exists a partition of unity $(\ffi_i)_i$ with $\ffi_i\in C_c^{\infty}(V_i;\R^2)$ for $i=1,\ldots,N$ and $\sum_{i=1}^N\ffi_i = 1$ on $\bigcup_{i=1}^NV_i$.
	With standard mirroring techniques (cf. \cite[Chapter 5]{AdF03}) performed on each of the sets $V_1,\ldots, V_N$, we can thus find a linear and continuous operator $L_r^{(1)}: W^{1,p}(\Zstiff;\R^2)\cap C^0(\overline{\Zstiff};\R^2)\to W^{1,p}(\Zstiff\cup\Ysoftr;\R^2)$ such that
	\begin{align}
		L_r^{(1)} u &= u\quad \text{a.e.~in }\Zstiff\nonumber\\
		\norm{L_r^{(1)} u}_{L^p(\Zstiff\cup\Ysoftr;\R^2)} &\leq C(\lambda,p,r)\norm{u}_{L^p(\Zstiff;\R^2)}\nonumber\\
		\norm{\nabla(L_r^{(1)} u)}_{L^p(\Zstiff\cup\Ysoftr;\R^{2\times 2})} &\leq C(\lambda,p,r)\norm{u}_{W^{1,p}(\Zstiff;\R^{2\times 2})}\label{bad_ext}
	\end{align}
	for every $u\in W^{1,p}(\Zstiff;\R^2)\cap C^0(\overline{\Zstiff};\R^2)$.
	We shall point out that not all of $\Zstiff$ is needed to obtain an extension on $\Ysoftr$ but it keeps the notation easier later on.	
	
	\medskip
	
	\textit{Step 2: Improvement on the first unit cell.} 
	The next objective is to construct another extension operator that allows better estimates for the gradients than in \eqref{bad_ext}.
	This step can be handled similarly to the proof of \cite[Lemma 2.6]{ACDP92}.
	We define $L_r^{(2)}: W^{1,p}(\Zstiff;\R^2)\cap C^0(\overline{\Zstiff};\R^2)\to W^{1,p}(\Zstiff\cup\Ysoftr;\R^2)$ by setting
	\begin{align*}
		L_r^{(2)} u := L_r^{(1)}\big(u - (u)_{\Zstiff}\big) + (u)_{\Zstiff},\quad u\in W^{1,p}(\Zstiff;\R^2)\cap C^0(\overline{\Zstiff};\R^2)
	\end{align*}
	with $(u)_{\Zstiff}=\int_{\Zstiff}u\dd x$ and $L_r^{(1)}$ as in Step 1. 
	Let $u\in W^{1,p}(\Zstiff;\R^2)\cap C^0(\overline{\Zstiff};\R^2)$ be arbitrary. Since $L_r^{(1)}$ is a linear extension operator, it holds that $L_r^{(2)}u=u$ on $\Zstiff$.
	As for the $L^p$-estimate of the gradient, we use the properties of $L_r^{(1)}$, and invoke Lemma \ref{lem:poincare} for $U=\Zstiff$ to compute that
	\begin{align*}
		\int_{\Zstiff\cup\Ysoftr}|\nabla(L_r^{(1)}u)|^p &= \int_{\Zstiff\cup\Ysoftr}|\nabla\big(L_r^{(1)}(u-(u)_{\Zstiff})\big)|^p \dd x \\
		&\leq C(\lambda,p,r)\Big(\int_{\Zstiff}|u-(u)_{\Zstiff}|^p \dd x + \int_{\Zstiff}|\nabla u|^p \dd x\Big)\\
		&\leq C(\lambda,p,r)\int_{\Zstiff}|\nabla u|^p \dd x,
	\end{align*}
	where the constant $C(\lambda,p,r)$ may change from line to line.
	The estimate
	\begin{align*}
		\int_{\Zstiff\cup\Ysoftr}|L_r^{(1)}u|^p \leq C(\lambda,p,r)\int_{\Zstiff}|u|^p \dd x
	\end{align*}
	can be acquired exactly as in the proof of \cite[Lemma 2.6]{ACDP92}.
	
	\medskip
	
	\textit{Step 3: Extension on large domains and scaling analysis.} The rest of the proof can be executed exactly as in the Steps 2 \& 3 of the proof of Lemma \ref{lem:extension2}.		
\end{proof}

We are now positioned to prove the Poincar\'e-type estimate in Lemma \ref{lem:poincare2}.

\begin{proof}[Proof of Lemma \ref{lem:poincare2}]
	We argue via contradiction and suppose that for every $j\in\N$ there exists $u_j:=u_{\eps_j}\in W^{1,p}(U;\R^2)$ satisfying \eqref{mean_value_rig} and \eqref{domain_est} for $\eps=\eps_j$, and
	\begin{align}\label{est_contra}
		\norm{u_j}_{L^p(U'\cap\jYstiff;\R^2)}>j\norm{\nabla u_j}_{L^p(U\cap\jYstiff;\R^{2\times 2})}.
	\end{align}\smallskip
	
	\textit{Step 1: Rescaling and extending $u_j$.} We now normalize $u_j$ by introducing
	\begin{align}\label{normalization}
		v_j:= \norm{u_j}_{L^p(U'\cap\jYstiff;\R^2)}^{-1}u_j\quad \text{on }U,
	\end{align}
	and observe that $v_j$ still satisfies \eqref{mean_value_rig} for all $j\in\N$. Moreover, it holds that
	\begin{align}\label{v_j_est}
		\norm{v_j}_{L^p(U\cap\jYstiff;\R^2)} \leq M\qand \norm{\nabla v_j}_{L^p(U\cap\jYstiff;\R^{2\times 2})}< \frac{1}{j}
	\end{align}
	due to \eqref{domain_est} and \eqref{est_contra}.
	
	Now, choose $U''\Subset U'\Subset U$, where the set $U''$ is to be specified later, and apply Lemma \ref{lem:extension} for the pair of sets $U''\Subset U$ and sufficiently small $\eps$. 
	This way, we find $\bar{v}_j:=L_{r}v_j\in W^{1,p}(U'';\R^2)$ with the properties
	\begin{align}\label{bar_v_j}
		\begin{split}
		\bar{v}_j = v_j &\text{ a.e.~in } U''\cap\jYstiff\setminus \eps_j B_r,\\
		\norm{\bar{v}_j}_{L^p(U'';\R^2)} &\leq C(\lambda,p,r)\norm{v_j}_{L^p(U \cap \jYstiff;\R^2)},\\
		\norm{\nabla\bar{v}_j}_{L^p(U'';\R^{2\times 2})} &\leq C(\lambda,p,r)\norm{\nabla v_j}_{L^p(U\cap\jYstiff ;\R^{2\times 2})}.
		\end{split}
	\end{align}
	Combining \eqref{v_j_est} with \eqref{bar_v_j} then produces
	\begin{align}\label{bar_v_j2}
		\norm{\bar{v}_j}_{L^p(U'';\R^2)}\leq C(\lambda,p,r)M&\qand\norm{\nabla\bar{v}_j}_{L^p(U'';\R^{2\times 2})} < C(\lambda,p,r)\frac{1}{j}.
	\end{align}
	
	\smallskip
	
	\textit{Step 2: Asymptotic behavior of $(\bar{v}_j)_j$.}
	In view of \eqref{bar_v_j2}, we find that there exists a (non-relabeled) subsequence of $(\bar{v}_j)_j$ such that
	\begin{align}\label{bar_v_j_convergence}
		\bar{v}_j \to d \text{ in }W^{1,p}(U'';\R^2) \qand \bar{v}_j \to d \text{ in }C^0(U'';\R^2).
	\end{align}
	for some constant vector $d\in\R^2$. In this final step, we prove that $|d|$ is close to $1$ and close to $0$ if $|U'\setminus U''|$ and $r$ are sufficiently small, which is a contradiction.
	
	\textit{Step 2a: The length $|d|$ is small.}
	First, we exploit \eqref{mean_value_rig} to compute that
	\begin{align}
		0 &= \int_{U'\cap\jYstiff} v_j \dd x = \int_{U''\cap\jYstiff} v_j \dd x + \int_{(U'\setminus U'') \cap\jYstiff} v_j \dd x\nonumber\\
		&=\int_{U''\cap\jYstiff\setminus \eps_j B_r} \bar{v}_j \dd x + \int_{U''\cap\jYstiff\cap \eps_j B_r} v_j \dd x + \int_{(U'\setminus U'') \cap\jYstiff} v_j \dd x\nonumber\\
		&=\int_{U''\cap\jYstiff} \bar{v}_j \dd x + \int_{U''\cap\jYstiff\cap \eps_j B_r} v_j-\bar{v}_j \dd x + \int_{(U'\setminus U'') \cap\jYstiff} v_j \dd x.\label{three_terms}
	\end{align}
	The first term in \eqref{three_terms} converges to $|\Ystiff||U''|d$ due to \eqref{bar_v_j_convergence} as $j\to\infty$.
	We handle the third term in \eqref{three_terms} via H\"older's inequality and \eqref{v_j_est},
	\begin{align}\label{third}
		\Big|\int_{(U'\setminus U'') \cap\jYstiff} v_j \dd x\Big| \leq \norm{v_j}_{L^p(U\cap\jYstiff;\R^2)}|U'\setminus U''|^{1-\frac{1}{p}} \leq M |U'\setminus U''|^{1-\frac{1}{p}}
	\end{align}
	The second term in \eqref{three_terms} can be similarly estimated:
	\begin{align}\label{second}
		\Big|\int_{U''\cap\jYstiff\cap \eps_j B_r} v_j-\bar{v}_j \dd x \Big| &\leq \int_{U''}|v_j|\one_{\jYstiff}\one_{\eps_jB_r} \dd x + \int_{U''} |\bar{v}_j| \one_{\eps_j B_r} \dd x\nonumber\\
		&\leq M|U''\cap \eps_j B_r|^{1-\frac{1}{p}} + \int_{U''} |\bar{v}_j| \one_{\eps_j B_r} \dd x.
	\end{align}
	We shall also point out that $|U''\cap \eps_j B_r|\to |U''|\pi r^2$ and $\int_{U''} |\bar{v}_j| \one_{\eps_j B_r} \dd x \to |d||U''|\pi r^2$ as $j\to \infty$.
	
	In conclusion, we first select for arbitrary $\delta>0$ the set $U''\Subset U'$ large enough that \eqref{third} is smaller than $\delta$.
	Then, we select $r$ small enough and subsequently $j$ sufficiently large in such a way that \eqref{second} is also bounded by $\delta$, and that
	\begin{align*}
		\Big| \int_{U''\cap\jYstiff} \bar{v}_j \dd x - |\Ystiff||U''|d\Big| \leq \delta
	\end{align*}		
	for all such $j$.
	These choices then produce
	\begin{align}\label{close_to_0}
		|\Ystiff||U''||d| \leq 3\delta.
	\end{align}
	
	\textit{Step 2b: Bound of $|d|$ from below.}
	The convergence \eqref{bar_v_j_convergence} yields that
	\begin{align*}
		\norm{\bar{v}_j}_{L^1(U'';\R^2)}\to |U''||d|\quad \text{ as } j\to \infty.
	\end{align*}
	On the other hand, this norm can be estimated from below by exploiting \eqref{bar_v_j2} and \eqref{normalization}:
	\begin{align}\label{two_terms}
		\norm{\bar{v}_j}_{L^1(U'';\R^2)} &\geq \norm{\bar{v}_j}_{L^1(U''\cap \jYstiff \setminus \eps_j B_r;\R^2)} = \norm{v_j}_{L^1(U''\cap \jYstiff \setminus \eps_j B_r;\R^2)}\nonumber\\
		&= \norm{v_j}_{L^1(U'\cap \jYstiff \setminus \eps_j B_r;\R^2)} - \norm{v_j}_{L^1((U'\setminus U'')\cap \jYstiff \setminus \eps_j B_r;\R^2)} \nonumber\\
		&= 1- \norm{v_j}_{L^1(U'\cap \jYstiff \cap \eps_j B_r;\R^2)} - \norm{v_j}_{L^1((U'\setminus U'')\cap \jYstiff \setminus \eps_j B_r;\R^2)} \nonumber\\
		&= 1- \norm{v_j}_{L^1(U''\cap \jYstiff \cap \eps_j B_r;\R^2)} - \norm{v_j}_{L^1((U'\setminus U'')\cap \jYstiff;\R^2)}.
	\end{align}
	The last two terms in \eqref{two_terms} are handled via \eqref{v_j_est} and H\"older's inequality, namely
	\begin{align*}
		\norm{v_j}_{L^1(U''\cap \jYstiff \cap \eps_j B_r;\R^2)} \leq M |U''\cap \eps_j B_r|^{1-\frac{1}{p}},
	\end{align*}
	and
	\begin{align*}
		\norm{v_j}_{L^1((U'\setminus U'')\cap \jYstiff;\R^2)} \leq M |U'\setminus U''|^{1-\frac{1}{p}}.
	\end{align*}
	Similarly to Step 2a, we then find that $|d-1|$ is very small which is a contradiction to \eqref{close_to_0}. This concludes the proof of this lemma.
\end{proof}

Finally, we state a brief technical Lemma which ensures that \eqref{domain_est} is satisfied for a suitable sequence later on.
\begin{lemma}\label{lem:complement}
	Let $p>2$, and $U\subset \R^2$ be a bounded Lipschitz domain and let $(v_\eps)_\eps\subset W^{1,p}(U;\R^2)$ be bounded and satisfy $\norm{v_\eps}_{L^p(U\cap\eYstiff;\R^2)}\geq a > 0$.
	Then, there exist $U'\Subset U$ and $M>0$ such that
	\begin{align*}
		\norm{v_\eps}_{L^p(U\cap\eYstiff;\R^2)} \leq M \norm{v_\eps}_{L^p(U'\cap\eYstiff;\R^2)}
	\end{align*}
	for every $\eps>0$.
\end{lemma}
\begin{proof}
	Since the family $(v_\eps)_\eps$ is bounded in $W^{1,p}(U;\R^2)$, we find via Sobolev-embeddings that $(v_\eps)_\eps$ is bounded in $C^0(\overline{U};\R^2)$.
	With $U_j:=\{x\in U: \dist(x,\partial U) > \frac{1}{j}\}$, we estimate 
	\begin{align*}
		\int_{(U\setminus U_j)\cap\eYstiff}|v_\eps|^p \dd x \leq \norm{v_\eps}^p_{C^0(\overline{U};\R^2)}|(U\setminus U_j)\cap\eYstiff| \leq \norm{v_\eps}^p_{C^0(\overline{U};\R^2)}|U\setminus U_j|;
	\end{align*}
	in particular, we find some $j_0\in\N$ such that we find on $U':=U_{j_0}$ that
	\begin{align*}
		\norm{v_\eps}_{L^p((U\setminus U')\cap\eYstiff;\R^2)}\leq \frac{a}{2}
	\end{align*}
	for all $\eps>0$.
	We argue for the rest of the proof via contradiction. 
	Assume that for the set $U'$ chosen above and every $j\in\N$ there exists $v_j:=v_{\eps_j}$ such that
	\begin{align*}
		\norm{v_j}_{L^p(U\cap\jYstiff;\R^2)}^p > j^p \norm{v_j}_{L^p(U'\cap\jYstiff;\R^2)}^p=j^p\big(\norm{v_j}_{L^p(U\cap\jYstiff;\R^2)}^p - \norm{v_j}^p_{L^p((U\setminus U')\cap\jYstiff;\R^2)}\big).
	\end{align*}
	We then find for every $j>1$ that
	\begin{align*}
		a^p \leq \norm{v_\eps}_{L^p(U\cap\eYstiff;\R^2)}^p \leq \frac{j^p}{j^p-1} \norm{v_\eps}_{L^p((U\setminus U')\cap\eYstiff;\R^2)}^p \leq \frac{j^p}{j^p-1}\frac{a^p}{2^p},
	\end{align*}
	which is a contradiction if $j$ is large enough.
\end{proof}

\subsection{Macroscopic deformation behavior}\label{sec:elastic_deform}
With technical tools and results about the local behavior in place, we are now in a position to derive global effects.
The next theorem serves as the compactness result in Theorem \ref{theo:homogenization_elastic_intro} and is the analogon of Proposition \ref{prop:K_lambda} a) in the rigid case.

\begin{proposition}[Criterion for limit deformations]\label{prop:elastic_compactness}
Let $p>2$, $\beta>2p-2$, and let $(u_\eps)_\eps\subset \Acal$, cf.~\eqref{Acal}, be a sequence that satisfies
\begin{align}\label{elastic_energy}
	\int_{\Omega\cap\eYstiff}\dist^p(\nabla u_\eps, \SO(2))\dd{x}\leq C\eps^\beta
\end{align}
for a constant $C>0$ independent of $\eps$.
If $u_\eps\weakly u$ in $W^{1,p}(\Omega;\R^2)$ for some $u\in W^{1,p}(\Omega;\R^2)$, then $u$ is affine with $\nabla u \in K$ with $K$ as in \eqref{Klambda}.
\end{proposition}
\begin{proof}
	Throughout this proof as well as the proof of Theorem \ref{theo:homogenization_elastic_intro}, we work on several nested compactly contained Lipschitz domains $\Omega_i\Subset \Omega_{i-1}$ with $i\in\{0,\ldots,4\}$ and $\Omega_0:=\Omega$.
	For such sets, we define
	\begin{align}\label{index_sets}
		J_{\eps,i} = \{k\in \Z^2 : \Omega_i\cap\eps(k+Y)\neq \emptyset\}
	\end{align}
	and observe that 
	\begin{align*}
		\Omega_i\Subset \bigcup_{k\in J_{\eps,i}}\eps(k+Y)\Subset \Omega_{i-1}
	\end{align*}		
	for sufficiently small $\eps>0$.
	For $k\in J_{\eps,1}$, let $\eps(k+Z)$ be the union of $\eps(k+Y)$ and its eight neighboring cells, see \eqref{cellZ}, and note that $\eps(k+Z) \subset \Omega$ for $\eps$ sufficiently small.
	
	\smallskip
	
	\textit{Step 1: Setup.} Let $\Omega_1\Subset \Omega$ be an arbitrary bounded Lipschitz domain.
	We briefly write 
	\begin{align}\label{cross2}
		\begin{split}
			E_1(\eps,k)&:=\eps(k+Y_1),\ E_2(\eps,k):=\eps(k-e_1+Y_3),\\
			E_3(\eps,k)&:=\eps(k+e_2+Y_1),\ E_4(\eps,k):=\eps(k+Y_3),\\
			E'(\eps,k)&:=\bigcup_{i=1}^4 E_i(\eps,k)
		\end{split}
	\end{align}
	for $k\in \Z^2$ and observe that
	\begin{align}\label{same_squares2}
		E_3(\epsilon,k) = E_1(\eps,k+e_2)\qand E_4(\eps,k) = E_2(\eps,k+e_1)\quad\text{ for all }k\in\Z^2.
	\end{align}
	Since $\beta > 2p-2 > 2$ we find that $\norm{\dist(\nabla u_\eps,\SO(2))}_{L^p(E'(\eps,k))}^p \leq C\eps^\beta \leq \delta_0\eps^{2}$ with $\delta_0$ as in Lemma \ref{lem:luftschloss}, and $k\in \Z^2$ with $\eps(k+Z)\subset \Omega$.
	For all these $k$, we may apply Lemma \ref{lem:luftschloss} and Remark \ref{rem:scaling} for $\mu=\frac{1-\lambda}{\lambda}$, $\rho=\lambda\eps$, and the sets $\rho E_0 = \eps(k+Y_2)$, $\rho E_i = E_i(\eps,k)$ for $i\in\{1,\ldots,4\}$, to obtain two rotations $R_\eps^k,S_\eps^k\in\SO(2)$ such that 
	\begin{align}\label{anw_luftschloss2}
		\begin{split}
			\norm{\nabla u_\eps - S_{\eps}^k}_{L^p(E_1(\eps,k)\cup E_3(\eps,k);\R^{2\times 2})} &+ \norm{\nabla u_\eps - R_{\eps}^k}_{L^p(E_2(\eps,k)\cup E_4(\eps,k);\R^{2\times 2})}\\
			&\qquad\qquad\qquad\leq C\eps^{\frac{1}{p}}\norm{\dist(\nabla u_\eps,\SO(2))}_{L^p(E'(\eps,k))}^{\frac{1}{2}}
		\end{split}
	\end{align}
	and
	\begin{align}\label{scalar_product2}
		S_\eps^k e_1 \cdot R_{\eps}^k e_1 \geq -C \eps^{-\frac{1}{p}}\norm{\dist(\nabla u_\eps,\SO(2))}_{L^p(E'(\eps,k))}^{\frac{1}{2}}
	\end{align} 
	for a constant $C>0$ independent of $\eps$ and $k$.
	On this basis, we define two auxiliary piecewise constant maps $S_\eps, R_\eps :\Omega\to \SO(2)$ as
	\begin{align}\label{Seps_Reps2}
		S_\eps:=\sum_{k\in J_{\eps,1}}S_\eps^k \one_{\eps(k+Y)},\qand R_\eps:=\sum_{k\in J_{\eps,1}}R_\eps^k \one_{\eps(k+Y)}.
	\end{align}
	
	\textit{Step 2: Strong convergence of rotations.} We consider only $(S_\eps)_\eps$, the other sequence can be dealt with analogously.
	Recalling \eqref{cross2}, \eqref{same_squares2}, and \eqref{anw_luftschloss2} we find for $x\in \eps(k+Y)$ and $\tilde{x} \in \eps(k+ e_2 + Y)$ that
	\begin{align*}
		|S_\eps(x) - S_\eps(\tilde{x})|^p &= C\eps^{-2}\norm{S_\eps^k - S_\eps^{k+e_2}}_{L^p(\eps(k+ e_2 + Y_1);\R^{2\times 2})}^p \\
		&\leq C\epsilon^{-2}\big(\norm{S_\eps^k - \nabla u_\eps}_{L^p(E_3(\eps,k);\R^{2\times 2})}^p + \norm{S_\eps^{k+ e_2} - \nabla u_\eps}_{L^p(E_1(\eps,k+e_2);\R^{2\times 2})}^p\big)\\
		&\leq C\epsilon^{-1}\norm{\dist(\nabla u_\eps,\SO(2))}_{L^p(E'(\eps,k)\cup E'(\eps, k+e_2))}^{\frac{p}{2}}.
	\end{align*}
	Analogously (and with an additional triangle inequality), we obtain for any $\tilde{x}\in \eps(k+Z)$ that
	\begin{align}\label{help1}
		|S_\eps(x) - S_\eps(\tilde{x})|^p \leq C\epsilon^{-1}\norm{\dist(\nabla u_\eps,\SO(2))}_{L^p\big(\bigcup_{e\in I} E'(\eps,k + e)\big)}^{\frac{p}{2}}
	\end{align}
	with a constant $C>0$ independent of $\eps$, see \eqref{cellZ} for the definition of $I$.
	We choose now an arbitrary $\xi\in \R^2$ with $|\xi|\leq\frac{1}{2}\dist(\Omega_1,\partial \Omega)$ and we set $m_{\eps}=\lceil \frac{|\xi|_{\infty}}{\eps} \rceil$ with $|\xi|_{\infty}:=\max\{|\xi_1|, |\xi_2|\}$. 
	Now, select $m_{\eps} +1$ points $0=\xi\ui{0}, \xi\ui{1}, \ldots, \xi\ui{m_{\eps}}= \xi$ in such a way that $|\xi\ui{j+1}-\xi\ui{j}|_{\infty} \leq \eps$ for every $j=0, \ldots, m_{\eps}-1$,
	which produces a piecewise straight path from the origin to $\xi$ with maximal step length $\eps$. With the help of a telescoping sum argument and the discrete H\"older's inequality, we obtain
	\begin{align*}
		|S_\eps(x) - S_\eps(x + \xi)|^p \leq m_\eps^{p-1}\sum_{i=0}^{m_\eps-1}|S_\eps(x + \xi\ui{j}) - S_\eps(x + \xi\ui{j+1})|^p.
	\end{align*}
	Integrating this estimate on $\eps(k+Y)$ and combining the result with \eqref{help1} generates
	\begin{align*}
		\int_{\eps(k+Y)}|S_\eps(x) - S_\eps(x + \xi)|^p \dd x \leq C\epsilon m_\eps^{p-1}\sum_{i=0}^{m_\eps -1}\norm{\dist(\nabla u_\eps,\SO(2))}_{L^p\big(\bigcup_{e\in I} E'(\eps,k+\lfloor\xi\ui{j}\rfloor + e)\big)}^{\frac{p}{2}}
	\end{align*}		
	with $\lfloor \zeta \rfloor := (\lfloor \zeta_1\rfloor,\lfloor \zeta_2 \rfloor)$ for $\zeta \in \R^2$.
	Summing over all $k\in J_{\eps,1}$ and considering that $m_\eps \leq C\frac{|\xi|}{\eps} + 1$ as well \eqref{anw_luftschloss2} we derive
	\begin{align*}
		\int_{\Omega_1}|S_\eps(x)& - S_\eps(x +\xi)|^p \dd x \leq C\epsilon m_\eps^{p-1}\sum_{i=0}^{m_\eps -1}\sum_{k\in J_{\eps,1}}\norm{\dist(\nabla u_\eps,\SO(2))}_{L^p\big(\bigcup_{e\in I} E'(\eps,k+\lfloor\xi\ui{j}\rfloor + e)\big)}^{\frac{p}{2}}\nonumber \\
		&\leq C\epsilon m_\eps^p \norm{\dist(\nabla u_\eps,\SO(2))}_{L^p(\Omega\cap\eYstiff)}^{\frac{p}{2}}
		\leq C\epsilon^{1+\frac{\beta}{2}}\Big(\frac{|\xi|}{\eps}+1\Big)^p
		\leq C\Big(|\xi|^p\eps^{1+\frac{\beta}{2} - p}+\eps^{1+\frac{\beta}{2}}\Big)
	\end{align*}
	for a constant $C>0$ independent of $\eps$. 
	In light of Fr\'ech\'et-Kolmogorov's theorem and the fact that $\beta > 2p-2$, we conclude that $(S_\eps)_\eps$ converges strongly in $L^p(\Omega_1;\R^{2\times 2})$ to a constant rotation $S\in \SO(2)$.
	We find analogously that $(R_\eps)_{\eps}$ converges strongly to some $R\in \SO(2)$ in $L^p(\Omega_1;\R^{2\times 2})$.
	Moreover, we derive from \eqref{scalar_product2} and \eqref{elastic_energy} that
	\begin{align*}
		S_\eps(x) e_1\cdot R_\eps(x)e_1 \geq - C \eps^{\frac{-2+\beta}{2p}} \to 0\quad \text{ as } \eps \to 0
	\end{align*}
	since $\beta>2$, which yields that $Se_1 \cdot Re_1 \geq 0$:
	
	\medskip
	
	\textit{Step 3: Approximation of $u_\eps$ by piecewise affine functions.} 
	We recall $w$ as in \eqref{nablaw}, set $\hat{w}_\eps: \R^2\to \R^2,\ x\mapsto \eps w(\tfrac{x}{\eps}),$ and define
	\begin{align}\label{w_eps_new}
		w_\eps:= \hat{w}_\eps + \int_{\Omega_1\cap\eYstiff}u_\eps-\hat{w}_\eps\dd x.
	\end{align}
	The goal is the prove that $w_\eps$ and $u_\eps$ are close to each other in the $L^p$-sense on a suitable large subset by comparing them on the soft and stiff components separately.
	We first distinguish between two cases: if $\norm{w_\eps - u_\eps}_{L^p(\Omega_1;\R^2)}\to 0$ as $\eps \to 0$, then there is nothing to prove.
	Otherwise we find via Lemma \ref{lem:complement}, applied to $v_\eps = u_\eps - w_\eps$ and $U=\Omega'$, a subset $\Omega_2\Subset\Omega_1$ and $M>0$ such that
	\begin{align}\label{est_Omega'}
		\norm{v_\eps}_{L^p(\Omega_1\cap\eYstiff;\R^2)} \leq M \norm{v_\eps}_{L^p(\Omega_2\cap\eYstiff;\R^2)}.
	\end{align}
	This subset can be chosen as close to $\Omega_1$ as we wish without changing the estimate \eqref{est_Omega'}.
	
	On each $\eps(k+Y_1)$ for $k\in J_{\eps,2}$ we estimate
	\begin{align}\label{u-w_eYrig}
		\norm{\nabla u_\eps-\nabla w_\eps}^p_{L^p(\eps(k+Y_1);\R^{2\times 2})} &= \norm{\nabla u_\eps- S}^p_{L^p(\eps(k+Y_1);\R^{2\times 2})}\nonumber\\
		&\leq C(\norm{\nabla u_\eps- S_\eps}^p_{L^p(\eps(k+Y_1);\R^{2\times 2})}  + \norm{S_\eps- S}^p_{L^p(\eps(k+Y_1);\R^{2\times 2})})\nonumber\\
		&\leq C(\eps\norm{\dist(\nabla u,\SO(2))}_{L^p(E'(\eps,k))}^{\frac{p}{2}} + \norm{S_\eps- S}^p_{L^p(\eps(k+Y_1);\R^{2\times 2})}),
	\end{align}
	and analogously on $\eps(k+Y_3)$.
	A summation of these estimates over all $k\in J_{\eps,2}$ in combination with \eqref{elastic_energy} then produces
	\begin{align}\label{u-w_Omega2}
		\norm{\nabla u_\eps-\nabla w_\eps}^p_{L^p(\Omega_2\cap\eYstiff;\R^{2\times 2})}\leq C(\eps^{1+\frac{\beta}{2}} + \norm{S_\eps- S}^p_{L^p(\Omega_1;\R^{2\times 2})} + \norm{R_\eps- R}^p_{L^p(\Omega_1;\R^{2\times 2})}),
	\end{align}
	where the latter two terms vanish in the limit due to Step 1.
	Now, let $\Omega_3\Subset\Omega_2$ be another arbitrary bounded Lipschitz domain. 
	In light of \eqref{est_Omega'}, we may apply the Poincar\'e-type Lemma \ref{lem:poincare2} to $v=u_\eps - w_\eps$, $U=\Omega_2$, and $U'=\Omega_3$, to obtain
	\begin{align}\label{u-w_Omega''}
		\norm{u_\eps-w_\eps}^p_{L^p(\Omega_3\cap\eYstiff;\R^{2\times 2})}&\leq C\norm{\nabla u_\eps-\nabla w_\eps}^p_{L^p(\Omega_2\cap\eYstiff;\R^{2\times 2})}\nonumber\\
		&\leq C(\eps^{1+\frac{\beta}{2}} + \norm{S_\eps- S}^p_{L^p(\Omega_1;\R^{2\times 2})}+ \norm{R_\eps- R}^p_{L^p(\Omega_1;\R^{2\times 2})}).
	\end{align}
	for a constant $C>0$ independent of $\eps$.
	
	To produce a similar estimate on the soft components, we first choose yet another bounded Lipschitz domain $\Omega_4\Subset\Omega_3$.
	We also observe that $Y_2 \subset Y_1 + \lambda e_2$ if $\lambda\geq\frac{1}{2}$, which allows us to estimate
	\begin{align}\label{est_Y22}
		\int_{\eps(k + Y_2)}|u_\eps - w_\eps|^p \dd x &\leq C\int_{\eps(k + Y_1)}|u_\eps - w_\eps|^p \dd x + C\int_{\eps(k+Y_1)} | (u_\eps - w_\eps)(x) - (u_\eps - w_\eps)(x + \eps e_2)|^p \dd x \nonumber\\
		&\leq C\big(\norm{u_\eps - w_\eps}_{L^p(\eps(k+Y_1);\R^2)}^p + \eps^p\norm{\nabla u_\eps - \nabla w_\eps}_{L^p(\eps(k+Y_1\cup Y_2);\R^{2\times 2})}^p\big)\nonumber\\
		&\leq C\big(\norm{u_\eps - w_\eps}_{L^p(\eps(k+Y_1);\R^2)}^p + \eps^p\norm{\nabla u_\eps}_{L^p(\eps(k+Y_1\cup Y_2);\R^{2\times 2})}^p +\eps^p |\eps(k+Y_1\cup Y_2)|\big);
	\end{align}
	in the case $\lambda<\frac{1}{2}$, it holds that $Y_2 \subset Y_3 - \lambda e_2$, which leads to a similar estimate.
	Analogously, we find that
	\begin{align}\label{est_Y42}
		\int_{\eps(k + Y_4)}|u_\eps  - w_\eps|^p \dd x \leq C\big(\norm{u_\eps - w_\eps}_{L^p(\eps(k+Y_1);\R^2)}^p + \eps^p\norm{\nabla u_\eps}_{L^p(\eps(k+Y_1\cup Y_4);\R^{2\times 2})}^p +\eps^p |\eps(k+Y_1\cup Y_4)|\big)
	\end{align}
	since $Y_4\subset Y_1 + \lambda e_1$ if $\lambda\geq \frac{1}{2}$.
	Summing \eqref{u-w_eYrig}, \eqref{est_Y22}, \eqref{est_Y42} over all $k\in J_{\eps,4}$, and combining the result with \eqref{u-w_Omega''} then yields
	\begin{align*}
		\int_{\Omega_4} |u_\eps -w_\eps|^p \dd x &\leq C\big(\eps^{1+\frac{\beta}{2}} + \norm{S_\eps- S}^p_{L^p(\Omega_1;\R^{2\times 2})} + \norm{R_\eps- R}^p_{L^p(\Omega_1;\R^{2\times 2})} + \eps^p \norm{u_\eps}_{W^{1,p}(\Omega;\R^2)}^p + \eps^p |\Omega|\big)\\
		&\leq C\big(\eps^{1+\frac{\beta}{2}} + \eps^p +\norm{S_\eps- S}^p_{L^p(\Omega_1;\R^{2\times 2})} + \norm{R_\eps- R}^p_{L^p(\Omega_1;\R^{2\times 2})}\big),
	\end{align*}
	with a constant $C>0$ independent of $\eps>0$. This shows that $u_\eps$ and $w_\eps$ have the same limit in $L^p(\Omega_4;\R^2)$. 
	Since $w_\eps$ converges to an affine function $w$ with gradient $\nabla w = F\in K$, cf.~\eqref{Klambda}, the limit function $u$ satisfies $\nabla u = F\in K$ on $\Omega_4$.
	
	An exhaustion argument proves the desired result.
\end{proof}

\subsection{Proof of Theorem \ref{theo:homogenization_elastic_intro}}\label{sec:elastic_hom}

Finally, we give the proof of Theorem \ref{theo:homogenization_elastic_intro}, which consists of verifying the lower and upper bounds for the $\Gamma$-convergence result.

\begin{proof}[Proof of Theorem \ref{theo:homogenization_elastic_intro}]
	\textit{Step 1: The lower bound.} Assume that $(u_\eps)_\eps\subset L^p_0(\Omega;\R^2)$ converges strongly to $u\in L^p_0(\Omega;\R^2)$ and satisfies
	\begin{align*}
		\lim_{\eps\to 0}\Ical_\eps(u_\eps) = \liminf_{\eps\to 0}\Ical_\eps(u_\eps)<\infty;
	\end{align*} 
	in this case, it holds that $(u_\eps)_\eps\subset \Acal$.
	Due to the lower bounds \eqref{Wsoft} and \eqref{Wrig} for $\Wsoft$ and $\Wstiff$, we find that $(\nabla u_\eps)_{\eps}$ is bounded in $L^p(\Omega;\R^2)$.
	A direct application of Poincar\'e's inequality on $\Omega$ yields the boundedness of $(u_\eps)_{\eps}$ in $W^{1,p}(\Omega;\R^2)$ as well as $u_\eps\weakly u$ in $W^{1,p}(\Omega;\R^2)$ up to the selection of a (non-relabeled) subsequence.
	Moreover, Proposition \ref{prop:elastic_compactness} produces that $u$ is affine with gradient $\nabla u = F\in K$, cf.~\eqref{Klambda}, in view of the lower bound \eqref{Wrig}.	
	
	From the proof of Proposition \ref{prop:elastic_compactness}, we recall the five sets $\Omega_4\Subset\ldots\Subset\Omega_1\Subset\Omega_0=\Omega$ together with their index sets $J_{\eps,1},\ldots,J_{\eps,4}$ as in \eqref{index_sets}, the quantities $S_\eps,R_\eps$ as in \eqref{Seps_Reps2} with limits $S,R\in\SO(2)$, as well as $w_\eps$ as in \eqref{w_eps_new}.
	Since $W^\qc$ is polyconvex, there exists a lower semicontinuous and convex function $g:\R^{2\times 2}\times \R\to [0,\infty]$ such that $W^\qc(F) = g(F,\det F)$.	
	In light of Jensen's inequality (see \cite[Lemma A.2]{MVGSN20}) for extended-valued functions, we estimate
	\begin{align}\label{lower_bound1}
		\Ical_\eps(u_\eps) &\geq \int_{\Omega_4\cap\eYsoft}\Wsoft(\nabla u_\eps) \dd x 
		\geq \int_{\Omega_4\cap\eps Y_2}\Wsoft^\qc(\nabla u_\eps) \dd x + \int_{\Omega_4\cap \eps Y_4}\Wsoft^\qc(\nabla u_\eps) \dd x\nonumber\\
		&\geq \sum_{i\in\{2,4\}}|\Omega_4\cap\eps Y_i|\,g\Big(\dashint_{\Omega_4\cap\eps Y_i}(\nabla u_\eps,\det \nabla u_\eps) \dd x \Big).
	\end{align}
	We now want to exchange every $u_\eps$ by the easier piecewise affine function $w_\eps$, which is suitably close to $u_\eps$.
	Since $|\Omega_4\cap\eps Y_i| \to |\Omega_4||Y_i|>0$ as $\eps \to 0$, it remains to show that
	\begin{align}\label{difference_minors}
		\Big|\int_{\Omega_4\cap \eps Y_i} (\nabla u_\eps,\det \nabla u_\eps) - (\nabla w_\eps,\det \nabla w_\eps) \dd x\Big| \to 0
	\end{align}
	as $\eps \to 0$ for $i\in\{2,4\}$.
	
	Lemma \ref{lem:extension2} applied to $U=\Omega_2$ and $U'=\Omega_4$ now generates a linear and continuous operator $L: W^{1,p}(\Omega_2\cap \eYstiff;\R^2)\cap C^0(\overline{\Omega_2\cap\eYstiff};\R^2)\to W^{1,p}(\Omega_4;\R^2)$.
	We then find that we may replace $u_\eps$ and $w_\eps$ in \eqref{difference_minors} by the continuous functions $\tilde{w}_\eps:= L(w_\eps\restrict{\Omega_2\cap \eYstiff})$ and $\tilde{u}_\eps:= L(u_\eps\restrict{\Omega_2\cap \eYstiff})$ since the minors are Null-Lagrangians. 
	For the difference in the gradients, we compute
	\begin{align}\label{diff_gradients}
		\Big|\int_{\Omega_4\cap\eps Y_i} \nabla \tilde{u}_\eps - \nabla \tilde{w}_\eps \dd x\Big| \leq C \norm{\nabla \tilde{u}_\eps - \nabla \tilde{w}_\eps}_{L^{p}(\Omega_4;\R^{2\times 2})} \leq C \norm{u_\eps - w_\eps}_{W^{1,p}(\Omega_2\cap\eYstiff;\R^{2})}
	\end{align}
	with a constant $C>0$ independent of $\eps$ for $i\in\{2,4\}$. 
	As for the determinants, we use that $(w_\eps)_\eps$ and $(u_\eps)_\eps$ are both bounded in $W^{1,p}(\Omega_2;\R^2)$ to estimate
	\begin{align}\label{diff_dets}
		\Big|\int_{\Omega_4\cap \eps Y_i} \det \nabla \tilde{u}_\eps - \det \nabla \tilde{w}_\eps \dd x\Big| &\leq 
		C\Big(\norm{\partial_1 \tilde{w}_\eps}_{L^{p'}(\Omega_4;\R^2)}\norm{\partial_2 \tilde{w}_\eps - \partial_2 \tilde{u}_\eps}_{L^p(\Omega_4;\R^2)} \nonumber\\
		&\qquad\qquad+ \norm{\partial_2 \tilde{u}_\eps}_{L^{p'}(\Omega_4;\R^2)}\norm{\partial_1 \tilde{w}_\eps - \partial_1 \tilde{u}_\eps}_{L^p(\Omega_4;\R^2)}\Big)\nonumber\\
		&\leq C\norm{w_\eps - u_\eps}_{W^{1,p}(\Omega_2\cap\eYstiff;\R^2)}
	\end{align}
	with $\frac{1}{p'} + \frac{1}{p} = 1$ and a constant $C>0$ independent of $\eps$, for $i\in\{2,4\}$.
	By combining \eqref{diff_gradients} and \eqref{diff_dets} with \eqref{u-w_Omega2}, we verify \eqref{difference_minors}.
	
	Considering the definition \eqref{w_eps_new} of $w_\eps$, we can now pass to the limit in \eqref{lower_bound1},
	\begin{align*}
		\liminf_{\eps\to 0}\Ical_\eps(u_\eps) &\geq \sum_{i\in\{2,4\}}|\Omega_4||Y_i|\,g\Big(\liminf_{\eps\to 0}\dashint_{\Omega_4\cap\eps Y_i}(\nabla w_\eps,\det \nabla w_\eps) \dd x \Big)\\
		&\geq |\Omega_4||Y_2| g\big((Se_1|Re_2), Se_1\cdot Re_1\big) + |\Omega_4||Y_4| g\big((Re_1|Se_2), Se_1\cdot Re_1\big)\\
		&= |\Omega_4|\frac{|\Ysoft|}{2}\big(\Wsoft^\qc(Se_1|Re_2) + \Wsoft^\qc(Re_1|Se_2)\big)\geq |\Omega_4| \Whom(F)
	\end{align*}
	with $\nabla u = F\in K$. 
	By taking the supremum over all compactly contained $\Omega_4\Subset\Omega$, we produce the desired estimate.

	\textit{Step 2: The upper bound.} The recovery sequence can be constructed exactly as in the proof of Theorem \ref{theo:homogenization_rigid}.	
\end{proof}

\subsection*{Acknowledgements}
This work was initiated when CK and DE were affiliated with Utrecht University. CK acknowledges partial support by the Dutch Research Council NWO through the project TOP2.17.01 and the Westerdijk Fellowship program.

\bibliographystyle{abbrv}
\bibliography{EDK23}

	\end{document}